\documentclass[11pt,a4paper,reqno]{amsart}
\usepackage{amsfonts}
\usepackage{amsmath,amssymb,amsthm,amsxtra}
\usepackage{float}
\usepackage[colorlinks, linkcolor=blue,anchorcolor=Periwinkle,
citecolor=red,urlcolor=Emerald]{hyperref}
\usepackage[usenames,dvipsnames]{xcolor}
\usepackage{enumitem}
\usepackage{geometry,array} 
\usepackage{graphicx}
\usepackage[normalem]{ulem}
\usepackage{subfigure}
\usepackage{bookmark}
\usepackage{tikz}
\usepackage{url}

\usetikzlibrary{matrix,positioning,decorations.markings,arrows,decorations.pathmorphing,
	backgrounds,fit,positioning,shapes.symbols,chains,shadings,fadings,calc}

\tikzset{->-/.style={decoration={  markings,  mark=at position #1 with
			{\arrow{>}}},postaction={decorate}}}
\tikzset{-<-/.style={decoration={  markings,  mark=at position #1 with
			{\arrow{<}}},postaction={decorate}}}

\usepackage[all]{xy}
\def\nn{node{$\bullet$}}
\def\ww{node{$\circ$}}

\newcommand\Old{\bgroup\markoverwith{\textcolor{red}{\rule[0.5ex]{2pt}{0.4pt}}}\ULon}

\theoremstyle{plain}
\newtheorem{theorem}{Theorem}[section]
\newtheorem*{thm}{Theorem~A}

\newtheorem{lemma}[theorem]{Lemma}
\newtheorem{corollary}[theorem]{Corollary}
\newtheorem{proposition}[theorem]{Proposition}

\theoremstyle{definition}
\newtheorem{definition}[theorem]{Definition}

\newtheorem{remark}[theorem]{Remark}
\newtheorem{notations}[theorem]{Notations}

\numberwithin{equation}{section}
\newtheorem{assumption}[theorem]{Assumption}

\def\hh{\mathcal}

\def\kong{\mathbb}
\def\<{\langle}
\def\>{\rangle}
\def\ZZ{\mathbb{Z}}

\def\R{\mathbb{R}}

\def\CC{\operatorname{CC}}
\def\OC{\operatorname{OC}}
\def\OA{\operatorname{OA}}

\def\Aut{\operatorname{Aut}}

\def\Hom{\operatorname{Hom}}

\def\Ext{\operatorname{Ext}}

\def\diff{\operatorname{d}}

\def\pas{\operatorname{AS}^+}

\def\deg{\operatorname{deg}}

\renewcommand{\k}{\mathbf{k}}

\def\wc{\widetilde{c}}
\def\wg{\widetilde{\gamma}}
\def\we{\widetilde{\eta}}
\def\ws{\widetilde{\sigma}}

\def\wwe{\overset{\approx}{\eta}}
\def\wws{\overset{\approx}{\sigma}}
\def\wwt{\overset{\approx}{\tau}}

\def\lws{\overset{\approx}{\sigma}}
\def\lwt{\overset{\approx}{\tau}}
\def\lwe{\overset{\approx}{\eta}}
\def\lwa{\overset{\approx}{\alpha}}
\def\lwb{\overset{\approx}{\beta}}
\def\lwr{\overset{\approx}{\rho}}
\def\lwg{\overset{\approx}{\gamma}}

\def\wt{\widetilde{\tau}}
\def\wa{\widetilde{\alpha}}

\def\ue{\eta}
\def\us{\sigma}
\def\ut{\tau}
\def\ua{\alpha}
\def\ub{\beta}
\def\wr{\widetilde{\rho}}
\def\wss{\widetilde{s}}
\newcommand{\ind}{\operatorname{ind}}
\newcommand{\sind}{\operatorname{ind}^\XX}
\def\ii{\varrho} % intersection index
\def\sii{\varsigma} % intersection index

\newcommand{\id}{\operatorname{id}}

\newcommand{\D}{\operatorname{\hh{D}}}
\def\C{\operatorname{\hh{C}}}
\renewcommand{\H}{\operatorname{\hh{H}}}

\newcommand{\per}{\operatorname{per}}
\newcommand\Sph{\operatorname{Sph}}%^\circ}
\newcommand\Sphzz{\operatorname{Sph}^{\ZZ^2}}%^\circ}

\newcommand{\Tri}{\bigtriangleup}
\def\arrow{red}
\def\surf{\mathbf{S}}                       %FST's surface
\def\SS{\mathcal{S}}                       %HKK's
\def\MTS{\mathbb{R}T\surf^{\lambda}}
\def\MTSo{\mathbb{R}T\surfo}
\def\PTS{\mathbb{P}T\surf}
\def\PTSO{\mathbb{P}T(\surf\setminus\Tri)}
\newcommand{\PTSo}[1]{\mathbb{P}{T_{#1}}\surfo}
\def\cc{\mathbf{c}} 					% cut
\newcommand{\ST}{\operatorname{ST}}        %spherical twists
\newcommand{\BT}{\operatorname{BT}}        %braid twists
\newcommand{\MCG}{\operatorname{MCG}}

\newcommand{\Int}{\operatorname{Int}}

\newcommand{\qqInt}{\operatorname{Int}^{\qv}}

\newcommand{\qv}{\mathbf{q}}
\newcommand{\qdH}{\dim^{\qv}\Hom^{\ZZ^2}}
\newcommand\coho[1]{\operatorname{H}^{#1}}
\newcommand\ho[1]{\operatorname{H}_{#1}}

\def\M{\mathbf{M}}
\def\Y{\mathbf{Y}}

\def\jiantou{edge[->,>=stealth]}

\def\RHom{\operatorname{\mathbf{R}Hom}}

\def\EE{\operatorname{\hh{E}}}

\def\cov{\operatorname{\pi_\Tri}}
\def\<{\langle}
\def\>{\rangle}
\def\surfo{{\mathbf{S}}_\Tri}
\def\surfoi{\surfo^\circ}

\def\CA{\operatorname{CA}}
\def\ZL{\wACC^0}
\def\wCA{\widetilde{\CA}}
\def\wCC{\widetilde{\CC}}
\def\wOC{\widetilde{\OC}}
\def\wOA{\widetilde{\OA}}

\def\ACC{\operatorname{AC}}
\def\wACC{\widetilde{\ACC}}

\def\wXCA{\widetilde{\CA}^\XX}

\def\wXACC{\widetilde{\ACC}^\XX}

\def\Z{\mathbf{Z}}
\def\add{\operatorname{add}}

\def\XX{\kong{X}}
\newtheorem{construction}[theorem]{Construction}

\def\ac{ {\widetilde{\mathbf{A}}} }
\def\dac{\ac^*_{\Tri}}
\def\acps{{\ac}^{'*}_{\Tri}}
\def\udac{\mathbf{A}^*_{\Tri}}
\def\ic{\hh{L}}
\def\ex{\mathbf{im}_{\cut}}
\def\cut{\operatorname{\mathbf{c}}}
\def\DMS{\surfo^\Lambda}

\usepackage{extarrows}
\title[Graded DMS: CY-$\XX$ categories of gentle algebras]
{Graded decorated marked surfaces: Calabi-Yau-$\XX$ categories of graded gentle algebras}

\author{Akishi Ikeda}
\address{Department of Mathematics,
    Josai University,
    Saitama, Japan}
\email{akishi@josai.ac.jp}

\author{Yu Qiu}
\address{Yau Mathematical Science Center,
	Tsinghua University,
	Beijing, China}
\email{yu.qiu@bath.edu}

\author{Yu Zhou}
\address{Yau Mathematical Science Center,
	Tsinghua University,
	Beijing, China}
\email{yuzhoumath@gmail.com}

\begin{document}
%=========================================================
%=========================================================

\thanks{The work was supported by National Natural Science Foundation of China (Grant No. 11801297), Tsinghua University Initiative Scientific Research Program (2019Z07L01006), Beijing Natural Science Foundation (Z180003).}

%=========================================================
\begin{abstract}
Let $\surf$ be a graded marked surface.
We construct a string model for the Calabi-Yau-$\XX$ category $\D_\XX(\surf)$,
via the graded DMS (=decorated marked surface) $\surfo$.
We prove an isomorphism between the braid twist group of $\surfo$ and the spherical twist group of $\D_\XX(\surf)$,
and $\qv$-intersection formulas.
We also give a topological realization of the Lagrangian immersion $\D_\infty(\surf)\to\D_\XX(\surf)$,
where $\D_\infty(\surf)$ is the topological Fukaya category associated to $\surf$,
that is triangle equivalent to the bounded derived category of some graded gentle algebra.
This generalizes previous work \cite{QQ,QZ2} in the Calabi-Yau-3 case and
also unifies the `Calabi-Yau-$\infty$' case $\D_\infty(\surf)$ (cf. \cite{HKK,OPS}).
\end{abstract}

	
\maketitle
\setcounter{tocdepth}{1}
\tableofcontents\addtocontents{toc}{\setcounter{tocdepth}{1}}

%=========================================================
	
\setlength\parindent{0pt}
\setlength{\parskip}{5pt}
	
\def\Z{\mathbf{Z}}
	
%=========================================================
\section*{Introduction}
%=========================================================
This paper generalizes the results of topological realization of Calabi-Yau-3 gentle type differential graded algebras in \cite{QQ,QZ2} to the Calabi-Yau-$\XX$ version, which also unifies the construction for the Calabi-Yau-$\infty$ case (i.e. the bounded derived categories of graded gentle algebras, cf. \cite{HKK,LP,OPS}) via Lagrangian immersions.
%=========================================================
\subsection*{Stability conditions on Fukaya type categories}
%=========================================================
In the seminal works \cite{BS,HKK},
Bridgeland-Smith (BS) and Haiden-Katzarkov-Kontsevich (HKK) show, respectively,
that the stability conditions on two Fukaya type categories $\D$ from a surface $\surf$
can be identified with meromorphic quadratic differentials with corresponding predescribed singularities.
Here, a stability condition consists of a heart
(an abelian category, that is equivalent to the heart of a bounded t-structure) and
a central charge (a group homomorphism) $Z\colon K(\D)\to\mathbb{C}$,
where $K(\D)$ is the Grothendieck group of $\D$.
The key ingredients on their approach are the following:
\begin{enumerate}
  \item A (simple closed) arc $\eta$ on $\surf$ corresponds to an (indecomposable) object $X_\eta$ in $\D$.
  \item A quadratic differential $\phi$ determines a family of special arcs $\eta_j$ on $\surf$
  (the saddle trajectories) and the corresponding objects $X_{\eta_j}$ determine a heart in $\D$.
  \item The length of $\eta_j$ with respect to $\phi$ gives the central charge of $X_{\eta_j}$, i.e.
  via the formula $$Z(X_{\eta_j})=\displaystyle\int_{\eta_j} \sqrt{\phi}.$$
\end{enumerate}
In fact, \cite{BS} dexterously bypassed $1^\circ$ to get $2^\circ$ using cluster theory
and \cite{QQ} filled the gap with further applications on understanding
the symmetries (i.e. fundamental groups) of spaces of stability conditions in \cite{KQ1}.

One interesting question is about the relation between \cite{BS} and \cite{HKK}.
In \cite{IQ1,IQ2}, we introduce the $q$-deformations of stability conditions and quadratic differentials
to give an answer.
On the categorical level, HKK's topological Fukaya category $\D_\infty(\surf)$
can be embedded into a Calabi-Yau-$\XX$ category $\D_\XX(\surf)$ with a distinguish automorphism $[\XX]$
as a grading shift functor.
When specifying $\XX$ to be 3, i.e. taking the orbit quotient of $\D_\XX(\surf)$
by the automorphism $[\XX-3]$, one recovers Calabi-Yau-3 category $\D_3(\surf)$ in the setting of BS
(which is the subcategory of a certain derived Fukaya category \cite{S}).

In this sequel, we focus on ingredient $1^\circ$ of the story, i.e. the topological realization of
(objects and morphisms in) categories, that we construct a string model for $\D_\XX(\surf)$.
%=========================================================
\subsection*{Previous works on string models}
%=========================================================
The previous related works on various categories of (graded) gentle algebras are summarized as follows:
\begin{itemize}
  \item The derived category of the Calabi-Yau-$\XX$ type $A_n$ algebra was constructed and investigated in \cite{KS} via a decorated disk, where it was shown that there is a faithful action of the (classical) braid group on the category and the $\qv$-dimension of the morphism space between two certain objects is equal to the $\qv$-intersection number between the corresponding arcs.
  \item The derived categories of Calabi-Yau-3 gentle type (dg) algebras were studied in \cite{QQ,QZ2} via decorated marked surfaces (without punctures), where it was shown that closed arcs correspond to reachable spherical objects (up to shift), the braid twist group is isomorphic to the spherical twist group and the intersection number between two closed arcs equals half of the dimension of the total homomorphism space between the corresponding spherical objects.
  \item The derived categories of graded gentle algebras appeared in \cite{HKK} as topological Fukaya categories of surfaces, where it was showed that the indecomposable objects correspond to admissible curves with local systems. \item Later, it was shown in \cite{OPS} (resp. \cite{LP}) that the derived categories of (ungraded) finite dimensional gentle algebras (resp. homological smooth graded gentle algebras) are obtained in this way. The formula connecting dimensions of homomorphism spaces between objects and graded intersection numbers was also given.
  \item The cluster categories of Jacobian gentle algebras were studied in \cite{BZ,ZZZ} via triangulated marked surfaces without punctures, where the correspondence between curves and valued closed curves and indecomposable objects, the interpretation of the Auslander-Reiten translation via rotation of marked points along the boundary, and the equality between the intersection of two curves and the dimension of $\Ext^1$ of the corresponding objects were given.
  \item Gentle algebras were realized as tiling algebras associated to partial triangulations of marked surfaces in \cite{BaS}, where it was shown that there is a bijection between indecomposable modules and permissible curves. An interpretation of the Auslander-Reiten translation via the pivot elementary move and a method on how to read morphisms from relative positions of curves were also given.
\end{itemize}

%=========================================================
\subsection*{String model with double grading}
%=========================================================
The input data in our story is a graded marked surface $\surf=(\surf,\lambda)$,
where $\surf=(\surf,\M,\Y)$ is a marked surface and $\lambda$ is a line field/grading. This is the model for the derived categories of graded gentle algebras.
We construct the graded DMS $\surfo$ from $\surf$ by pulling the marked points in $\Y$
into the interior of $\surf$,
which naturally introduces another grading, denoted by $\XX$ and known as the Adams grading in topology.
To realize this grading, one can use the log surface model (cf. \S~\ref{sec:log}).

Then we construct the double graded version of string model on
the Calabi-Yau-$\XX$ category $\D_\XX(\surf)$ associated to $\surf$ and prove the following main result
(cf. Theorem~\ref{thm:final}).

\begin{thm}
Let $\surfo$ be a graded DMS. There is a full formal arc system $\ac$ with an associated differential graded $\XX$-graded algebra $\Gamma_\ac$ and the Calabi-Yau-$\XX$ category $$\D_\XX(\surf)\colon=\D_{fd}(\Gamma_\ac),$$ such that the following hold.
\begin{itemize}
  \item There is a bijection
  $$X\colon \lwe\mapsto X_{\lwe}$$
%  from the set of double graded admissible curves on $\surfo$ to the set of objects in $\D_\XX(\surf)$, which restricts to a bijection
  from the set of double graded closed arcs on $\surfo$ to the set of reachable spherical objects in $\D_\XX(\surf)$.
  % (Theorem~\ref{thm:X:}).
  \item $X$ induces an isomorphism
  \begin{equation}
    \iota:\BT(\surfo)\cong \ST_\ast(\Gamma_\ac)
  \end{equation}
  between the braid twist group of $\surfo$ and the spherical twist group of $\Gamma_\ac$,
  sending a braid twist $b_{\eta}$ to the spherical twist $\phi_{X_{\eta}}$. %(Theorem~\ref{thm:st=bt}).
  \item For any pair of double graded closed arcs $\lws$ and $\lwt$ on $\surfo$, we have
  \begin{equation}
    \qqInt(\lws,\lwt)=\qdH (X_{\lws},X_{\lwt}),
  \end{equation}
  %(Corollary~\ref{cor:int=dim}),
  where $\qqInt$ is the $\ZZ^2$-graded $\qv$-intersection defined in \eqref{eq:q-int} and $\qdH$ is the $\qv$-dimension of double graded $\Hom$ defined in \eqref{eq:qHom}.
\end{itemize}
\end{thm}

The main difficulty on generalizing the previous work \cite{KS,QQ,QQ2,QZ2} lies on the lack of corresponding cluster theory (which is specific to the Calabi-Yau-3 case).

One of the applications is that we obtain the topological realization of Lagrangian immersion (Theorem~\ref{thm:L}), that our model contains the string model for the Calabi-Yau-$\infty$ category $\D_{\infty}(\surf)$ associated to a graded marked surface $\surf$, which is triangle equivalent to the bounded derived category of the graded gentle algebra $\Gamma^0_\ac$ associated to $\ac$.

%=========================================================
\subsection*{Contents}
%=========================================================
The paper is organized as follows.
In \S~\ref{sec:XX}, we review the basics of differential graded $\XX$-graded algebras and their derived categories.
In \S~\ref{sec:gDMS}, we introduce the graded decorated marked surfaces (DMS) as our topological model.
In \S~\ref{sec:string}, we describe the string model on graded DMS and in \S~\ref{sec:int=dim}
we prove the main result with technical assumptions.
Then we remove such assumptions in \S~\ref{sec:general}
and finish the paper with the description of a topological realization of the Lagrangian immersion in \S~\ref{sec:Lag}.

In Table~\ref{table}, we list our convention of notations of topological objects.
%=========================================================
\subsection*{Acknowledgments}\hfill \break
%=========================================================
Qy would like to thank Alastair King for explaining and pushing him to understand
the homological point of view on the grading of marked surfaces.

\def\grad{\lambda}

\begin{table}[t]
\caption{List of notations}\label{table}
\setlength{\extrarowheight}{2pt}
\begin{tabular}{ccc}
\hline
$\surf  $&& a topological surface with non-empty boundaries \\ \hline
$\M (\Y)  $&& a set of open (closed) marked points on $\partial\surf$ \\ \hline
$\grad  $&& a grading on $\surf$, or a class in $\coho{1}(\PTS)$ \\ \hline
$\surf^\grad  $&& a graded marked surface $(\surf,\grad)$ \\ \hline
$\Tri  $&& a set of decorations on $\surf$ \\ \hline
$\surfo  $&& a decorated marked surface (DMS) \\ \hline
$\cut  $&& a cut, i.e. a pairing between the set $\Y$ and the set $\Tri$ \\ \hline
$\Lambda  $&& a class in $\coho{1}(\PTSo{},\ZZ^2)$  \\ \hline
$\surfo^\Lambda$&& a graded DMS associated to $\surf^\lambda$ w.r.t. some cut $\cut$  \\ \hline
$\log\surfo  $&& the $\log$ surface associated to $\surfo^\Lambda$ \\ \hline
$\overline{\us}  $&&the inverse of a curve $\us$, i.e. $\overline{\us}(t)=\overline{\us}(1-t)$ \\ \hline
$\ws  $&& a $\ZZ$-graded curve as a lift of $\dot{\us}$ \\ \hline  \\[-14pt]
$\wws  $&& a $\ZZ\oplus\ZZ\XX$-graded curve as a lift of $\dot{\us}$ \\ \hline  \\[-12pt]
$\Int_?^?(a,b)  $&& (various) intersection number between $a$ and $b$ \\ \hline  \\[-12pt]
$\ind_p^?(a,b)  $&& the index of an intersection $p$ between $a$ and $b$   \\[2pt] \hline
$\ACC(?)  $&& the set of admissible closed curves in $?$ \\ \hline
$\CA(?)  $&& the set of closed arcs in $?$\\ \hline \\[-12pt]
$\ac  $&& an open full formal arc system of $\surfo^\Lambda$ (or $\surf^\grad$) \\  \hline \\[-12pt]
$\dac (\ac^*)$ && a closed full formal arc system dual to $\ac$ in $\surfo^\Lambda$ (or $\surf^\grad$)\\ \hline
\end{tabular}
\end{table}

%=========================================================
\section{Derived categories of differential graded $\XX$-graded algebras}\label{sec:XX}
%=========================================================
Throughout this paper, $\k$ denotes an algebraically closed field. We review the categorical preliminaries on differential graded $\XX$-graded stuffs.

%=========================================================	
\subsection{Differential graded $\XX$-graded algebras}
%=========================================================	
A \emph{$\ZZ\XX$-graded (or $\XX$-graded for short)  $\k$-module} is a $\k$-module $V$ which decomposes into a direct sum
\[V=\bigoplus_{i\in\ZZ} V_{i},\]
where each $V_i$ is a $\k$-module. The shift $V[\sii\XX]$, for $\sii\in\ZZ$, is defined to be the $\XX$-graded $\k$-module whose components $V[\sii\XX]_i=V_{i+\sii}$, $i\in\ZZ$. A \emph{morphism} $f:V\to V'$ between two $\XX$-graded $\k$-modules is a $\k$-linear map such that $f(V_i)\subset V'_i$ for any $i\in\ZZ$. A \emph{complex of $\XX$-graded $\k$-modules} is a collection of morphisms $\{f^j:V^j\to V^{j+1}\}_{j\in\mathbb{Z}}$ of $\XX$-graded $\k$-modules with $f^{j+1}\circ f^j=0$ for any $j\in\ZZ$.
	
A \emph{$\ZZ\oplus\ZZ\XX$-graded (or $\ZZ^2$-graded for short) $\k$-module} is a $\k$-module $V$ which decomposes into a direct sum
\[V=\bigoplus_{j\in\ZZ} V^j,\]
where each $$V^j=\bigoplus\limits_{i\in\ZZ} V_i^j$$ is an $\XX$-graded $\k$-module. Each element $x$ in $V_i^j$ is called to have \emph{bi-degree} $j+i\XX$. We also denote by $|x|_1=j$ and $|x|_2=i$.
The \emph{shift} $V[\ii+\sii\XX]$, for $\ii,\sii\in\mathbb{Z}$, is defined to be the $\ZZ^2$-graded $\k$-module whose components $V[\ii+\sii\XX]^j_i=V^{j+\ii}_{i+\sii}$, $i,j\in\ZZ$.
A \emph{morphism} $f:V\to V'$ of $\ZZ^2$-graded $\k$-modules of \emph{bi-degree} $\ii+\sii\XX$ is a $\k$-linear map such that $f(V^j_i)\subset (V')^{j+\ii}_{i+\sii}$ for all $i,j\in\ZZ$.
Thus, for any $\ZZ^2$-graded $\k$-modules $V$ and $V'$, we have a $\ZZ^2$-graded $\k$-module $$\Hom_{\ZZ^2\text{-gr}}(V,V')=\bigoplus_{\ii,\sii\in\ZZ}\Hom_{\ZZ^2\text{-gr}}(V,V')_\ii^\sii,$$
where $\Hom_{\ZZ^2\text{-gr}}(V,V')_\ii^\sii$ consists of the morphisms from $V$ to $V'$ of bi-degree $\ii+\sii\XX$.
	
A \emph{differential graded (=dg) $\XX$-graded $\k$-module} is a $\ZZ^2$-graded $\k$-module $V$ endowed with a \emph{differential} $d_V$, i.e. a morphism $d_V:V\to V$ of bi-degree $1$ such that $d_V\circ d_V=0$.
Equivalently, a dg $\XX$-graded $\k$-module is a complex of $\XX$-graded $\k$-modules.
The \emph{shift} $V[\ii+\sii\XX]$ of a dg $\XX$-graded $\k$-module $V$, for $\ii,\sii\in\ZZ$, is the shift endowed with the differential $(-1)^{\ii}d_V$.
For a dg $\XX$-graded $\k$-module $V$, denote by $H^*(V)$ the \emph{homology} of $V$ with respect to the differential $d_V$.
Note that $H^*(V)$ is a $\ZZ^2$-graded $\k$-module.
	
Let $\C^\XX_{dg}(\k)$ be the category whose objects are dg $\XX$-graded $\k$-modules and whose morphism space %$\Hom_{\C^\XX_{dg}(\k)}(V,V')$
from $V$ to $V'$ is a dg $\XX$-graded $\k$-module whose underlying $\ZZ^2$-graded $\k$-module is $\Hom_{\ZZ^2\text{-gr}}(V,V')$
and whose differential $d$ is given by
\[d(f)=d_{V'}\circ f-(-1)^{\ii}f\circ d_V\]
for $f$ a morphism of bi-degree $\ii+\sii\XX$.
% is a dg category.
	
A \emph{dg $\XX$-graded $\k$-algebra} is a dg $\XX$-graded $\k$-module $(\Gamma,d_\Gamma)$ endowed with a multiplication
\[\begin{array}{ccc}
	\Gamma_i^j\otimes\Gamma_{i'}^{j'}&\to& \Gamma_{i+i'}^{j+j'}\\
	x\otimes y&\mapsto & xy
\end{array}\]
such that the Leibniz rule holds:
\[d_\Gamma(xy)=d_\Gamma(x)y+(-1)^jxd_\Gamma(y) \]
for all $x\in \Gamma^{j}_i$ and all $y\in\Gamma$.
	
Let $(\Gamma,d_\Gamma)$ be a dg $\XX$-graded $\k$-algebra. A \emph{dg $\XX$-graded $\Gamma$-module} is a dg $\XX$-graded $\k$-module $(M,d_M)$ endowed with a $\Gamma$-action from the right
\[\begin{array}{ccc}
	M^{j}_i\otimes \Gamma^{v}_u&\to& M^{j+v}_{i+u}\\
	m\otimes a&\mapsto& ma
\end{array}\]
such that the Leibniz rule holds
\[d_M(ma)=d_M(m)a+(-1)^{j}d_\Gamma(a),\]
for all $m\in M_i^j$ and all $a\in\Gamma$. For two dg $\XX$-graded $\Gamma$-modules $M$ and $N$, we define $\mathcal{H}om_\Gamma(M,N)$ to be the dg $\XX$-graded $\k$-submodule of $\Hom_{\C^\XX_{dg}(\k)}(M,N)$ as follows: % whose degree $i$ component is the subspace
\[
\mathcal{H}om_\Gamma(M,N)=\{f\in \Hom_{\C^\XX_{dg}(\k)}(M,N)\mid f(ma)=f(m)a \text{ for any $m\in M, a\in \Gamma$}\}.
\]
	
The \emph{category $\C(\Gamma)$ of dg $\XX$-graded $\Gamma$-modules} is the category whose objects are the dg $\XX$-graded $\Gamma$-modules, and whose morphisms are $Z^0\mathcal{H}om_\Gamma(M,N)_0$, which consisting of the morphisms $f\in\mathcal{H}om_\Gamma(M,N)$ of bi-degree $0$ satisfying $d(f)=0$. The \emph{homotopy category} $\H(\Gamma)$ is the category whose objects are the dg $\XX$-graded $\Gamma$-modules, and whose morphisms are $H^0\mathcal{H}om_\Gamma(M,N)_0$. The homotopy category $\H(\Gamma)$ is a triangulated category whose suspension functor is the given by $M\mapsto M[1]$. The \emph{derived category} $\D(\Gamma)$ of dg $\XX$-graded $\Gamma$-modules is the localization of $\H(\Gamma)$ at the full subcategory of acyclic dg $\XX$-graded $\Gamma$-modules. Note that in each of the categories $\C(\Gamma)$, $H(\Gamma)$ and $\D(\Gamma)$, the map $M\mapsto M[\XX]$ induces an exact/triangle equivalence.

\begin{definition}
	For any $X_1,X_2\in\D(\Gamma)$, define the double graded Hom $$\Hom^{\ZZ^2}(X_1,X_2):=\bigoplus_{\ii,\sii\in\ZZ}\Hom_{\D(\Gamma)}(X_1,X_2[\ii+\sii\XX])$$
	and its $\qv$-dimension as
	\begin{gather}\label{eq:qHom}
	\qdH(X_1,X_2)\colon=\sum_{\ii,\sii\in\mathbb{Z}}
	\qv^{\ii+\sii\XX}\cdot\dim\Hom_{\D(\Gamma)}(X_1,X_2[\ii+\sii\XX]).
	\end{gather}
	When $X_1=X_2=X$, $\Hom^{\ZZ^2}(X,X)$ becomes a $\ZZ^2$-graded algebra, called the Ext-algebra of $X$ and denoted by $\Ext^{\ZZ^2}(X,X)$.
\end{definition}
By definition, we have
$$\dim\Hom^{\ZZ^2}(X_1,X_2)=\qdH(X_1,X_2)\mid_{\qv=1}.$$
	
\begin{remark}
Let $A=\bigoplus_{i\in\mathbb{Z}}A_i$ be an ordinary graded algebra. Regard it as a dg $\XX$-graded algebra with $A^0=A$ and $A^j=0$, for $j\neq 0$. Then the derived category of graded $A$-modules coincides with $\D(A)$.
\end{remark}

A morphism $s:L\to M$ in $\C(\Gamma)$ is called a \emph{quasi-isomorphism} if its induced map $H^\ast(s):H^\ast(L)\to H^\ast(M)$ is an isomorphism. A dg $\XX$-graded $\Gamma$-module $P$ is called \emph{cofibrant} if
$$\Hom_{\C(\Gamma)}(P,L)\xrightarrow{\Hom_{\C(\Gamma)}(P,s)}\Hom_{\C(\Gamma)}(P,M)$$
is surjective for each quasi-isomorphism $s:L\to M$ that is surjective in each component. Let $P$ be a cofibrant dg $\XX$-graded $\Gamma$-module. Then the canonical map
$$\Hom_{\H(\Gamma)}(P,N)\to \Hom_{\D(\Gamma)}(P,N)$$
is bijective for all dg $\XX$-graded $\Gamma$-module $N$. The canonical projection from $\H(\Gamma)$ to $\D(\Gamma)$ admits a left adjoint functor $\mathbf{p}$ which sends a dg $\XX$-graded $\Gamma$-module $M$ to a cofibrant dg $\XX$-graded $\Gamma$-module $\mathbf{p}M$ quasi-isomorphic to $M$. Thus, we have
$$\Hom_{\D(\Gamma)}(M,N)=\Hom_{\H(\Gamma)}(\mathbf{p}M,N)=H^0\mathcal{H}om_\Gamma(\mathbf{p}M,N)_0.$$
	
The \emph{perfect derived category} $\per(\Gamma)$ is the smallest full subcategory of $\D(\Gamma)$ containing $\Gamma$ and which is stable under taking shifts (i.e. $[\ii+\sii\XX]$, $\ii,\sii\in\ZZ$), extensions and direct summands. The \emph{finite dimensional derived category} $\D_{fd}(\Gamma)$ is the full subcategory of $\D(\Gamma)$ consisting of those dg $\XX$-graded $\Gamma$-modules whose homology is of finite total dimension.

For any dg $\XX$-graded $\Gamma$-modules $M$ and $N$, define $$\RHom(M,N)=\mathcal{H}om_\Gamma(\mathbf{p}M,\mathbf{p}N).$$
Taking homology, we have $$H^\ast(\RHom(M,N))=\Hom^{\ZZ^2}(M,N).$$
For any object $T\in\D(\Gamma)$, denote by $\<T\>$ the closure of $T$ in $\D(\Gamma)$ under shifts (i.e. $[\ii+\sii\XX], \ii,\sii\in\ZZ$), extensions and direct summands (e.g. $\<\Gamma\>=\per(\Gamma)$). We have the following derived Morita equivalence (an $\XX$ variation of \cite[Theorem in Section~4.3]{Ke94}).

\begin{theorem}\label{thm:dme}
	There is a $\k$-linear triangle equivalence
	\[\RHom(T,-):\<T\>\simeq\per(\RHom(T,T)).\]
\end{theorem}

%=========================================================
\subsection{Dg $\XX$-graded quiver algebras}\label{subsec:Dgg}
%=========================================================
	
A \emph{quiver} $Q$ consists of a set $Q_0$ of \emph{vertices}, a set $Q_1$ of \emph{arrows},
and two functions $s,t: Q_1\to Q_0$ sending an arrow to its \emph{starting} and \emph{ending} vertices, respectively.
We denote an arrow by $a:s(a)\to t(a)$.

A (nontrivial) \emph{path} $\rho$ of \emph{length} $s>0$ in $Q$ is a sequence $a_1a_2\cdots a_s$ of arrows with $t(a_i)=s(a_{i+1})$ for $1\leq i<s$. The \emph{starting} and \emph{ending} of a path $\rho=a_1a_2\cdots a_s$ are $s(\rho)=s(a_1)$ and $t(\rho)=t(a_s)$, respectively. The \emph{composition} of paths $\rho$ and $\rho'$ is $\rho\rho'$ if it is again a path (i.e. $t(\rho)=s(\rho')$), or zero otherwise. To each vertex $v\in Q_0$, there is an associated \emph{trivial path} $e_v$ of length 0 with $s(e_v)=t(e_v)=v$. The path algebra $\k Q$ of $Q$ is the $\k$-algebra whose basis is the family of (trivial or nontrivial) paths and whose multiplication is given by the composition of paths.
	
A \emph{$\ZZ\oplus\ZZ\XX$-graded} (or \emph{$\ZZ^2$-graded} for short) quiver is a triple $(Q,|\cdot|_1,|\cdot|_2)$, where $Q$ is a quiver and $|\cdot|_1$ and $|\cdot|_2$ are maps from $Q_1$ to $\mathbb{Z}$. An arrow $a$ in $Q_1$ is called to have \emph{bi-degree} $\deg(a)=\ii+\sii\XX$ if $|a|_1=\ii$ and $|a|_2=\sii$. Any non-trivial path $\rho=a_1a_2\cdots a_s$ has bi-degree
$$\left(\sum_{i=1}^s|a_i|_1\right)+\left(\sum_{i=1}^s|a_i|_2\right)\XX$$
and any trivial path has bi-degree $0$. Thus, the path algebra $\k Q$ becomes a $\ZZ^2$-graded algebra.
	
A \emph{differential} on a $\ZZ^2$-graded quiver $Q$ is a map $d:Q_1\to \k Q$ such that for any $a\in Q_1$, $d(a)$ is a linear combination of paths $p$ of bi-degree $\deg(a)+1$ with $s(a)=s(p)$ and $t(a)=t(p)$, and such that if we extend $d$ to a map $\k Q\to \k Q$ linearly and by the Leibniz rule, then $d\circ d=0$. Thus, $\Gamma:=(\k Q,d)$ is a dg $\XX$-graded algebra.
	
Let $S_i$ be the simple $\Gamma$-module corresponding to a vertex $i$ in $Q_0$ and
$$\SS=\bigoplus_{i\in Q_0}S_i.$$
By \cite[Appendix~A.15]{K8}, we have the following result (cf. also \cite[Lemma~2.15]{KY}).

\begin{theorem}\label{thm:ain}
A basis of the Ext-algebra $\Ext^{\ZZ^2}(\SS,\SS)$ is formed by $$\pi_a:S_i\to S_j[\ii+\sii\XX],$$
where $a$ is an arrow of bi-degree $(1-\ii)-\sii\XX$ or a trivial path at $i$ with $i=j$ and $\ii=\sii=0$.
Moreover, there is an $\XX$-graded $A_\infty$ structure on $\Ext^{\ZZ^2}(\SS,\SS)$ given by $$m_r(\pi_{a_1},\pi_{a_2},\cdots,\pi_{a_r})=\pi_b$$
whenever $a_1a_2\cdots a_r$ appears in the expression of $d(b)$,
and such that there is a quasi-isomorphism
$$\RHom(\SS,\SS)\xrightarrow{\;\;\cong\;\;}\Ext^{\ZZ^2}(\SS,\SS)$$
of $\XX$-graded $A_\infty$ algebras.
\end{theorem}

We have the following useful consequence of Theorem~\ref{thm:dme} and Theorem~\ref{thm:ain}.

\begin{corollary}\label{cor:kd}
	There is a triangle equivalence
	$$\D_{fd}(\Gamma)\simeq\per(\Ext^{\ZZ^2}(\SS,\SS)),$$
	where $\Ext^{\ZZ^2}(\SS,\SS)$ is equipped with the $\XX$-graded $A_\infty$ structure in Theorem~\ref{thm:ain}.
\end{corollary}

%=========================================================
\section{Calabi-Yau-$\XX$ categories from graded decorated marked surfaces}\label{sec:gDMS}
In this section, we introduce the topological model for Calabi-Yau-$\XX$ version
of graded gentle algebras.
%=========================================================
\subsection{Marked surfaces with line field}
%=========================================================
We partially follow \cite{HKK,KS}. A \emph{marked surface} $\surf$ is a compact oriented surface with non-empty boundary $\partial\surf$ and with two finite sets $\M$ and $\Y$ of marked points on $\partial\surf$ satisfying that each connected component of $\partial\surf$ contains the same number (at least one) of marked points in $\M$ and $\Y$, and they are alternative. We call the points in $\M$ (resp. $\Y$) open (resp. closed) marked points.

Let $\PTS$ be the real projectivization of the tangent bundle of $\surf$.
Take a \emph{line filed, or grading} $\lambda$ of $\surf$, that is, a section $\lambda:\surf\to\PTS$. The projection $\PTS\to\surf$ with $\R\mathbb{P}^1\backsimeq \mathrm{S}^1$-fiber gives a short exact sequence
$$0\to\pi_1(\mathrm{S}^1)\to\pi_1(\PTS)\to\pi_1(\surf)\to0,$$
or
$$0\to\ho{1}(\mathrm{S}^1)\to\ho{1}(\PTS)\to\ho{1}(\surf)\to0,$$
or
\[
    0\to \coho{1}(\surf) \to \coho{1}(\PTS) \xrightarrow{\pi_\surf} \coho{1}(\mathrm{S}^1)=\ZZ \to0.
\]
The grading $\lambda$ is determined by a class in $\coho{1}(\PTS)$ (\cite[Lemma~1.2]{LP}), denoted by $[\lambda]$,
induced from a split of $\pi_\surf$ (i.e. $\pi_\surf([\lambda])=1$ in $\coho{1}(\mathrm{S}^1)$).
Such a class is equivalent to a split of $\ho{1}(\mathrm{S}^1)\to\ho{1}(\PTS)$
or a split of $\pi_1(\mathrm{S}^1)\to\pi_1(\PTS)$, as $\ho{1}$ is the abelianization of $\pi_1$.

\begin{definition}
A \emph{graded marked surface} $\surf^\lambda=(\surf,\lambda)$ consists of a marked surface $\surf$ and a line filed $\lambda$.
\end{definition}

Let $\MTS$ be the $\R$-bundle of $\surf$ constructed by gluing $\ZZ$ copies of $\PTS$ cut by $\lambda$. Thus, $\lambda$ determines a $\ZZ$-covering
\begin{equation}\label{eq:covering}
\operatorname{cov}\colon \MTS\to\PTS,
\end{equation}
which sends the $0$ in the fiber $\mathbb{R}T_p\surf^\lambda\cong\mathbb{R}$ to $\lambda(p)$ for any point $p$ in $\surf$.

A \emph{morphism} $(f,\widetilde{f})\colon(\surf,\lambda)\to(\surf',\lambda')$ between two graded marked surfaces is an orientation preserving local diffeomorphism $f\colon\surf\to\surf'$ such that it sends marked points to marked points and $[\grad]=f^*[\grad']$,
regarding $[\grad]\in \coho{1}(\PTS)$, together with
a map $\widetilde{f}\colon\MTS\to \mathbb{R}T{\surf'}^{\grad'}$ that fits into the commutative diagram
\[\xymatrix@C=3pc@R=3pc{
	\MTS \ar[r]^{\widetilde{f}} \ar[d]_{\operatorname{cov}} & \mathbb{R}T{\surf'}^{\grad'}\ar[d]^{\operatorname{cov}'}\\
	\PTS \ar[r]^{\diff f} & \mathbb{P}T\surf'.
}\]
There is a natural automorphism on $(\surf,\lambda)$, called the grading shift $[1]$,
which rotates $\lambda\colon\surf\to\PTS$ clockwise by an angle of $\pi$. Equivalently, $[1]$ is
(the identity map together with) the deck transformation of \eqref{eq:covering}.

%=========================================================
\subsection{Graded DMS}
%=========================================================
Let $(\surf,\lambda)$ be a graded marked surface. The \emph{decorated marked surface} (or \emph{DMS} for short) $\surfo$ of $\surf$ is obtained from $\surf$ by decorating a set $\Tri$ of points in the interior of $\surf$, such that $|\Tri|=|\Y|=|\M|$. The points in $\Tri$ are called \emph{decorations}. A \emph{cut} $\cut$ is a set of curves on $\surf$, pairing (connecting) points in $\Tri$ and $\Y$ without intersections nor self-intersections.

Let $\PTSO$ be the real projectivization of the tangent bundle of $\surf\setminus\Tri$.
By abuse of notation, we will write $\PTSo{}=\PTSO$.
A \emph{grading} $\Lambda$ on $\surfo$ is a class in $\coho{1}(\PTSo{},\ZZ^2)$,
with value $(1,0)$ on each clockwise loop $\{p\}\times\R\mathbb{P}^1$ on $\PTSo{p}$ for $p\notin\Tri$
and value $(-2,1)$ on each clockwise loop $l_Z\times\{x\}$ on $\surf$ around any $Z\in\Tri,x\in\R\mathbb{P}^1$.
For any simple loop $\alpha$ on $\surfo$,
we denote $\Lambda(\alpha)=(\Lambda_1(\alpha),\Lambda_2(\alpha))\in\ZZ^2$,
where $\Lambda_1$ is called the first grading and $\Lambda_2$ is called the second grading.

\begin{definition}
Let $(\surf,\lambda)$ be a graded marked surface with a cut $\cut$. A grading $\Lambda$ on $\surfo$ is called \emph{compatible} with $\lambda$ and $\cut$ if $\Lambda_1(\alpha)=\lambda(\alpha)$,
for any simple loop $\alpha$ on $\surfo$ that does not intersect with $\cut$. A \emph{graded DMS} $\DMS$ of $(\surf,\lambda,\cut)$ consists of the DMS $\surfo$ and a grading $\Lambda$ on $\surfo$ that is compatible with $\lambda$ and $\cut$.
For simplicity, we say $\DMS$ a graded DMS whenever it is a graded DMS of some $(\surf,\lambda,\cut)$.
\end{definition}

For any graded DMS $\DMS$, the first grading $\Lambda_1$ gives rise to a section $\Lambda_1:\surf\setminus\Tri\to \PTSo{}$, which determines a $\ZZ$-covering
\begin{equation}\label{eq:covT}
	\operatorname{cov}_\Tri\colon\MTSo^{\Lambda_1}\to\PTSo{}
\end{equation}
similar as \eqref{eq:covering}.

%=========================================================
\subsection{Graded curves and intersection index}
%=========================================================
Let $\DMS$ be a graded DMS. Denote by $\surfoi=\surf\setminus(\partial\surf\cup\Tri)$. For a curve $c:[0,1]\to \surf$, we always assume $c(t)\in\surfoi$ for any $t\in(0,1)$.
The \emph{inverse} $\overline{c}$ of a curve $c$ is defined by $\overline{c}(t)=c(1-t)$ for any $t\in[0,1]$.

A \emph{grading} $\widetilde{c}$ on \emph{a curve} $c$ is given by a family of (homotopy classes of) paths in $\PTSo{c(t)}$ from $\Lambda_1(c(t))$ to $\dot{c}(t)$, varying continuously with $t\in(0,1)$.
The pair $(c,\widetilde{c})$ is called a \emph{graded curve}, and will be simply denoted by $\widetilde{c}$ usually.
Alternatively, a graded curve $\widetilde{c}$ is a lift of the tangents $\dot{c}(t),0\leq t\leq 1$, in the covering  \eqref{eq:covT},
of a usual curve $c$ on $\surf$, . Note that there are $\ZZ$ lifts of $c$, which are related by shift gradings that $\wc[m](t)=\wc(t)+m$ for any $m\in\ZZ$.

For any graded curves $\ws$ and $\wt$ which are in minimal position with respect to each other, let $p=\us(t_1)=\ut(t_2)\in\surf\setminus\Tri$ be a point (which is possibly in $\M$) where $\us$ and $\ut$ intersect transversally. The \emph{intersection index} \cite{HKK} of $\ws$ and $\wt$ at $p$ is defined to be
\[\ind_p(\ws,\wt)=\ws(t_1)\cdot\kappa\cdot\wt^{-1}(t_2)\ \in\pi_1(\PTSo{p})\cong \mathbb{Z}\]
where $\kappa$ is (the homotopy class of) the path in $\PTSo{p}$ from $\dot{\sigma}(t_1)$ to $\dot{\tau}(t_2)$ given by clockwise rotation by an angle smaller than $\pi$. Equivalently, the intersection index $\ind_p(\ws,\wt)$ is the shift $[i]$ such that
the lift $\wt[i]\mid_p$ is in the interval
$$( \ws\mid_p , \ws[1]\mid_p )\subset \R T_p\surfo \cong \R .$$
By definition, we have the following equalities.

\begin{lemma}\label{lem:+}
	Let $\ws,\wt,\wa$ be graded curves which intersect transversely at $p\in\surfoi$ in clockwise order as in Figure~\ref{fig:3int1}. Then we have
	\begin{equation}\label{eq:3int1}
	\ind_{p}(\ws,\wt)+\ind_{p}(\wt,\wa)=\ind_p(\ws,\wa).
	\end{equation}
	In particular, we have (see \cite[Equation~(2.5)]{HKK})
	\begin{equation}\label{eq:hkk2.5}
	\ind_p(\ws,\wt)+\ind_p(\wt,\ws)=1.
	\end{equation}
\end{lemma}

\begin{figure}[h]\centering
	\begin{tikzpicture}[scale=.6]
	\foreach \j in {3,2}{
		\draw[red] (180+120*\j:3)to(120*\j:3);}
	\draw[red](120:3)to(-60:3)node[right]{$\wt$} (0,0)\nn (3,0)node[right]{$\ws$}(-120:3)node[left]{$\wa$};
	\draw[Emerald,thick,->-=.5,>=stealth]
	(0:1.2)to[bend left=60](-120:1.2);
	\draw[Green,thick,->-=.7,>=stealth]
	(0:.9)to[bend left=15](-60:.9);
	\draw[Emerald!50,thick,->-=.7,>=stealth]
	(-60:.7)to[bend left=15](-120:.7);
	\draw[red] (0,0)node[above]{$p$};
	\end{tikzpicture}
	\caption{Curves intersect at the same point in clockwise}\label{fig:3int1}
\end{figure}
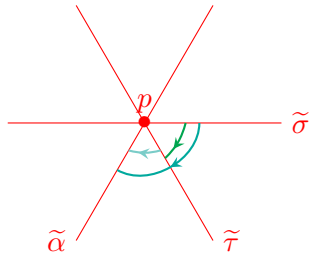

The notion of intersection index can be generalized to the case $p\in\Tri$ as in \cite{KS}. Fix a small circle $l\subset\surf\setminus\Tri$ around $p$. Let $\alpha:[0,1]\to l$ be an embedded arc which moves clockwise around $l$, such that $\alpha$ intersect $\ws$ and $\wt$ at $\alpha(0)$ and $\alpha(1)$, respectively (cf. Figure~\ref{fig:iid}). The arc $\alpha$ is unique up to a change of parametrization. Fixing an arbitrary grading $\widetilde{\alpha}$ on $\alpha$, the \emph{intersection index} $\ind_p(\ws,\wt)$ is defined to be
\begin{equation}\label{eq:ks}
\ind_p(\ws,\wt):=\ind_{\alpha(0)}(\ws,\widetilde{\alpha})-\ind_{\alpha(1)}(\wt,\widetilde{\alpha}).
\end{equation}

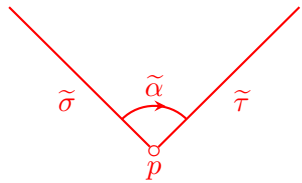
\begin{figure}[htpb]
	\begin{tikzpicture}[scale=.6]
	\draw[red,thick,-<-=.5,>=stealth](45:1)arc(45:135:1) (135:1);
	\draw[red] (90:1)node[above]{$\wa$};
	\draw[red,thick](0,0)to node[below]{$\quad\wt$}(45:4.5);
	\draw[red,thick](0,0)to node[below]{$\ws\quad$}(135:4.5);
	\draw[red](0,0)node[white]{$\bullet$} \ww;
	\draw[red](0,0)node[below]{$p$};
	\end{tikzpicture}
	\caption{Intersection index at a decoration}\label{fig:iid}
\end{figure}

Note that formula \eqref{eq:ks} can be regraded as the decorated intersection version of \eqref{eq:3int1}. The decorated intersection version of \eqref{eq:hkk2.5} is the following.

\begin{lemma}
	Let $\ws,\wt$ be graded curves which intersect transversely at $p\in\Tri$. Then we have
	\begin{equation}\label{eq:hkkd}
	\ind_p(\ws,\wt)+\ind_p(\wt,\ws)=0.
	\end{equation}
\end{lemma}

\begin{proof}
	The first grading $\Lambda_1$ takes value $-2$ for any loop around a decoration with a fixed direction as in the left picture of Figure~\ref{fig:x-2}.
	Thus the grading changes $0$ when going around a decoration with the tangent direction of the circle, as in the right picture of Figure~\ref{fig:x-2}. This implies the equation \eqref{eq:hkkd}.
	
	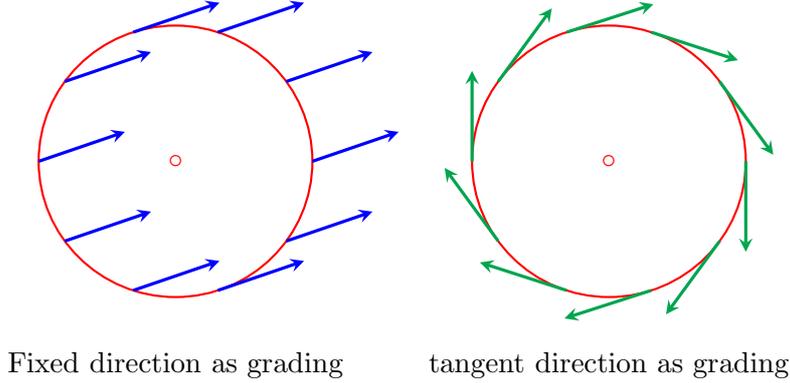
\begin{figure}[t]\centering
		\begin{tikzpicture}[scale=.3]
		\begin{scope}[shift={(0,0)}]
		\draw(0,0)[thick,red]   circle (6) \ww (0,-9)node[black]{Fixed direction as grading};
		\foreach \j in {1,...,10}{
			\draw[blue,very thick,->,>=stealth] (36*\j:6)to($(36*\j:6)+(19:4)$);}
		\end{scope}
		
		\begin{scope}[shift={(19,0)}]
		\draw(0,0)[thick,red]   circle (6) \ww (0,-9)node[black]{tangent direction as grading};
		\foreach \j in {1,...,10}{
			\draw[Green,very thick,->,>=stealth] (36*\j:6)to($(36*\j:6)+(36*\j-90:4)$);}
		\end{scope}
		\end{tikzpicture}
		\caption{Loops around decorations}\label{fig:x-2}
	\end{figure}
\end{proof}

%=========================================================
\subsection{Log DMS}\label{sec:log}
%=========================================================
Next, we introduce the log DMS to unwind the second grading $\Lambda_2$ of DMS.

%Denote by $\Lambda_1$ the section $\surf\setminus\Tri\to \mathbb{P}T\left(\surf\setminus\Tri\right)$ corresponding to the projection of $\Lambda(\alpha)$ on the first $\ZZ$.

\begin{definition}\label{def:logdms}
Let $\DMS$ be a graded DMS. The \emph{log DMS} $\log\surfo$ is obtained from $\surfo$ by
\begin{itemize}
\item taking $\ZZ$ copies of $\surfo$, denoted by $\surfo^m$, $m\in\mathbb{Z}$,
\item cutting each sheet along all arcs $c_i^m\in\cc^m$, where $\cc^m$ is the $m$-th copy of the cut $\cc$, denoting by $c_{i\pm}^m$ the cut marks, and
\item gluing $c_{i+}^m$ with $c_{i-}^{m+1}$ for any $m\in\ZZ$ and $c_i\in\cc$ (see Figure~\ref{fig:LogS}).
\end{itemize}
The log surface $\log\surfo$ inherits the grading from the first grading $\Lambda_1$. The copy $\surfo^m$ is called the \emph{$m$-th sheet} of $\log\surfo$.

\begin{figure}[h]\centering
	\begin{tikzpicture}[xscale=.8,yscale=.27]
	\draw[<->,>=stealth,orange,thick](225:2) to ($(243:2)+(0,-8.5)$);
	\draw[very thick](0,0)to (-120:4)arc(-120:225:4)to(0,0)node[right]{$\surfo^{m+1}$};
	\draw[Green,very thick](225:4) to node[above left]{$c^{m+1}_{i-}$} (0,0);
	\begin{scope}[shift={(0,-8.8)}]
	\draw[gray,thick](0,0)to(-120:4)arc(-120:225:4)to (0,0)node[right]{$\surfo^m$};
	\draw[Green,very thick](-120:4) to node[below]{$c^{m}_{i+}$} (0,0);
	\end{scope}
	\end{tikzpicture}
	\caption{Log surfaces via cuts}\label{fig:LogS}
\end{figure}
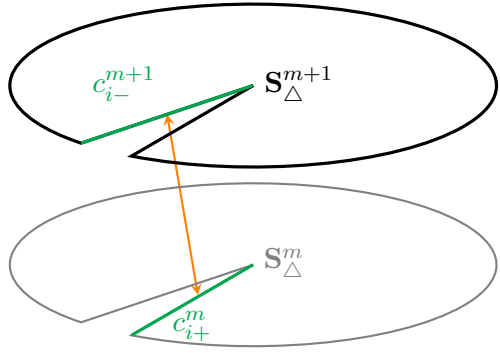

%such that the segment of $\we^m$ from $\we^m(0)$ to its first intersection with $\cut$ is in the sheet $\surfo^m$.
%We will use $[\XX]$ for the $\XX$-grading shift, i.e. $\we^m[m'\XX]=\we^{m+m'}$.

Denote by $\cov:\log\surfo\to\surfo$ the covering map. The deck transformation is denoted by $[\XX]$ and called the \emph{$\XX$-grading shift}.
For any graded curve $\we$ onn $\surfo$ in a minimal position with respect to the cut $\cut$, there are $\ZZ$ \emph{lifts} $\wwe$ of a graded curve $\we$ on $\log\surfo$ (i.e. $\cov(\wwe)=\we$), which are related by $\XX$-grading shifts $[\sii\XX]$, $\sii\in\ZZ$. Any lift $\wwe$ on $\log\surfo$ of a graded curve $\we$ on $\surfo$ is called a \emph{double graded curve} on $\log\surfo$. Similar to the $\MTS$ case, the double graded curve $\wwe$ really lives in $\mathbb{R}T\log\surfo^{\Lambda_1}$.
\end{definition}

For any lifts $\wws$ and $\wwt$ in $\log\surfo$ of graded curves $\ws$ and $\wt$, we call an intersection $p$ of $\ws$ and $\wt$ is an intersection of $\wws$ and $\wwt$ with \emph{bi-index}
\begin{equation}\label{eq:Z2-int}
\ind_p^{\ZZ^2}(\lws,\lwt)=\ind_p(\lws,\lwt)+\sind_p(\lws,\lwt)\XX
\end{equation}
where $\ind_p(\lws,\lwt):=\ind_p(\ws,\wt)$ and $\sind_p(\lws,\lwt)=\sii$ if
\begin{itemize}
\item
$p\notin\Tri$ and some lift of $p$ is an intersection between $\lws$ and $\lwt[\sii\XX]$;
\item
$p\in\Tri$ and
\begin{itemize}
	\item either the segments of $\lws$ and $\lwt[\sii\XX]$ near $p$ are in the same sheet of $\log\surfo$ and the angle from $\lws$ to $\lwt[\sii\XX]$ clockwise around $p$ does not cross the cut $\cut$ (see the left picture of Figure~\ref{fig:lr});
	\item or the segments of $\lws$ and $\lwt[(\sii-1)\XX]$ near $p$ are in the same sheet of $\log\surfo$ and the angle from $\lws$ to $\lwt[(\sii-1)\XX]$ clockwise around $p$ crosses the cut $\cut$
	(see the right picture of Figure~\ref{fig:lr}).
\end{itemize}
\end{itemize}
\begin{figure}[htpb]\centering
	\begin{tikzpicture}[scale=.6]
	\draw[gray!99,thick, dashed](0,0)to node[left]{$\cut$}($(-90+20:4)!.5!(-90-20:4)$);
	\draw[red]($(-90+20:4)!.5!(-90-20:4)$)node[white]{$\bullet$} \ww;
	\draw[blue,very thick,-<-=.5,>=stealth](45:1)arc(45:135:1) (135:1);
	\draw[red,thick](0,0)to node[below]{$\qquad\lwt[\sii\XX]$}(45:4.5);
	\draw[red,thick](0,0)to node[above]{$\quad\lws$}(135:4.5);
	\foreach \j in {1,0}{
		\draw[blue,very thick](45+90*\j+20:4)to(45+90*\j-20:4);}
	\draw[dashed,blue,thin](45+20:4)to[bend left=-15](135-20:4)
	(45-20:4)to[bend left=45](-90+20:4)(135+20:4)to[bend left=-45](-90-20:4);
	\draw[very thick](-90+20:4)to(-90-20:4);
	\draw[red](0,0)node[white]{$\bullet$} \ww;
	\end{tikzpicture}\quad
	\begin{tikzpicture}[scale=.6]
	\draw[gray!99,thick, dashed](0,0)to node[left]{$\cut$}($(-90+20:4)!.5!(-90-20:4)$);
	\draw[red]($(-90+20:4)!.5!(-90-20:4)$)node[white]{$\bullet$} \ww;
	\draw[red,thick](0,0)to node[above]{$\qquad\qquad\lwt[(\sii-1)\XX]$}(135:4.5);
	\draw[red,thick](0,0)to node[below]{$\quad\lws$}(45:4.5);
	\foreach \j in {1,0}{
		\draw[blue,very thick](45+90*\j+20:4)to(45+90*\j-20:4);}
	\draw[dashed,blue,thin](45+20:4)to[bend left=-15](135-20:4)
	(45-20:4)to[bend left=45](-90+20:4)(135+20:4)to[bend left=-45](-90-20:4);
	\draw[very thick](-90+20:4)to(-90-20:4);
	\draw[red](0,0)node[white]{$\bullet$} \ww;
	\draw[blue,very thick,->-=.55,>=stealth](45:1)arc(45:-225:1) (135:1);
	\end{tikzpicture}
	\caption{$\XX$-index $\sind_p(\lws,\lwt)$ at a decoration}\label{fig:lr}
\end{figure}

We generalize equations \eqref{eq:3int1}, \eqref{eq:hkk2.5}, \eqref{eq:ks} and \eqref{eq:hkkd} to the bi-index version, which will be used later.

\begin{lemma}
Let $\lws,\lwt$ be double graded curves in $\surfo$ with an intersection $p\in\surf\setminus\partial\surf$. If $p\notin\Tri$, we have
\begin{equation}\label{eq:hkkg}
	\ind^{\ZZ^2}_p(\lws,\lwt)+\ind^{\ZZ^2}_p(\lwt,\lws)=1.
\end{equation}
If $p\in\Tri$, we have
\begin{equation}\label{eq:hkkg1}
	\ind^{\ZZ^2}_p(\lws,\lwt)+\ind^{\ZZ^2}_p(\lwt,\lws)=\XX.
\end{equation}
\end{lemma}

\begin{proof}
	By definition, we have
	$$\ind^{\XX}_p(\lws,\lwt)+\ind^{\XX}_p(\lwt,\lws)=\begin{cases}
	0&\text{if $p\notin\Tri$,}\\
	\XX&\text{if $p\in\Tri$.}
	\end{cases}$$
	Combining this with \eqref{eq:hkk2.5} and  \eqref{eq:hkkd}, we have the required formulas.
\end{proof}

\begin{lemma}
	Let $\lws,\lwt,\lwa$ be double graded curves on $\surfo$. If they are in the case in Figure~\ref{fig:3int1}, we have
	\begin{equation}\label{eq:3int}
	\ind^{\ZZ^2}_p(\lws,\lwa)=\ind^{\ZZ^2}_{p}(\lws,\lwg)+\ind^{\ZZ^2}_{p}(\lwg,\lwa).
	\end{equation}
	If they are in the case in Figure~\ref{fig:iid}, we have
	\begin{equation}\label{eq:ksg}
		\ind^{\ZZ^2}_p(\lws,\lwt)=\ind^{\ZZ^2}_{\alpha(0)}(\lws,\lwa)-\ind^{\ZZ^2}_{\alpha(1)}(\lwt,\lwa).
	\end{equation}
\end{lemma}

\begin{proof}
	For the case in Figure~\ref{fig:3int1}, by definition, we have
	$$\ind^{\XX}_p(\lws,\lwa)=\ind^{\XX}_{p}(\lws,\lwg)+\ind^{\XX}_{p}(\lwg,\lwa).$$
	So we deduce \eqref{eq:3int} by \eqref{eq:3int1}.
	
	For the case in Figure~\ref{fig:iid}, by \eqref{eq:ks}, we only need to show
	$$\ind^{\XX}_p(\lws,\lwt)=\ind^{\XX}_{\alpha(0)}(\lws,\lwa)-\ind^{\XX}_{\alpha(1)}(\lwt,\lwa).$$
	Assume $\ind^{\ZZ^2}_p(\lws,\lwt)=\sii$ and $\ind^{\XX}_{\alpha(0)}(\lws,\lwa)=\sii'$. By definition, $\lws$ and $\lwt$ are in one of the cases in Figure~\ref{fig:lr} and $\lwa[\sii'\XX]$ crosses $\lws$ at some lift of $\alpha(0)$. 	
	Then in both cases in Figure~\ref{fig:lr}, we have that $\lwa[\sii'\XX]$ crosses $\lwt[\sii\XX]$ at some lift of $\alpha(1)$. So we have
	$\ind_{\alpha(1)}(\lwt,\lwa)=\sii'-\sii,$ which implies the required equality.
%	So we have $\ind^{\XX}_{\alpha(0)}(\lws,\wa^{m''})-\ind_{\alpha(1)}(\lwt,\wa^{m''})=m''-m'=\ind^{\XX}_p(\lws,\lwt)$.
%	If  $\lws$ and $\lwt$ are in the case of the right picture of Figure~\ref{fig:lr}, assume $\lwa[\sii'\XX]$ crosses $\lws$ at some lift of $\alpha(0)$. Then $\lwa[\sii'\XX]$ crosses $\lwt$ at some lift of $\alpha(1)$. So we have
%	$\ind^{\XX}_{\alpha(0)}(\lws,\lwa)=\sii',\ \ind_{\alpha(1)}(\lwt,\lwa)=\sii'-\sii,$ which imply the required equality.
%	
%	we have
%	$$\ind^{\XX}_{\alpha(0)}(\lws,\wa^{m''})=m-m'',\ \ind_{\alpha(1)}(\lwt,\wa^{m''})=m'-m''-1.$$
%	So we have $\ind^{\XX}_{\alpha(0)}(\lws,\wa^{m''})-\ind_{\alpha(1)}(\lwt,\wa^{m''})=m''-m'+1=\ind^{\XX}_p(\lws,\lwt)$. Combining this with  we get the formula \eqref{eq:ksg}.
\end{proof}

For any $(\ii,\sii)\in\ZZ^2$, we denote by $\cap^{\ii+\sii\XX}(\lws,\lwt)$ the set of intersections between $\lws$ and $\lwt$ with bi-index $\ii+\sii\XX$. We will use the notations \begin{gather*}
\Int^{\ii+\sii\XX}_{\Tri}(\lws,\lwt)\colon=\left|\cap^{\ii+\sii\XX}(\lws,\lwt)\cap \Tri\right|,\\
\Int^{\ii+\sii\XX}_{\surfoi}(\lws,\lwt)\colon=\left|\cap^{\ii+\sii\XX}(\lws,\lwt)\cap \surfoi\right|,\\
\Int^{\ii+\sii\XX}_{\surfo}(\lws,\lwt)\colon=
\frac{1}{2}\cdot\Int^{\ii+\sii\XX}_{\Tri}(\lws,\lwt)+
\Int^{\ii+\sii\XX}_{\surfoi}(\lws,\lwt),
\end{gather*}
for the bi-index $(\ii+\sii\XX)$ intersection numbers at decorations, in the interior, and at all of $\surf\setminus\partial\surf$, respectively.
The \emph{total intersection}
$$
    \Int_{?}(\lws,\lwt)=\sum_{\ii,\sii\in\ZZ}\Int^{\ii+\sii\XX}_?(\lws,\lwt)
$$
is the sum over all indices, where $?=\Tri,\surfoi$ or $\surfo$.

The following notion is useful in the paper.

\begin{definition}[Extension of curves]\label{def:ext}
Let $\sigma,\tau$ be two curves in $\surfo$ with $\sigma(0)=\tau(0)\in\Tri$.
The (positive) extension $\tau\wedge\sigma$ of $\tau$ by $\sigma$ (with respect to the common starting point) is defined in Figure~\ref{fig:ext.},
which consists of the operation so-called smoothing out at such an intersection
and possibly (at most) one operation--the unknotting a non-admissible intersection.
	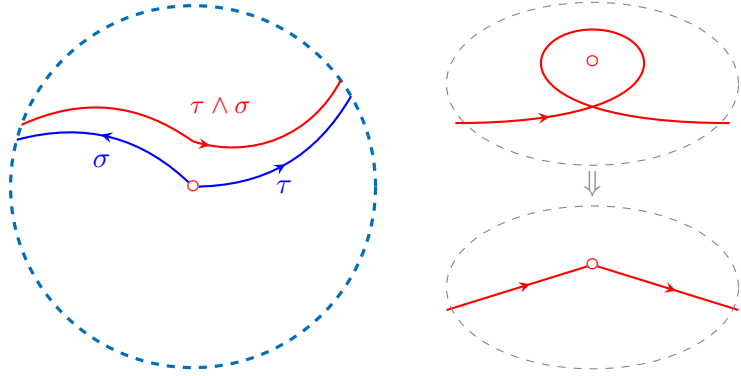
\begin{figure}[ht]\centering
		\begin{tikzpicture}[scale=1.2]
		\draw[NavyBlue,dashed,very thick](0,0)circle(2);
		\draw[thick,blue](-195:2)edge[bend left,-<-=.5,>=stealth](0,0)
		(0,0)edge[bend right,->-=.5,>=stealth](30:2)
		(1,0)node{$\tau$}(-1,.1)node[above]{$\sigma$};
		\draw[thick,red](160:2)to[bend left](0,.5)(0.3,.7)
		node[above]{\small{$\tau\wedge\sigma$}};
		\draw[thick,red](0,0.5)edge[bend right=40,->-=.1,>=stealth](36:2);
		\draw(0,0)node[white] {$\bullet$} node[red](a){$\circ$};
		\end{tikzpicture}
        \qquad
    \begin{tikzpicture}[xscale=.48,yscale=.6]\draw[gray, dashed](0,-.5)ellipse(4 and 1.8);
	\draw[thick,red,-<-=.85,>=stealth](-20:4) .. controls +(180:6.5) and +(180:2) .. (0,.7)
        .. controls +(0:2) and +(-0:6.5) .. (200:4);
	\draw[thick,red,-<-=.5,>=stealth](4,-5.5)to(0,-4.5);
	\draw[thick,red,-<-=.5,>=stealth](0,-4.5)to (-4,-5.5) (0,-2.7)node[gray]{$\Downarrow$};
	\draw(0,-4.5)node[white] {$\bullet$};
	\draw[red](0,0)\ww(0,0) (0,-4.5)\ww;
    \draw[gray, dashed](0,-5)ellipse(4 and 1.8);
	\end{tikzpicture}
		\caption{The extension as smoothing out and unkontting}
		\label{fig:ext.}
	\end{figure}
	
	Let $\lws$ and $\lwt$ be double graded curves on $\surfo$ with underlying curves $\us$ and $\ut$, respectively. The extension $\lwt\wedge\lws$ is the double graded curve whose underlying curves is $\us\wedge\wt$ and which inherits the double grading from $\lwt$.
\end{definition}

%=========================================================
\subsection{Types of curves and  $\qv$-intersections}
%=========================================================

\begin{definition}\label{not:CA}
	We have the following types of curves in $\surfo$.
	\begin{itemize}
		\item A curve $c$ is called \emph{open} (resp. \emph{closed}) if $c(0)$ and $c(1)$ are open marked points (resp. decorations), i.e., in $\M$ (resp. $\Tri$).
		\item An \emph{open arc} is an open curve without self-intersections in $\surfoi$. We call two open arcs do not cross each other if they do not have intersections in $\surfoi$.
		\item A \emph{closed arc} is a closed curve without self-intersections in $\surfoi$ and whose two endpoints are not the same decoration.
		\item A closed curve is called \emph{admissible} if it does not cut out a once-decorated monogon by one of its self-intersections in $\surfoi$. See the upper right picture in Figure~\ref{fig:ext.} for the failure of the condition.
	\end{itemize}
	We always consider open/closed curves up to taking inverse and homotopy relative to endpoints.
\end{definition}

By definition, any closed arc is an admissible closed curve. Denote by $\CA(\surfo)\subset\ACC(\surfo)$, respectively, the set of closed arcs and the set of admissible closed curves, and denote by $\wCA(\surfo)\subset\wACC(\surfo)$, respectively, the set of graded closed arcs and the set of graded admissible closed curves.

For any $\we\in\wCA(\surfo)$, we call any of its lifts in $\log\surfo$ a \emph{double graded closed arc}; for any $\we\in\wACC(\surfo)$, we call any of its lifts in $\log\surfo$ a \emph{double graded admissible closed curve} in $\log\surfo$. Denote by $\wXCA(\surfo)\subset\wXACC(\surfo)$, respectively, the set of double graded closed arcs and the set of double graded admissible closed curves.

\begin{definition}[$\qv$-intersections]
	The \emph{$\ZZ^2$-graded $\qv$-intersection} of $\lws,\lwt\in\wXACC(\surfo)$ is defined to be
	\begin{equation}\label{eq:q-int}
	\begin{array}{rcl}
	\qqInt(\lws,\lwt)&=&
	\displaystyle\sum_{\ii,\sii\in\mathbb{Z}}
	\qv^{\ii+\sii\XX}\cdot\Int_{\Tri}^{\ii+\sii\XX} (\lws,\lwt)\\
	&&+(1+\qv^{\XX-1})
	\displaystyle\sum_{\ii,\sii\in\mathbb{Z}}
	\qv^{\ii+\sii\XX}\cdot\Int_{\surfoi}^{\ii+\sii\XX} (\lws,\lwt).\end{array}
	\end{equation}
	Note that we have $\qqInt(-,-)\mid_{\qv=1}=2\Int_{\surfo}(-,-)$.

	Define the \emph{$\ZZ^2$-graded $\qv$-intersection} of a lift $\lwg$ of a graded open curve $\gamma$ and $\lwt\in\wXACC(\surfo)$ to be
	\begin{equation}\label{eq:q-int2}
	\qqInt(\lwg,\lwt)=
	\displaystyle\sum_{\ii,\sii\in\mathbb{Z}}
	\qv^{\ii+\sii\XX}\cdot\Int_{\surfoi}^{\ii+\sii\XX} (\lwg,\lwt).
	\end{equation}
\end{definition}

%=========================================================
\subsection{Dg $\XX$-graded quiver algebras from surfaces}
%=========================================================
\begin{definition}\label{def:f.f.a.s.}
	An \emph{(open) full formal arc system} $\ac$ of a graded DMS $\DMS$ is a collection of pairwise non-crossing graded open arcs
	which do not intersect the cut $\cc$, such that it divides the surface $\surf$ into polygons,
	called \emph{$\ac$-polygons}, satisfying that
	each $\ac$-polygon contains exactly one decoration in $\Tri$ (or equivalently, contains exactly one arc in the cut $\cut$).
\end{definition}

Let $\ac=\{\wg_i\mid 1\leq i\leq n \}$ be a full formal arc system of $\DMS$. Note that any arc $\wg_i$ in $\ac$ does not cross any arc in the cut $\cc$. So we have the lifts $\wg_i^m$, $m\in\ZZ$ in $\log\surfo$ of $\wg_i$, with each $\wg_i^m$ in the sheet $\surfo^m$. We denote by $\ac^\ZZ:=\cup_{m\in\mathbb{Z}}\ac^m$, where $\ac^m=\{\wg_i^m\mid 1\leq i\leq n,\ m\in\ZZ \}$.

%. Note that any edge of an $\ac$-polygon is either an arc in $A$ or a boundary arc of $\surf$. An open arc system $A$ is called  if any $\ac$-polygon has exactly  called

%The number of graded open arcs in any $\ac$ only depends on the numerical data of $\surf$,
%which is the analogue of the well-known fact that the number of arcs in a triangulation
%only depends on the numerical data of a marked surface.
%Such a number is called the rank of $\surf$, which is given by the following formula.
%We leave the proof to the readers.
%
%\begin{lemma}%\label[]{HKK}
%	Let $\ac$ be a full formal open arc system of $\DMS$. Then
%	\[|\ac|=2g+b+|\M|-2,\]
%	where $g$ is the genus of $\surf$ and $b$ is the number of components of $\partial\surf$.
%\end{lemma}

\begin{definition}[Arc segments]
Let $\ac$ be a full formal arc system of $\DMS$.
An \emph{arc segment} is a curve $\rho:[0,1]\to \surfoi$ without self-intersections and such that
its interior is in the interior of an $\ac$-polygon $P$ and whose endpoints are on the edges of $P$. Any arc segment $\rho$ will be considered up to isotopy with respect to the interior of the edges of $P$.
	
An arc segment $\rho$ in $P$ is called \emph{positive} (resp. \emph{negative}),
if the decoration in $P$ is on the right (resp. left) hand side of $\rho$ when going from $\rho(0)$ to $\rho(1)$.
Note that the inverse of an arc segment $\rho$ in $P$ has the opposite sign of $\rho$.
See the green arc segments in Figure~\ref{fig:st}.
%\Old{An arc segment $\rho$ in $P$ is called of \emph{type II} (resp. \emph{type I}) if $\rho$ does (resp. does not) intersect the cut (gray dashed arcs in Figure~\ref{fig:st}).}
\end{definition}
\begin{figure}[h]\centering
	\begin{tikzpicture}[scale=-.4]
	\foreach \j in {1,2,0,3}{
		\draw[blue,very thick](90*\j+20:4)to(90*\j-20:4);
		\draw[dashed,blue,thin](90*\j+20:4)to[bend left=-15](90*\j-20+90:4);}
	\draw[very thick, Emerald,->-=.5,>=stealth](-5:3.7)to[bend left=45]node[below]{$+$}(185:3.7);
	\draw[very thick](90+20:4)to(90-20:4);
	\draw[gray!99,very thick, dashed](0,0)to($(90+20:4)!.5!(90-20:4)$);
	\draw[red](0,0)node[white]{$\bullet$} \ww;
	\end{tikzpicture}\quad
	\begin{tikzpicture}[scale=-.4]
	\foreach \j in {1,2,0,3}{
		\draw[blue,very thick](90*\j+20:4)to(90*\j-20:4);
		\draw[dashed,blue,thin](90*\j+20:4)to[bend left=-15](90*\j-20+90:4);}
	\draw[very thick, Emerald,-<-=.4,>=stealth](-5:3.7)to[bend left=-45]node[below]{$+\qquad$}(185:3.7);
	\draw[very thick](90+20:4)to(90-20:4);\draw[Emerald](-90:2);
	\draw[gray!99,very thick, dashed](0,0)to($(90+20:4)!.5!(90-20:4)$);
	\draw[red](0,0)node[white]{$\bullet$} \ww;
	\end{tikzpicture}\quad
	\begin{tikzpicture}[scale=-.4]
	\foreach \j in {1,2,0,3}{
		\draw[blue,very thick](90*\j+20:4)to(90*\j-20:4);
		\draw[dashed,blue,thin](90*\j+20:4)to[bend left=-15](90*\j-20+90:4);}
	\draw[very thick, Emerald,-<-=.5,>=stealth](-5:3.7)to[bend left=45]node[below]{$-$}(185:3.7);
	\draw[very thick](90+20:4)to(90-20:4);
	\draw[gray!99,very thick, dashed](0,0)to($(90+20:4)!.5!(90-20:4)$);
	\draw[red](0,0)node[white]{$\bullet$} \ww;
	\end{tikzpicture}\quad
	\begin{tikzpicture}[scale=-.4]
	\foreach \j in {1,2,0,3}{
		\draw[blue,very thick](90*\j+20:4)to(90*\j-20:4);
		\draw[dashed,blue,thin](90*\j+20:4)to[bend left=-15](90*\j-20+90:4);}
	\draw[very thick, Emerald,->-=.65,>=stealth](-5:3.7)to[bend left=-45]node[below]{$\qquad-$}(185:3.7);
	\draw[very thick](90+20:4)to(90-20:4);\draw[Emerald](-90:2);
	\draw[gray!99,very thick, dashed](0,0)to($(90+20:4)!.5!(90-20:4)$);
	\draw[red](0,0)node[white]{$\bullet$} \ww;
	\end{tikzpicture}
	\caption{Signs of two pairs of arc segments}\label{fig:st}
\end{figure}
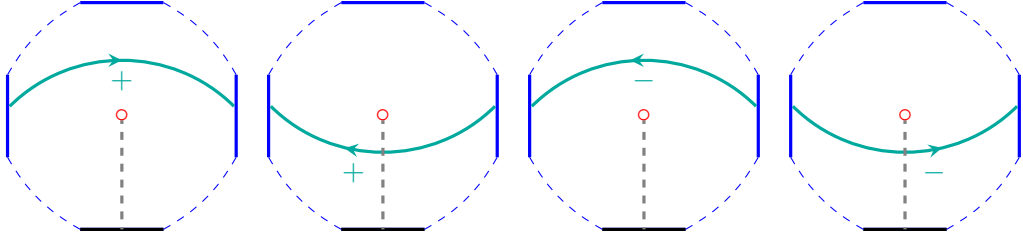
Define an equivalent relation on positive arc segments, that $\rho_1\sim\rho_2$ if and only if $\rho_1$ and $\rho_2$ are in two neighbor $\ac$-polygons and form a digon (see Figure~\ref{fig:digon}).

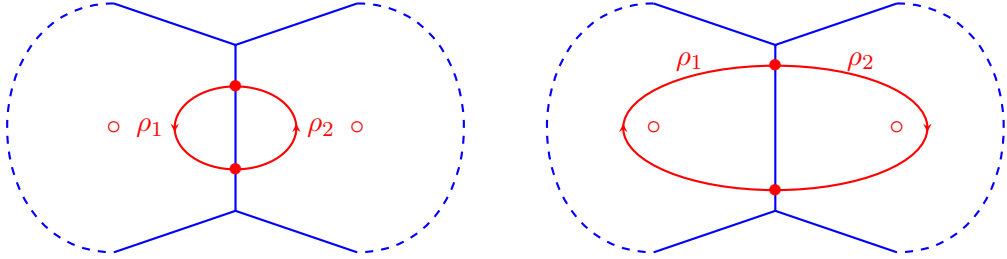
\begin{figure}[htpb]\centering
	\begin{tikzpicture}[xscale=.8,yscale=.55]
	\draw[thick,blue](0,3)to(2,2)to(4,3) (2,2)to(2,-2) (0,-3)to(2,-2)to(4,-3);
	\draw[thick,blue,dashed](0,3)to[bend left=-90,dashed](0,-3) (4,3)to[bend left=90,dashed](4,-3);
	\draw[red, thick](0,0)\ww(4,0)\ww (2,0) ellipse (1 and 1) (1,0)node[left]{$\rho_1$}(3,0)node[right]{$\rho_2$}
	(2,1)\nn(2,-1)\nn;
	\draw[red] (1,0) \jiantou (1,-.1);\draw[red] (3,0) \jiantou (3,.1);
	\end{tikzpicture}
	\begin{tikzpicture}[xscale=.8,yscale=.55]
	\draw[thick,blue](0,3)to(2,2)to(4,3) (2,2)to(2,-2) (0,-3)to(2,-2)to(4,-3);
	\draw[thick,blue,dashed](0,3)to[bend left=-90,dashed](0,-3) (4,3)to[bend left=90,dashed](4,-3);
	\draw[red, thick](0,0)\ww(4,0)\ww (2,0) ellipse (2.5 and 1.5) (1,1.6)node[left]{$\rho_1$}(3,1.6)node[right]{$\rho_2$}
	(2,1.5)\nn(2,-1.5)\nn;
	\draw[red] (-.5,0) \jiantou (-.5,0.1);
	\draw[red] (4.5,0) \jiantou (4.5,-.1);
	\end{tikzpicture}
	\caption{The equivalent relation on positive arc segments}\label{fig:digon}
\end{figure}

For any positive arc segment $\rho$, we denote by $[\rho]$ the class of double graded curves on $\surfo$ whose underlying curve is equivalent to $\rho$. Let $\pas(\ac)$ the set of classes $[\rho]$ of positive arc segments $\rho$. For any $[\rho_1],[\rho_2]\in\pas(\ac)$, if there are representatives $\overset{\approx}{\rho'_1}\in[\rho_1]$ and $\overset{\approx}{\rho'_2}\in[\rho_2]$ such that the composition $\rho'_1\cdot\rho'_2$ is again a positive arc segment, we define $[\rho_1]\cdot[\rho_2]=[\rho'_1\cdot\rho'_2]$ (see Figure~\ref{fig:com}).

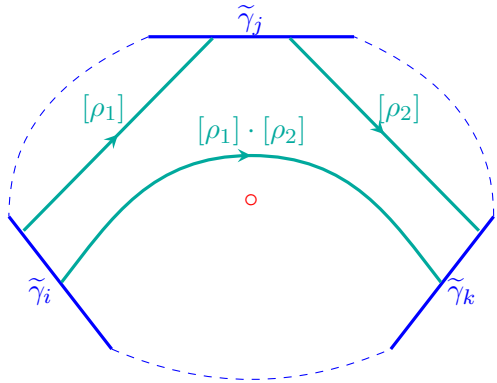
\begin{figure}[htpb]
	\begin{tikzpicture}[xscale=.8,yscale=.6]
	\foreach \j in {0,1,2}{
		\draw[blue,very thick](120*\j+115:4)to(120*\j+65:4);
		\draw[dashed,blue,thin](120*\j+115:4)to[bend right=30](120*\j+185:4);}
	\draw[very thick, Emerald,->-=.5,>=stealth](-150:3.6)[out=60,in=180]to(90:1)[out=0,in=120]to(-30:3.6);
	\draw[very thick, Emerald,->-=.5,>=stealth](-170:3.8)to node[above]{$[\rho_1]\quad$}(100:3.65);
	\draw[very thick, Emerald,->-=.5,>=stealth](80:3.65)to node[above]{$\quad[\rho_2]$}(-10:3.8);
	\draw[Emerald] (90:1)node[above]{$[\rho_1]\cdot[\rho_2]$};
	\draw[red](0,0)node[white]{$\bullet$} \ww;
	\draw[blue] (0,4)node{$\wg_j$} (-30:4)node{$\wg_k$} (-150:4)node{$\wg_i$};
	\end{tikzpicture}
	\caption{Composition of positive arc segments}\label{fig:com}
\end{figure}

A class $[\rho]$ is called \emph{trivial} if $\rho$ is isotopy to a segment of an arc in $\ac$,
i.e. the ones in the left picture of Figure~\ref{fig:digon}.
A class $[\rho]$ is called \emph{loop-type} if its endpoints are in the same arc of $\ac$,
and together with that arc, it encloses a decoration as shown in the right picture of Figure~\ref{fig:digon}.

% with $[\gamma_1],[\gamma_2]$.
%We also consider an arc segment having a decoration as an endpoint.

%We regard $\wg_i\in\ac$ as $\wg_i^0\in\ac^\ZZ$ sometimes.

\begin{definition}\label{def:Z2 Q}
For a full formal arc system $\ac=\{\widetilde{\gamma}_i\mid 1\leq i\leq n \}$ of $\DMS$, there is an associated ($\ZZ\oplus\ZZ\XX$=)$\ZZ^2$-graded quiver $Q_\ac$ with a map $d_\ac:(Q_\ac)_1\to\k Q_\ac$ given by the following data.
\begin{itemize}
	\item The vertices of $Q_\ac$ are (indexed by) the open arcs in $\ac$: $i:=\widetilde{\gamma}_i$ for $1\leq i\leq n$.
	\item Each non-trivial $[\rho]\in\pas(\ac)$ with $\rho(0)\in\gamma_i$ and $\rho(1)\in\gamma_j$ induces a $\ZZ^2$-graded arrow $b_{[\rho]}:i\to j$ of bi-degree
	\begin{equation}\label{eq:2}
	   \deg(b_{[\rho]})\colon=\ind^{\ZZ^2}_{\rho(1)}(\wg_j^0,\lwr)-\ind^{\ZZ^2}_{\rho(0)}(\wg_i^0,\lwr),
	\end{equation}
	where $\lwr$ is an arbitrary representative in $[\rho]$.	
	\item The map $d_\ac$ is given by
	\begin{equation}\label{eq:diff}
	d_\ac(b_{[\rho]})=\sum_{[\rho]=[\rho_1]\cdot[\rho_2]}(-1)^{|b_{[\rho_1]}|_1}b_{[\rho_1]}b_{[\rho_2]},
	\end{equation}
	where the sum runs over all pairs $([\rho_1],[\rho_2])$ of non-trivial positive arc segments $[\rho_1],[\rho_2]\in\pas(\ac)$ satisfying $[\rho]=[\rho_1]\cdot[\rho_2]$.
	\end{itemize}
\end{definition}
	
\begin{lemma}\label{lem:2}
	For any nontrivial $[\rho_1],[\rho_2]\in\pas(\ac)$ such that $[\rho_1]\cdot[\rho_2]$ exists, we have
	$$|b_{[\rho_1]\cdot[\rho_2]}|_1=|b_{[\rho_1]}|_1+|b_{[\rho_2]}|_1-1.$$
\end{lemma}

\begin{proof}
	Let $\wg_i,\wg_j,\wg_k$ be the graded arcs in $\ac$ where $\rho_1(0),\rho_1(1),\rho_2(1)$ are, respectively (cf. Figure~\ref{fig:com}). For any grading $\wr$ of $\rho$, we may take the gradings $\wr_1,\wr_2$ of $\rho_1$ and $\rho_2$ respectively such that $\ind_{\rho_2(1)}(\wg_k,\wr_2)=\ind_{\rho(1)}(\wg_k,\wr)$ and $\ind_{\rho_1(0)}(\wg_i,\wr_1)=\ind_{\rho(0)}(\wg_i,\wr)$. Then we have $\ind_{\rho_1(1)}(\wg_j,\wr_1)=\ind_{\rho_2(0)}(\wg_j,\wr_2)+1$. Hence
	$$\begin{array}{rcl}
	|b_{[\rho]}|_1&=&\ind_{\rho(1)}(\wg_k,\wr)-\ind_{\rho(0)}(\wg_i,\wr) \vspace{1ex}\\
	&=&\ind_{\rho_2(1)}(\wg_k,\wr_2)-\ind_{\rho_1(0)}(\wg_i,\wr_1) \vspace{1ex}\\
	&=&(\ind_{\rho_2(1)}(\wg_k,\wr_2)-\ind_{\rho_2(0)}(\wg_j,\wr_2))+(\ind_{\rho_1(1)}(\wg_j,\wr_1)-\ind_{\rho_1(0)}(\wg_i,\wr_1))-1 \vspace{1ex}\\
	&=&|b_{[\rho_1]}|_1+|b_{[\rho_2]}|_1-1.
	\end{array}
$$
\end{proof}

\begin{lemma}
	The map $d_\ac$ is a differential. That is, $d_\ac^2(b_{[\rho]})=0$ for any non-trivial $[\rho]\in\pas(\ac)$.
\end{lemma}
	
\begin{proof}
The statement follows from a direct calculation. Indeed,
\[\begin{array}{rl}
	&d_\ac^2(b_{[\rho]})\vspace{1ex} \\
    \xlongequal&d_\ac\left(\sum_{[\rho_1][\rho_2]=[\rho]}(-1)^{|b_{[\rho_1]}|_1}b_{[\rho_1]}b_{[\rho_2]}\right)
        \vspace{1ex} \\
    \xlongequal{}&\displaystyle\sum_{[\rho_1][\rho_2]=\rho}(-1)^{|b_{[\rho_1]}|_1}\left(d_\ac([b_{\rho_1}])b_{[\rho_2]}+(-1)^{|b_{[\rho_1]}|_1}b_{[\rho_1]}d_\ac(b_{[\rho_2]})\right)
        \vspace{1ex} \\
    =&\displaystyle\sum\limits_{\begin{smallmatrix}
			[\rho_1][\rho_2]=[\rho]\\
			[\rho'_1][\rho''_1]=[\rho_1]
			\end{smallmatrix}}(-1)^{|b_{[\rho'_1]}|_1+|b_{[\rho_1]}|_1}b_{[\rho'_1]}b_{[\rho''_1]}b_{[\rho_2]}
		+\displaystyle\sum\limits_{\begin{smallmatrix}
			[\rho_1][\rho_2]=[\rho]\\
			[\rho'_2][\rho''_2]=[\rho_2]
			\end{smallmatrix}}(-1)^{|b_{[\rho'_2]}|_1}b_{[\rho_1]}b_{[\rho'_2]}b_{[\rho''_2]}
        \vspace{1ex} \\
    =&\displaystyle\sum\limits_{[\rho_1][\rho_2][\rho_3]=[\rho]}(-1)^{|b_{[\rho_1]}|_1+|b_{[\rho_1][\rho_2]}|_1}b_{[\rho_1]}b_{[\rho_2]}b_{[\rho_3]}
        +\displaystyle\sum\limits_{[\rho_1][\rho_2][\rho_3]=[\rho]}(-1)^{|b_{[\rho_2]}|_1}b_{[\rho_1]}b_{[\rho_2]}b_{[\rho_3]}
        \vspace{1ex} \\
    \xlongequal{}&0,
\end{array}\]
where the second and the last equalities hold due to Leibniz rule and Lemma ~\ref{lem:2} respectively.
\end{proof}
	
So $(\k Q_\ac,d_\ac)$ gives a dg $\XX$-graded algebra $\Gamma_\ac$ (cf. Section~\ref{subsec:Dgg}), after extending $d_\ac$ linearly to a map from $\k Q_\ac$ to $\k Q_\ac$ of bi-degree $1$ via the Leibniz rule.

\begin{remark}\label{rmk:cy}
	Note that, for any $[\rho]\in\pas(\ac)$ of loop-type, we have  $\deg(b_{[\rho]})=1-\XX$.
	Moreover, (classes of) positive arc segments in $\pas(\ac)$ come in pairs, namely,
	\begin{itemize}
		\item a loop-type/trivial arc segment will be paired with a trivial/loop-type arc segment,
		where their endpoints are in the same arc of $\ac$.
		\item any other (non-trivial, non-loop-type) $[\rho]\in\pas(\ac)$ with $b_{[\rho]}\colon i\to j$
		can be uniquely paired with $[\rho]^*\in\pas(\ac)$ with $b_{[\rho]^*}\colon j\to i$, such that
		\[
		\deg(b_{[\rho]})+\deg(b_{[\rho]^*})=2-\XX.
		\]
		For example, the green arcs in the first two pictures in Figure~\ref{fig:st} are paired together.
	\end{itemize}
	Let $Q_\ac^0$ be the $\XX$-degree zero part of $Q_\ac$, i.e. the quiver with the same vertices
	and with arrows $b$ such that $\deg(b)\in\ZZ\subset\ZZ\oplus\ZZ\XX$.
	The property above basically means $Q_\ac$ is the Calabi-Yau-$\XX$ double of $Q_\ac^0$,
	in the sense of \cite[Def.~6.2]{KQ}.
	Equivalently, for any pair  $([\rho],[\rho]^*)$, where neither of them is trivial nor loop-type,
	exactly one of $b_{[\rho]}$ and $b_{[\rho]^*}$ is in $Q_\ac^0$.
\end{remark}

Let $\Gamma_\ac^0$ be the differential graded subalgebra of $\Gamma_\ac$
given by $\k Q_\ac^0$ with the induced differential $d^0_{\ac}$.
Then $\Gamma_\ac$ is the Calabi-Yau-$\XX$ completion of $\Gamma_\ac^0$ (see \cite{K8}). So $\D_{fd}(\Gamma_\ac)$ is Calabi-Yau-$\XX$, i.e., $\XX$ is the Serre functor on $\D_{fd}(\Gamma_\ac)$.

\begin{definition}
	An object $M$ in $\D_{fd}(\Gamma_\ac)$ is called \emph{$\XX$-spherical} if
	\[\qdH(M,M)=\qv^0+\qv^{\XX}.\]
\end{definition}
An $\XX$-spherical object $M$ in $\D_{fd}(\Gamma_\ac)$ induces an auto-equivalence $\phi_M\in\Aut\D_{fd}(\Gamma_\ac)$, called \emph{spherical twist} \cite{ST}, by
$$\RHom(M,X)\otimes M\to X\to \phi_M(X)\to (\RHom(M,X)\otimes M)[1]$$
for any $X\in\D_{fd}(\Gamma_\ac)$. By \cite[Lemma~2.11]{ST}, we have the formula
\begin{equation}\label{eq:st}
\phi_{\psi(M)}=\psi\circ\psi_{M}\circ\psi^{-1}
\end{equation}
for any spherical object $M$ and any $\psi\in\Aut\D_{fd}(\Gamma_\ac)$.

\begin{notations}
We denote by $S_i$ the simple $\Gamma_\ac$-module corresponding to $\wg_i\in\ac$, by $\EE_\ac$ the $\ZZ^2$-graded $\Ext$-algebra $\Ext^{\ZZ^2}(\SS,\SS)$, where $\SS=\oplus_{i=1}^n S_i$, and by $\add\EE_\ac$ the $\ZZ^2$-graded category consisting of the indecomposable direct summands of $\EE_\ac$. Note that $S_i,1\leq i\leq n$, can be regarded as the indecomposable direct summands of $\EE_\ac$.
\end{notations}

By the construction \eqref{eq:diff} of $d_\ac$,
the $A_\infty$ structure introduced in Theorem~\ref{thm:ain} is an ordinary associated multiplication.
So by Theorem~\ref{thm:ain} and Corollary~\ref{cor:kd}, we have the following result.

\begin{proposition}\label{prop:morsim}
There is a triangle equivalence
\begin{gather}\label{eq:d.e.}
    \D_{fd}(\Gamma_\ac)\simeq\per(\EE_\ac).
\end{gather}
The morphisms in the $\mathbb{Z}^2$-graded category $\add\EE_\ac$ can be described in the following way.	\begin{itemize}
    \item Each non-trivial $[\rho]\in\pas(\ac)$ with $\rho(0)\in\gamma_i$ and $\rho(1)\in\gamma_j$ induces a morphism $\pi_{[\rho]}:=\pi_{b_{[\rho]}}: S_i\to S_j$ of bi-degree
	\begin{equation}\label{eq:1}
		\deg(\pi_{[\rho]})=1-\ind^{\ZZ^2}_{\rho(1)}(\wg^0_j,\lwr)+\ind^{\ZZ^2}_{\rho(0)}(\wg^0_i,\lwr),
	\end{equation}
    where $\lwr$ is an arbitrary representative in $[\rho]$; each trivial $[\rho]\in\pas(\ac)$ with $\rho(0),\rho(1)\in\gamma_i$ induces the identity $\pi_{[\rho]}=\id_{S_i}: S_i\to S_i$ of bi-degree $0$. All of $\pi_{[\rho]}$, $[\rho]\in\pas(\ac)$ with $\rho(0)\in\gamma_i$ and $\rho(1)\in\gamma_j$, form a basis of the $\mathbb{Z}^2$-graded $\k$-vector space $\Hom_{\EE_\ac}^{\mathbb{Z}^2}(S_i,S_j)$.
	\item The compositions are given by
	\begin{equation}\label{eq:3}
        \pi_{[\rho_2]}\circ \pi_{[\rho_1]}=\begin{cases}\pi_{[\rho]} &\text{if $[\rho_1]\cdot[\rho_2]=[\rho]$}\\0&\text{otherwise.}
	\end{cases}
	\end{equation}
\end{itemize}
\end{proposition}

\begin{notations}
	For any representative $\lwr\in[\rho]\in\pas(\ac)$ with $\rho(0)\in\gamma_i$ and $\rho(1)\in\gamma_j$,
	we denote by
	\begin{equation}\label{prop:liftPAS}
	\pi_{\lwr}:S_i[-\ind^{\ZZ^2}_{\rho^(0)}(\wg^0_i,\lwr)]\to S_j[-\ind^{\ZZ^2}_{\rho(1)}(\wg^{0}_j,\lwr)]
	\end{equation}
	the morphism of bi-degree 1 induced by $\pi_{[\rho]}$.	
\end{notations}

%=========================================================
\subsection{The braid twists}
%=========================================================

The \emph{mapping class group} $\MCG(\surfo)$ of a DMS $\surfo$ consists of the isotopy classes of the homeomorphisms of $\surf$ that fix $\partial\surf$ pointwise and fix $\triangle$ setwise.
Recall from \cite[Section~3.3]{QQ} that for any $\alpha\in\CA(\surfo)$,
the associated \emph{braid twist} $B_\alpha\in\MCG(\surfo)$ is defined as in Figure~\ref{fig:bt}.
We have the formula
\begin{equation}\label{eq:bt}
B_{\Psi(\alpha)}=\Psi\circ B_\alpha\circ\Psi^{-1}
\end{equation}
for any $\alpha\in\CA(\surfo)$ and any $\Psi\in\MCG(\surfo)$.

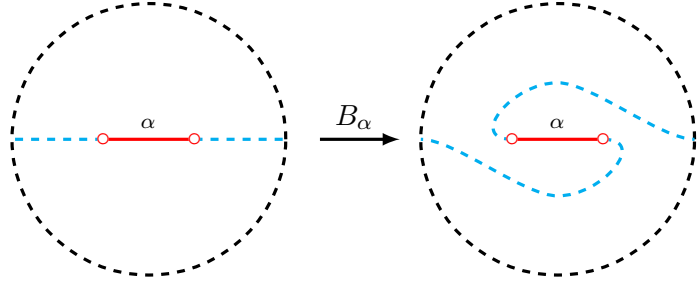
\begin{figure}[ht]\centering
	\begin{tikzpicture}[scale=.3]
	\draw[very thick,dashed](0,0)circle(6)node[above,black]{$_\alpha$};
	\draw(-2,0)edge[red, very thick](2,0)  edge[cyan,very thick, dashed](-6,0);
	\draw(2,0)edge[cyan,very thick,dashed](6,0);
	\draw(-2,0)node[white] {$\bullet$} node[red] {$\circ$};
	\draw(2,0)node[white] {$\bullet$} node[red] {$\circ$};
	\draw(0:7.5)edge[very thick,->,>=latex](0:11);\draw(0:9)node[above]{$B_{\alpha}$};
	\end{tikzpicture}\;
	%=======================================================
	\begin{tikzpicture}[scale=.3]
	\draw[very thick, dashed](0,0)circle(6)node[above,black]{$_\alpha$};
	\draw[red, very thick](-2,0)to(2,0);
	\draw[cyan,very thick, dashed](2,0).. controls +(0:2) and +(0:2) ..(0,-2.5)
	.. controls +(180:1.5) and +(0:1.5) ..(-6,0);
	\draw[cyan,very thick,dashed](-2,0).. controls +(180:2) and +(180:2) ..(0,2.5)
	.. controls +(0:1.5) and +(180:1.5) ..(6,0);
	\draw(-2,0)node[white] {$\bullet$} node[red] {$\circ$};
	\draw(2,0)node[white] {$\bullet$} node[red] {$\circ$};
	\end{tikzpicture}
	\caption{The braid twist}
	\label{fig:bt}
\end{figure}

\begin{definition}\cite{QQ}
The \emph{braid twist group} $\BT(\surfo)$ is the subgroup of the mapping class group $\MCG(\surfo)$
generated by the braid twists $B_\alpha$ for $\alpha\in\CA(\surfo)$.
\end{definition}

For any $\alpha\in\CA(\surfo)$ and any $\lws\in\wXACC(\surfo)$,
the braid twist $B_\alpha(\lws)$ of $\lws$ along $\alpha$ is the double graded curve whose underlying curve is $B_\alpha(\us)$ and which inherits the double grading from $\lws$. Here, inherit means that we can apply the braid twist on each sheet of $\log\surfo$.

\begin{definition}\label{def:dual.a.s.}
	Let $\ac=\{\wg_1,\cdots,\wg_n \}$ be a full formal arc system of $\DMS$. Let $\ws\in\wACC(\surfo)$ without self-intersections in $\surfoi$. We call $\ws$ the \emph{dual} to $\wg_i\in\ac$ with respect to $\ac$ if $\wg_i$ intersects it once with intersection index $0$ and $\wg_j, j\neq i,$ does not intersect it. Denote by $\wss_i$ the dual to $\wg_i$ and by $\dac=\{\wss_i\mid1\leq i\leq n \}$.
\end{definition}

Note that any $\wss_i$ in $\dac$ does not cross any arc in the cut $\cc$ except for the endpoints. So we have the lifts $\wss_i^m$, $m\in\ZZ$ in $\log\surfo$ of $\wss_i$, with each $\wss_i^m$ in the sheet $\surfo^m$. We denote by $(\dac)^{\ZZ}=\{\wss_i^m\mid 1\leq i\leq n, m\in\ZZ \}$ and by $\udac=\{s_1,\cdots,s_n\}$ the ungraded version of $\dac$.

\begin{definition}
	We define $\BT(\ac)$ to be the subgroup of $\MCG(\surfo)$ generated by $B_\alpha$ for $\alpha\in\udac\cap\CA(\surfo)$.
\end{definition}

The following decomposition of a curve in $\wXACC(\surfo)$ will be used frequently (cf. \cite[Lemma~3.14]{QQ}).

\begin{lemma}\label{lem:decomp}
Let $\lws\in\wXACC(\surfo)$ and $Z\in\Tri$ in an $\ac$-polygon $P$ which $\us$ crosses.
Take any line segment $l$ in $P$ from $Z$ to a point $p$ in an interior segment of $\us$ such that the interior of $l$ does not cross $\sigma$.
Let $\alpha$ be the composition of $l$ and the segment of $\us$ from $p$ to $\sigma(0)$ and $\beta$ be the composition of $l$ and the segment of $\us$ from $p$ to $\sigma(1)$, see Figure~\ref{fig:decomposition}. Let $\lwa$ and $\lwb$ be the double graded curves whose underlying graded curves are $\ua$ and $\ub$ respectively and which inherit the double gradings from $\lws$. Then we have
	\begin{itemize}
		\item[(1)] $\qqInt(\wg_i^0,\lwa)+\qqInt(\wg_i^0,\lwb)=\qqInt(\wg_i^0,\lws)$ for any $1\leq i\leq n$, and
		\item[(2)] $\lws=\lwb\wedge\lwa$.
	\end{itemize}
	In the case where $\lws\in\wXCA(\surfo)$ and $Z$ is not an endpoint of $\lws$, we have further,
	\begin{itemize}
		\item[(3)] $\Int_{\surfo}(\lwa,\lwb)=\frac{1}{2}$,
		\item[(4)] $\alpha\in\CA(\surfo)$ and $\lws=B_\alpha(\lwb)$.
	\end{itemize}
\end{lemma}

\begin{figure}[h]\centering
	\begin{tikzpicture}[scale=.7]\clip(-5,-3)rectangle(5,3);
	\foreach \j in {2,0}{ 		
		\draw[blue,very thick](90*\j+20:4)to(90*\j-20:4); \draw[dashed,blue,thin](90*\j+20:4)to[bend left=-15](90*\j-20+90:4);}
	\draw[dashed,blue,thin]
	(90*3+20:4)to[bend left=-15](90*3-20+90:4)
	(90*3+20:4)to[bend left=15](90*3-20:4);
	\draw[dashed,blue,thin](90+20:4)to[bend left=-15](90-20+90:4) (90+20:4)to[bend left=15](90-20:4);
	\draw[thick,blue](175:4.5)edge[bend left,-<-=.5,>=stealth](0,0)
	(0,0)edge[bend right,->-=.5,>=stealth](15:4.5)
	(2,-.5)node{$\beta$}(-2,.1)node[above]{$\alpha$};
	\draw[thick,red](170:4.5)to[bend left](0,1) (1.5,.7) 	node[above]{\small{$\sigma$}};
	\draw[thick,red](0,1)edge[bend right=40,->-=.3,>=stealth](20:4.5);
	\draw[Green] (0,0)to(0,1)node[above]{$p$}node{\tiny{$\bullet$}} (0,.6)node[left]{$l$};
	\draw[red](0,0)node[white]{$\bullet$} \ww;
	\draw[blue]	(0,0)node[red,below]{$Z$};
	\end{tikzpicture}\qquad
	\caption{A decompostion of a curve in $\wXACC(\surfo)$ }\label{fig:decomposition}
\end{figure}

\begin{proof}
	(1) follows from that $\lwa$ and $\lwb$ inherit the double gradings from $\lws$ and (2) follows from the definition of extension $\wedge$ in Definition~\ref{def:ext}. In the case where $\lws\in\wXCA(\surfo)$ and $Z$ is not an endpoint of $\lwe$, $Z$ is the only intersection of $\lwa$ and $\lwb$, so we have (3). (4) follows from the definition of braid twist.
\end{proof}

%=========================================================
\section{String model on log DMS}\label{sec:string}
%=========================================================
Throughout the rest of this paper, $\DMS$ is a graded DMS and $\ac=\{\wg_1,\cdots,\wg_n \}$ is a full formal arc system of $\DMS$ with its dual $\dac=\{\wss_i\mid1\leq i\leq n \}$. We still use the notations in the previous section. That is, $\Gamma_\ac$ is the dg $\XX$-graded quiver associated to $\ac$, defined by the $\ZZ^2$-graded quiver $Q_\ac$ and the differential $d_\ac$, $S_1,\cdots,S_n$ are the simple $\Gamma_\ac$ modules corresponding to $\wg_1,\cdots,\wg_n$, respectively, and $\EE_\ac$ is the $\ZZ^2$-graded algebra $\Ext^{\ZZ^2}(\SS,\SS)$, where $\SS=\oplus_{i=1}^nS_i$. Note that each $S_i$ can be regarded as an indecomposable direct summand of $\EE_\ac$. We have already shown that there is a triangle equivalence \eqref{eq:d.e.}
$$\D_{fd}(\Gamma_\ac)\simeq \per(\EE_\ac)$$

%=========================================================
\subsection{String model}
%=========================================================

In this subsection, we shall associate to each curve $\lws\in\wXACC(\surfo)$ an object $X_{\lws}$ in $\per(\EE_\ac)$, which then can be regarded as an object in $\D_{fd}(\Gamma_\ac)$.

A \emph{string} in $\EE_\ac$ is a sequence of morphisms $f_1,f_2,\cdots,f_p$ of bi-degree $1$ of shifts of indecomposable direct summands of $\EE_\ac$:
\begin{gather}\label{eq:1string}
	\xymatrix{
		\omega:\ S_{k_0}[\ii_0+\sii_0\XX]\ar@{-}[r]^{\quad f_1}&S_{k_1}[\ii_1+\sii_1\XX]\ar@{-}[r]^{\quad f_2}&\cdots\qquad
		\ar@{-}[r]^{\quad f_p}&S_{k_p}}[\ii_p+\sii_p\XX]
\end{gather}
such that
\begin{itemize}
	\item each $f_i$ is either from left to right, or from right to left;
	\item if both $f_i$ and $f_{i+1}$ point to the right, then $f_{i+1}\circ f_i=0$;
	\item if both $f_i$ and $f_{i+1}$ point to the left, then $f_i\circ f_{i+1}=0$.
\end{itemize}
The string $\omega$ gives rise to a dg $\XX$-graded $\EE_\ac$-module $X_\omega$ whose underlying $\ZZ^2$-graded $\EE_\ac$-module is
$$|X_\omega|=\bigoplus_{i=0}^pS_{k_i}[\ii_i+\sii_i\XX]$$
and whose differential is given by the $f_i, 1\leq i\leq p$. By definition, $X_\omega\in\per(\EE_\ac)$.

\begin{construction}\label{con:string}
Let $\lws\in\wXACC(\surfo)$ in a minimal position with respect to $\ac^\ZZ$. Suppose that $\lws$ intersects $\ac^\ZZ$ at $V_1,\cdots,V_{p}$ in order, which are in $\wg_0^{m_0},\wg_1^{m_1}\cdots,\wg_p^{m_p}$, respectively, and divide $\lws$ into segments $\lws_{-1,0},\lws_{0,1},\cdots,\lws_{p,p+1}$. Since the underlying $\us_{0,1},\cdots,\us_{p-1,p}$ of $\lws_{0,1},\cdots,\lws_{p-1,p}$ in $\surfo$ are (positive or negative) arc segments (which are called the arc segments of $\us$), we can define
\begin{gather}\label{eq:string}
	\omega(\lws):\xymatrix@C=3pc{
			S_{k_0}[\chi_0]\ar@{-}[r]^{f_1}&
            S_{k_1}[\chi_1]\ar@{-}[r]^{\quad f_2}&\cdots\ar@{-}[r]^{f_p}&
            S_{k_p}}[\chi_p],
\end{gather}
where
$$f_i=\begin{cases}
\pi_{\lws_{i-1,i}}&\text{if $\us_{i-1,i}$ is positive,}\\
\pi_{\overline{\lws_{i-1,i}}}&\text{if $\us_{i-1,i}$ is negative,}
\end{cases}$$
$1\leq i\leq p$, and $\chi_j=-\ind_{V_j}^{\ZZ^2}(\wg_{k_j}^0,\lws)$, $0\leq j\leq p$, see Figure~\ref{fig:wsm}, where $V_{-1}=\us(0)$ and $V_{p+1}=\us(1)$.
\end{construction}

\begin{figure}[htpb]\centering
	\begin{tikzpicture}[scale=1.6]
	\draw[red,thick,->-=.6,>=stealth](0,0)node[blue!50,left]{$\lws\colon\quad V_{-1}$} to
	(7,0)node[blue!50,right]{$V_{p+1}$};
	\draw[blue, thick](1,1)to(1,-1)node[below]{$\wg_{k_0}^{m_0}$} (1,0)\nn node[blue!50,above right]{$V_0$};
	\draw[blue, thick](2,1)to(2,-1)node[below]{$\wg_{k_1}^{m_1}$} (2,0)\nn node[blue!50,above right]{$V_1$};
	\draw[blue, thick](6,1)to(6,-1)node[below]{$\wg_{k_p}^{m_p}$} (6,0)\nn node[blue!50,above right]{$V_p$}
	(4,.5)node{$\cdots$}(4,-.5)node{$\cdots$};
	\draw[Green,thick,->-=.6,>=stealth](1,-.3)to[bend left=45] (1-.3,0);
	\draw[Green](1-.4,-.4)node{\small{$-\chi_0$}};
	\draw[Green,thick,->-=.6,>=stealth](2,-.3)to[bend left=45] (2-.3,0);
	\draw[Green](2-.4,-.4)node{\small{$-\chi_1$}};
	\draw[Green,thick,->-=.6,>=stealth](6,-.3)to[bend left=45] (6-.3,0);
	\draw[Green](6-.4,-.4)node{\small{$-\chi_p$}};
	\draw[white](0,0)\nn(7,0)\nn;
	\draw[red](0,0)\ww(7,0)\ww;
	\end{tikzpicture}
	\caption{Segments of $\lws$, cut out by $\ac^\ZZ$}\label{fig:wsm}
\end{figure}

\begin{lemma}
	$\omega(\lws)$ is a string.
\end{lemma}

\begin{proof}
	This follows from that if both $f_i$ and $f_{i+1}$ points to the right (resp. left), then the corresponding positive arc segments $\sigma_{i-1,i}$ and $\sigma_{i,i+1}$ (resp. $\overline{\sigma_{i,i+1}}$ and $\overline{\sigma_{i-1,i}}$) have no composition, which implies $f_{i+1}\circ f_i=0$ (resp. $f_i\circ f_{i+1}=0$) by \eqref{eq:3}.
\end{proof}

\begin{notations}
	For any $\lws\in\wXACC(\surfo)$, we denote by $X_{\lws}=X_{\omega(\lws)}$.
\end{notations}
By construction, we have $X_{\lws}=X_{\overline{\lws}}$ and $X_{\lws[\ii+\sii\XX]}=X_{\lws}[\ii+\sii\XX]$, for any $\lws\in\wXACC(\surfo)$ and $\ii,\sii\in\ZZ$.

We have the following result as a graded version of \cite[Lemma~5.7 and Corolary~5.8]{QQ}. Here $\Gamma^i_\ac=e_i\Gamma_\ac$ with $e_i$ the trivial path at $i$ in $Q_\ac$.

\begin{lemma}\label{lem:5.7}
	For any $\lws\in\wXACC(\surfo)$ and $\wg_i\in\ac$, we have
\begin{gather}
  \qdH(\Gamma_{\ac}^i,X_{\lws})=
    \qqInt(\wg_i^0,\lws).
\end{gather}
\end{lemma}
\begin{proof}
	Let $\omega(\lws)$ be the one in \eqref{eq:string}. Since $\qdH(\Gamma_\ac^i,S_{k})=\delta_{i,k}$, we have $$\Hom(\Gamma_\ac^i,f_j)=0$$ for any $1\leq i\leq n$ and $1\leq j\leq p$, as each $f_j$ corresponds to a non-trivial arc segment. Hence we have
	$$\begin{array}{cl}
	&\sum\limits_{\ii,\sii\in\mathbb{Z}}
	\qv^{\ii+\sii\XX}\cdot\dim\Hom_{\D(\Gamma_\ac)}(\Gamma_{\ac}^i,X_{\lws}[\ii+\sii\XX])\\
	=&\sum\limits_{\ii,\sii\in\mathbb{Z}}
	\qv^{\ii+\sii\XX}\cdot\left(\sum\limits_{j=0}^p\dim\Hom_{\D(\Gamma_\ac)}(\Gamma_{\ac}^i,S_{k_j}[-\ind_{V_j}^{\ZZ^2}(\wg_{k_j}^0,\lws)+(\ii+\sii\XX)])\right)\\
	=&\sum\limits_{\ii,\sii\in\mathbb{Z}}
	\qv^{\ii+\sii\XX}\cdot\Int^{\ii+\sii\XX}_{\surfoi}(\wg^0,\lws),
	\end{array}$$
	which gives the required formula.
\end{proof}

By Construction~\ref{con:string}, we have $X_{\wss_i^0}=S_i$ for any $1\leq i\leq n$. Conversely, we have the following.

\begin{lemma}\label{lem:5.8}
	Let $\lws\in\wXACC(\surfo)$. If $X_{\lws}=S_j$ for some $1\leq j\leq n$, then $\lws=\wss_j^0$.
\end{lemma}

\begin{proof}
	By Lemma~\ref{lem:5.7}, we have
	$$\begin{array}{rcl}
	\qqInt(\wg_i,\lws)&=&\qdH(\Gamma_{\ac}^i,X_{\lws})\\
	&=&\qdH(\Gamma_{\ac}^i,S_j)\\
	&=&\delta_{i,j}
	\end{array}$$
	So $\wg_j$ crosses $\ws$ once and the intersection index is $0$, while any other $\wg_i, i\neq j$, does not cross $\ws$. This implies that $\lws=\wss_j^m$ for some $m\in\ZZ$. Moreover, since the $\XX$-index of the intersection between $\wg_j$ and $\lws$ is zero, we have $m=0$.
\end{proof}

\begin{definition}\label{def:reachable}
	For any $\lws\in\wXCA(\surfo)$, if $X_{\lws}$ is $\XX$-spherical, we call it \emph{reachable $\XX$-spherical}. Denote by $\Sph^{\ZZ^2}(\Gamma_\ac)$ the set of reachable $\XX$-spherical objects.
	Let $$\Sph(\Gamma_\ac)=\Sph^{\ZZ^2}(\Gamma_\ac)/\<[1],[\XX]\>.$$
	Define $\ST(\Gamma_\ac)$ to be the subgroup of $\Aut\D_{fd}(\Gamma_\ac)$
	generated by $\phi_S$ for any $S \in\Sph^{\ZZ^2}(\Gamma_\ac)$.
\end{definition}

As in \cite{QQ}, in this paper we also will consider $\ST(\Gamma_\ac)$ as a subgroup of $\Aut^{o}\D_{fd}(\Gamma_0)$, which is the quotient of $\Aut\D_{fd}(\Gamma_\ac)$ by the subgroup consisting of the auto-equivalences that fix every object in $\Sph^{\ZZ^2}(\Gamma_\ac)$.

\begin{notations}
	Note that for any reachable $\XX$-spherical object $X_{\lws}$, the induced twist functor $\phi_{X_{\lws}}$ only depend on $\us\in\CA(\surf)$. Thus we will write $\phi_{X_{\us}}$ instead.
\end{notations}

%Let $Z_0^{\ST}=\ST(\Gamma_\ac)\cap\<[1],\XX\>$ and
%$$\ST_\ast(\Gamma_\ac)=\ST(\Gamma_\ac)/Z_0^{\ST}\subset\Aut^{o}\D_{fd}(\Gamma_0)/\<[1],\XX\>,$$
%where $\Aut^{o}\D_{fd}(\Gamma_0)/\<[1],\XX\>$ denotes the quotient of the auto-equivalence group $\Aut\D_{fd}(\Gamma_0)/\<[1],\XX\>$ of the orbit category $\D_{fd}(\Gamma_\ac)/\<[1],\XX\>$ by the subgroup consisting of the auto-equivalences that fix the spherical objects.

\subsection{Morphisms induced by oriented angles at decorations}

Now we take another $\lwt\in\wXACC(\surfo)$ whose underlying arc $\ut$ is in a minimal position with respect to $\ac$ and $\us$. Note that $\ut$ may coincide with $\us$. By Construction~\ref{con:string}, there is a string associated to $\lwt$:
\begin{gather}\label{eq:string2}
\omega(\lwt):\xymatrix@C=3pc{
	S_{l_0}[\chi_0']\ar@{-}[r]^{g_1}&
	S_{l_1}[\chi_1']\ar@{-}[r]^{\quad g_2}&\cdots\ar@{-}[r]^{g_q}&
	S_{l_q}}[\chi_q'],
\end{gather}
where
$$g_i=\begin{cases}
\pi_{\lwt_{i-1,i}}&\text{if $\ut_{i-1,i}$ is positive,}\\
\pi_{\overline{\lwt_{i-1,i}}}&\text{if $\ut_{i-1,i}$ is negative,}
\end{cases}$$
and $\chi_j'=-\ind_{W_j}^{\ZZ^2}(\wg_{l_j},\lwt)$, see Figure~\ref{fig:wtm}. Then we have an object $X_{\lwt}:=X_{\omega(\lwt)}\in\D_{fd}(\Gamma_\ac)$.
\begin{figure}[h]\centering
	\begin{tikzpicture}[scale=1.6]
	\draw[red,thick,->-=.6,>=stealth](0,0)node[blue!50,left]{$\lwt\colon\quad W_{-1}$} to
	(7,0)node[blue!50,right]{$W_{q+1}$};
	\draw[blue, thick](1,1)to(1,-1)node[below]{$\wg_{l_0}^{m_0'}$} (1,0)\nn node[blue!50,above right]{$W_0$};
	\draw[blue, thick](2,1)to(2,-1)node[below]{$\wg_{l_1}^{m_1'}$} (2,0)\nn node[blue!50,above right]{$W_1$};
	\draw[blue, thick](6,1)to(6,-1)node[below]{$\wg_{l_q}^{m_q'}$} (6,0)\nn node[blue!50,above right]{$W_q$}
	(4,.5)node{$\cdots$}(4,-.5)node{$\cdots$};
	\draw[Green,thick,->-=.6,>=stealth](1,-.3)to[bend left=45] (1-.3,0);
	\draw[Green](1-.49,-.49)node{\small{$-\chi_0'$}};
	\draw[Green,thick,->-=.6,>=stealth](2,-.3)to[bend left=45] (2-.3,0);
	\draw[Green](2-.49,-.49)node{\small{$-\chi_1'$}};
	\draw[Green,thick,->-=.6,>=stealth](6,-.3)to[bend left=45] (6-.3,0);
	\draw[Green](6-.49,-.49)node{\small{$-\chi_q'$}};
	\draw[white](0,0)\nn(7,0)\nn;
	\draw[red](0,0)\ww(7,0)\ww;
	\end{tikzpicture}
	\caption{Segments of $\lwt$, cut out by $\ac^\ZZ$}\label{fig:wtm}
\end{figure}
%
%suppose that $\lwt$ intersects $\ac^\ZZ$ at $W_1,\cdots,W_q$ in order, which are in $\wg_0^{m'_0},\wg_1^{m'_1}\cdots,\wg_q^{m'_q}$, respectively, and divide $\lwt$ into segments $\lwt_{-1,0},\lwt_{0,1},\cdots,\lwt_{q,q+1}$.
%
%
%Suppose that $\wt$ intersects $\ac$ at $W_0,W_1,\cdots, W_q$ in order, where $W_j$ is in $\wg_{l_j}\in\ac$.
%Denote by $W_{-1}$ and $W_{q+1}$ its starting and ending points, respectively,
%and $\wt_{i,j}$ the segment of $\wt$ from $W_i$ to $W_j$ for any $-1\le i<j\le q+1$. Denote by $\lwt_{-1,0},\lwt_{0,1},\cdots,\lwt_{q,q+1}$ the segments of $\lwt$ divided by $\ac^\ZZ$ in order. Then

%We denote by $\us_{i,j}$ the segment of $\us$ from $V_i$ to $V_j$ for any $-1\le i<j\le p+1$.
%Then $\ac$ divides $\us$ into arc segments $\us_{i-1,i}$, $0\leq i\leq p+1$. We call $\us_{i-1,i}$ an \emph{interior} arc segment, if $1\leq i\leq p$.

\begin{construction}\label{cons:mor}
Assume that $\us(0)=\ut(0)$.
There is an angle $\theta(\us,\ut)$ from $\us$ to $\ut$ clockwise at the decoration $\us(0)=\ut(0)$. We construct a morphism $$\varphi(\lws,\lwt)\in\mathcal{H}om_{\EE_\ac}(X_{\lws},X_{\lwt}[\nu(\lws,\lwt)])$$ of bi-degree $0$, for
\begin{equation}\label{eq:nu}
\nu(\lws,\lwt)=
\deg(\pi_{[\tau(-1,0)\wedge\sigma(-1,0)]})+\chi_0-\chi_0',
\end{equation}
(see \eqref{eq:1} for the notation $\deg(\pi_{[\rho]})$) induced by $\theta(\us,\ut)$ via the sequence of morphisms
$$\{\varphi_{s}:S_{k_s}[\chi_s]\to S_{l_s}[\chi_s'+\nu(\lws,\lwt)]\}_{s\geq 0},$$
where $$\varphi_{s}=\varphi(\lws,\lwt)_{s}=\begin{cases}
\pi_{[\tau(-1,s)\wedge\sigma(-1,s)]}[\chi_s]& \text{if $\tau_{-1,s+1}\wedge\sigma_{-1,s+1}$ is a positive arc segment,}\\ 0 & \text{otherwise.}
\end{cases}$$
Here $\us(-1,s)$ (resp. $\ut(-1,s)$) denotes the segment of $\us$ (resp. $\ut$) from $V_{-1}$ to $V_{s}$ (resp. from $W_{-1}$ to $W_{s}$).
\end{construction}

We now give explicit descriptions of $\varphi(\lws,\lwt)$ in different cases. By construction, if $\sigma_{-1,0}$ is not isotopic to $\tau_{-1,0}$ (i.e. $\sigma$ and $\tau$ separate at the beginning), the only non-zero component in $\varphi(\lws,\lwt)$ is $\varphi_0=\pi_{[\tau(-1,0)\wedge\sigma(-1,0)]}[\chi_0]$. See Figure~\ref{fig:cases0}, where the green arc segment $\rho=\tau_{-1,0}\wedge\sigma_{-1,0}$ is positive.
\begin{figure}[htpb]\centering
	\begin{tikzpicture}[scale=.7]\clip(-5,-3)rectangle(5,3);
	\foreach \j in {2,0}{ 		
    \draw[blue,very thick](90*\j+20:4)to(90*\j-20:4); \draw[dashed,blue,thin](90*\j+20:4)to[bend left=-15](90*\j-20+90:4);}
	\draw[dashed,blue,thin]
        (90*3+20:4)to[bend left=-15](90*3-20+90:4)
        (90*3+20:4)to[bend left=15](90*3-20:4);
	\draw[dashed,blue,thin](90+20:4)to[bend left=-15](90-20+90:4) (90+20:4)to[bend left=15](90-20:4);
	\draw[very thick, orange,->-=.5,>=stealth](0,0)to[bend left=-5](5:4.5)node[above]{$\sigma$};
	\draw[very thick, red,->-=.5,>=stealth](0,0)to[bend left=5](175:4.5)node[above]{$\tau$};
	\draw[very thick, Emerald,->-=.5,>=stealth](-5:3.7)to[bend right=-20]node[below]{$\rho$}(185:3.7);
	\draw[red](0,0)node[white]{$\bullet$} \ww;
    \draw[blue](0:4)node{\scriptsize{$\quad V_0$}}(180:4)node{\scriptsize{$W_0\quad$}}
    (3.75,.25)\nn(-3.75,.25)\nn (0,0)node[red,above]{$p$}
    (-22.5:4)node[below]{$\quad \gamma_{k_0}$}(180+22.5:4)node[below]{$\gamma_{l_0} \quad$};
	\end{tikzpicture}\qquad
	\caption{Cases of $\sigma_{-1,0}\nsim\tau_{-1,0}$}\label{fig:cases0}
\end{figure}

If $\sigma_{-1,s}$ is isotopic to $\tau_{-1,s}$ for some $s\geq 0$ then $\nu(\lws,\lwt)=\chi_0-\chi_0'$
and $\varphi_0,\cdots,\varphi_s$
are the identities; and if in addition $\sigma_{-1,s+1}$ is not isotopic to $\tau_{-1,s+1}$ any more,
then there are the following cases, where $P$ denotes the $\ac$-polygon containing $\sigma_{s,s+1}$ and $\tau_{s,s+1}$.
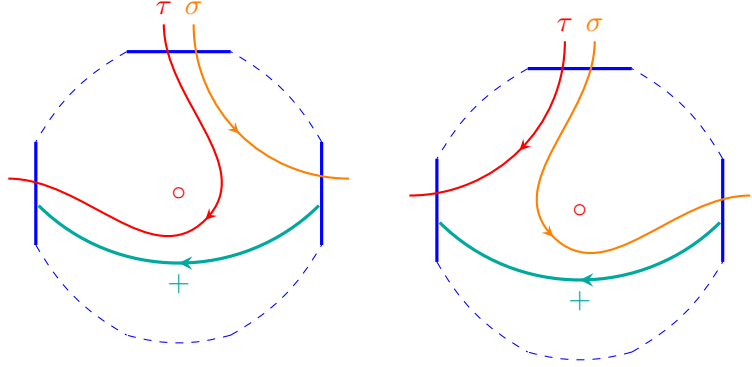
\begin{figure}[h]\centering
	\begin{tikzpicture}[scale=.5]\clip(-4.5,-4.5) rectangle (4.5,6);
	\foreach \j in {1,2,0}{
		\draw[blue,very thick](90*\j+20:4)to(90*\j-20:4);
		\draw[dashed,blue,thin](90*\j+20:4)to[bend left=-15](90*\j-20+90:4);}
	\draw[dashed,blue,thin](90*3+20:4)to[bend left=-15](90*3-20+90:4)
	(90*3+20:4)to[bend left=15](90*3-20:4);
	\draw[thick,orange,->-=.5,>=stealth](90-5:4.5)node[above]{$\sigma$}to[bend left=-45](5:4.5);
	\draw[thick,red,->-=.5,>=stealth](90+5:4.5)node[above]{$\tau$}
	.. controls +(-90:2) and +(45:2) .. (.7,-.7) .. controls +(225:2) and +(0:2) .. (175:4.5);
	
	\draw[very thick, Emerald,->-=.5,>=stealth](-5:3.7)to[bend left=45]node[below]{$+$}(185:3.7);
	\draw[red](0,0)\ww;
	\end{tikzpicture}\qquad
%    \begin{tikzpicture}[scale=.45]
%    \foreach \j in {1,2,0}{
%    \draw[blue,very thick](90*\j+20:4)to(90*\j-20:4);
%    \draw[dashed,blue,thin](90*\j+20:4)to[bend left=-15](90*\j-20+90:4);}
%    \draw[dashed,blue,thin](90*3+20:4)to[bend left=-15](90*3-20+90:4)
%        (90*3+20:4)to[bend left=15](90*3-20:4);
%    \draw[very thick, red,->-=.5,>=stealth](90-5:4.5)node[above]{$\sigma$}to[bend left=-45](5:4.5);
%    \draw[very thick, orange,->-=.5,>=stealth](90+5:4.5)node[above]{$\tau$}
%        .. controls +(-90:2) and +(45:2) .. (.7,-.7) .. controls +(225:2) and +(0:2) .. (175:4.5);
%
%    \draw[very thick, Emerald,->-=.5,>=stealth](-5:3.7)to[bend left=45]node[below]{$+$}(185:3.7);
%    \draw[red](0,0)\ww;
%    \end{tikzpicture}\qquad
\begin{tikzpicture}[scale=.5]
\foreach \j in {1,2,0}{
	\draw[blue,very thick](90*\j+20:4)to(90*\j-20:4);
	\draw[dashed,blue,thin](90*\j+20:4)to[bend left=-15](90*\j-20+90:4);}
\draw[dashed,blue,thin](90*3+20:4)to[bend left=-15](90*3-20+90:4)
(90*3+20:4)to[bend left=15](90*3-20:4);
\draw[thick, red,->-=.5,>=stealth](90+5:4.5)node[above]{$\tau$}to[bend right=-45](175:4.5);
\draw[thick,orange,->-=.5,>=stealth](90-5:4.5)node[above]{$\sigma$}
.. controls +(-90:2) and +(135:2) .. (-.7,-.7) .. controls +(-45:2) and +(180:2) .. (5:4.5);

\draw[very thick, Emerald,->-=.5,>=stealth](-5:3.7)to[bend left=45]node[below]{$+$}(185:3.7);
\draw[red](0,0)\ww;
\end{tikzpicture}
%    \begin{tikzpicture}[xscale=-.45,yscale=.45]
%    \foreach \j in {1,2,0}{
%    \draw[blue,very thick](90*\j+20:4)to(90*\j-20:4);
%    \draw[dashed,blue,thin](90*\j+20:4)to[bend left=-15](90*\j-20+90:4);}
%    \draw[dashed,blue,thin](90*3+20:4)to[bend left=-15](90*3-20+90:4)
%        (90*3+20:4)to[bend left=15](90*3-20:4);
%    \draw[very thick, orange,->-=.5,>=stealth](90-5:4.5)node[above]{$\tau$}to[bend left=-45](5:4.5);
%    \draw[very thick, red,->-=.5,>=stealth](90+5:4.5)node[above]{$\sigma$}
%        .. controls +(-90:2) and +(45:2) .. (.7,-.7) .. controls +(225:2) and +(0:2) .. (175:4.5);
%
%    \draw[very thick, Emerald,-<-=.5,>=stealth](-5:3.7)to[bend left=45]node[below]{$+$}(185:3.7);
%    \draw[red](0,0)\ww;
%    \end{tikzpicture}
	
%    \begin{tikzpicture}[xscale=.45,yscale=.45]
%    \foreach \j in {1,2,0}{
%    \draw[blue,very thick](90*\j+20:4)to(90*\j-20:4);
%    \draw[dashed,blue,thin](90*\j+20:4)to[bend left=-15](90*\j-20+90:4);}
%    \draw[dashed,blue,thin](90*3+20:4)to[bend left=-15](90*3-20+90:4)
%        (90*3+20:4)to[bend left=15](90*3-20:4);
%    \draw[thick, red,->-=.5,>=stealth](90-5:4.5)node[above]{$\sigma$}to[bend left=-45](5:4.5);
%    \draw[very thick, orange,->-=.5,>=stealth](90+5:4.5)node[above]{$\tau$}to[bend left=45](175:4.5);
%
%    \draw[Emerald,->-=.5,>=stealth](-5:3.7)to[bend left=-30]node[above]{$-$}(185:3.7);
%    \draw[red](0,0)\ww;
%    \end{tikzpicture}\qquad

    \caption{Cases of $\tau_{-1,s+1}\wedge\sigma_{-1,s+1}$ as positive arc segments}\label{fig:cases}
\end{figure}
\begin{itemize}
	\item The decoration in $P$ is on the same hand side of both $\sigma_{s,s+1}$ and $\tau_{s,s+1}$.
    See Figure~\ref{fig:cases}, where the green arc segment $\tau_{-1,s+1}\wedge\sigma_{-1,s+1}$ is positive.
%	\item The decoration in $P$ is on the left hand side of both $\sigma_{s,s+1}$ and $\tau_{s,s+1}$.
%	See the second picture of Figure~\ref{fig:cases},
%	where the green arc segment $\tau_{-1,s+1}\wedge\sigma_{-1,s+1}$ is positive.
	\item The decoration in $P$ is on the different hand sides of $\sigma_{s,s+1}$ and $\tau_{s,s+1}$. There are the following subcases.
	\begin{itemize}
		\item $\sigma_{s,s+1}$ does not intersect $\tau_{s,s+1}$, and none of $\sigma_{s,s+1}$ and $\tau_{s,s+1}$ has endpoints in the same edge of $P$. See the the first picture in the first row of Figure~\ref{fig:intersects}, where the green arc segment $\tau_{-1,s+1}\wedge\sigma_{-1,s+1}$ is negative.
		\item $\sigma_{s,s+1}$ intersects $\tau_{s,s+1}$ in $P$. Then the green curve $\tau_{-1,s+1}\wedge\sigma_{-1,s+1}$ is not an arc segment. See the second picture in the first row of Figure~\ref{fig:intersects}.
			\item $\sigma_{s,s+1}$ does not intersect $\tau_{s,s+1}$, and (at least) one of $\sigma_{s,s+1}, \tau_{s,s+1}$ satisfies
		that its endpoints are in the same edge of $P$. See the pictures in the second row of Figure~\ref{fig:intersects},
		where the green curve $\tau_{-1,s+1}\wedge\sigma_{-1,s+1}$ is not an arc segment.
	\end{itemize}
	\item At least one of $\sigma_{s,s+1}$ and $\tau_{s,s+1}$ connects to the decoration in $P$. Then the green curve $\tau_{-1,s+1}\wedge\sigma_{-1,s+1}$ is not an arc segment.
    The possible cases are shown in the figures in the third row of Figure~\ref{fig:intersects}.
\end{itemize}
So in each case in Figure~\ref{fig:cases}, we have $\varphi_s=\pi_{[\tau(-1,s)\wedge\sigma(-1,s)]}[\chi_s]$, while in each case in Figure~\ref{fig:intersects}, we have $\varphi_s=0$.
\begin{figure}[hb]\centering
	\begin{tikzpicture}[xscale=.45,yscale=.45]
	\foreach \j in {1,2,0}{
		\draw[blue,very thick](90*\j+20:4)to(90*\j-20:4);
		\draw[dashed,blue,thin](90*\j+20:4)to[bend left=-15](90*\j-20+90:4);}
	\draw[dashed,blue,thin](90*3+20:4)to[bend left=-15](90*3-20+90:4)
	(90*3+20:4)to[bend left=15](90*3-20:4);
	\draw[thick, orange,->-=.5,>=stealth](90-5:4.5)node[above]{$\sigma$}to[bend left=-45](5:4.5);
	\draw[thick, red,->-=.5,>=stealth](90+5:4.5)node[above]{$\tau$}to[bend left=45](175:4.5);
	
	\draw[Emerald,->-=.5,>=stealth](-5:3.7)to[bend left=-30]node[above]{$-$}(185:3.7);
	\draw[red](0,0)\ww;
	\end{tikzpicture}\qquad
	\begin{tikzpicture}[scale=.45]\clip(-4.5,-4) rectangle (4.5,5.5);
	\foreach \j in {1,2,0}{
		\draw[blue,very thick](90*\j+20:4)to(90*\j-20:4);
		\draw[dashed,blue,thin](90*\j+20:4)to[bend left=-15](90*\j-20+90:4);}
	\draw[dashed,blue,thin](90*3+20:4)to[bend left=-15](90*3-20+90:4)
	(90*3+20:4)to[bend left=15](90*3-20:4);
	\draw[thick,orange!0!red,->-=.3,>=stealth](90+5:4.5)node[above]{$\tau$} .. controls +(-90:2) and +(135:2) .. (-.7,-.7) .. controls +(-45:2) and +(180:2) .. (5:4.5);
	\draw[thick,orange!100!red,->-=.3,>=stealth](90-5:4.5)node[above]{$\sigma$} .. controls +(-90:2) and +(45:2) .. (.7,-.7) .. controls +(225:2) and +(0:2) .. (175:4.5);
	\draw[thick, Emerald](-5:3.7) .. controls +(210:6.5) and +(180:2) .. (0,.7) .. controls +(0:2) and +(-30:6.5) .. (185:3.7);
	\draw[red](0,0)\ww;
	\end{tikzpicture}\qquad
	
    \begin{tikzpicture}[scale=.45]\clip(-4.5,-4) rectangle (4.5,5.5);
    \foreach \j in {1,0}{
    \draw[blue,very thick](90*\j+20:4)to(90*\j-20:4);
    \draw[dashed,blue,thin](90*\j+20:4)to[bend left=-15](90*\j-20+90:4);}
    \foreach \j in {2,3}{
    \draw[dashed,blue,thin](90*\j+20:4)to[bend left=-15](90*\j-20+90:4)
        (90*\j+20:4)to[bend left=15](90*\j-20:4);}
    \draw[thick,orange!0!red,->-=.3,>=stealth](90+9:4.5)node[above]{$\tau$}
        .. controls +(-90:2) and +(165:3) .. (0,-1) .. controls +(-15:3) and +(-90:1) .. (90-3:3.8);
    \draw[thick,orange!100!red,->-=.3,>=stealth](90-9:4.5)node[above]{$\sigma$}to[bend right=30](5:4.5);
    \draw[thick, Emerald](-5:3.7)
        .. controls +(155:7) and +(165:1) .. (0,-.7) .. controls +(-15:2) and +(-90:1) .. (90+3:3.7);
    \draw[red](0,0)\ww;
    \end{tikzpicture}\qquad
    \begin{tikzpicture}[xscale=-.45,yscale=.45]\clip(-4.5,-4) rectangle (4.5,5.5);
    \foreach \j in {1,0}{
    	\draw[blue,very thick](90*\j+20:4)to(90*\j-20:4);
    	\draw[dashed,blue,thin](90*\j+20:4)to[bend left=-15](90*\j-20+90:4);}
    \foreach \j in {2,3}{
    	\draw[dashed,blue,thin](90*\j+20:4)to[bend left=-15](90*\j-20+90:4)
    	(90*\j+20:4)to[bend left=15](90*\j-20:4);}
    \draw[thick,orange!100!red,->-=.3,>=stealth](90+9:4.5)node[above]{$\sigma$}
    .. controls +(-90:2) and +(165:3) .. (0,-1) .. controls +(-15:3) and +(-90:1) .. (90-3:3.8);
    \draw[thick,orange!0!red,->-=.3,>=stealth](90-9:4.5)node[above]{$\tau$}to[bend right=30](5:4.5);
    \draw[thick, Emerald](-5:3.7)
    .. controls +(155:7) and +(165:1) .. (0,-.7) .. controls +(-15:2) and +(-90:1) .. (90+3:3.7);
    \draw[red](0,0)\ww;
    \end{tikzpicture}

    \begin{tikzpicture}[xscale=.45,yscale=.45]\clip(-4.5,-4) rectangle (4.5,5.5);
    \foreach \j in {1,2,0}{
    \draw[blue,very thick](90*\j+20:4)to(90*\j-20:4);
    \draw[dashed,blue,thin](90*\j+20:4)to[bend left=-15](90*\j-20+90:4);}
    \draw[dashed,blue,thin](90*3+20:4)to[bend left=-15](90*3-20+90:4)
        (90*3+20:4)to[bend left=15](90*3-20:4);
    \draw[thick, orange!100!red,->-=.5,>=stealth](90-5:4.5)node[above]{$\sigma$}to[bend left=0](0,0);
    \draw[thick, orange!0!red,->-=.5,>=stealth](90+5:4.5)node[above]{$\tau$}to[bend left=45](175:4.5);

    \draw[Emerald](0,0)to[bend left=0]node[above]{}(185:3.7);
    \draw[white](0,0)\nn;\draw[red](0,0)\ww;
    \end{tikzpicture}\qquad
    \begin{tikzpicture}[xscale=-.45,yscale=.45]\clip(-4.5,-4) rectangle (4.5,5.5);
    \foreach \j in {1,2,0}{
    \draw[blue,very thick](90*\j+20:4)to(90*\j-20:4);
    \draw[dashed,blue,thin](90*\j+20:4)to[bend left=-15](90*\j-20+90:4);}
    \draw[dashed,blue,thin](90*3+20:4)to[bend left=-15](90*3-20+90:4)
        (90*3+20:4)to[bend left=15](90*3-20:4);
    \draw[thick, orange!0!red,->-=.5,>=stealth](90-5:4.5)node[above]{$\tau$}to[bend left=0](0,0);
    \draw[thick, orange!100!red,->-=.5,>=stealth](90+5:4.5)node[above]{$\sigma$}to[bend left=45](175:4.5);

    \draw[Emerald](0,0)to[bend left=0]node[above]{}(185:3.7);
    \draw[white](0,0)\nn;\draw[red](0,0)\ww;
    \end{tikzpicture}\qquad
    \begin{tikzpicture}[xscale=-.45,yscale=.45]\clip(-4.5,-4) rectangle (4.5,5.5);
    \foreach \j in {1,2,0}{
    	\draw[blue,very thick](90*\j+20:4)to(90*\j-20:4);
    	\draw[dashed,blue,thin](90*\j+20:4)to[bend left=-15](90*\j-20+90:4);}
    \draw[dashed,blue,thin](90*3+20:4)to[bend left=-15](90*3-20+90:4)
    (90*3+20:4)to[bend left=15](90*3-20:4);
    \draw[thick, orange!0!red,->-=.5,>=stealth](90-5:4.5)node[above]{$\tau$}to[bend left=0](0,0);
    \draw[thick, orange!100!red,->-=.5,>=stealth](90+5:4.5)node[above]{$\sigma$}to[bend right=0](0,0);
    \draw[white](0,0)\nn;\draw[red](0,0)\ww;
    \end{tikzpicture}
\caption{Cases of $\tau_{-1,s+1}\wedge\sigma_{-1,s+1}$ that are negative or not arc segments}\label{fig:intersects}
\end{figure}
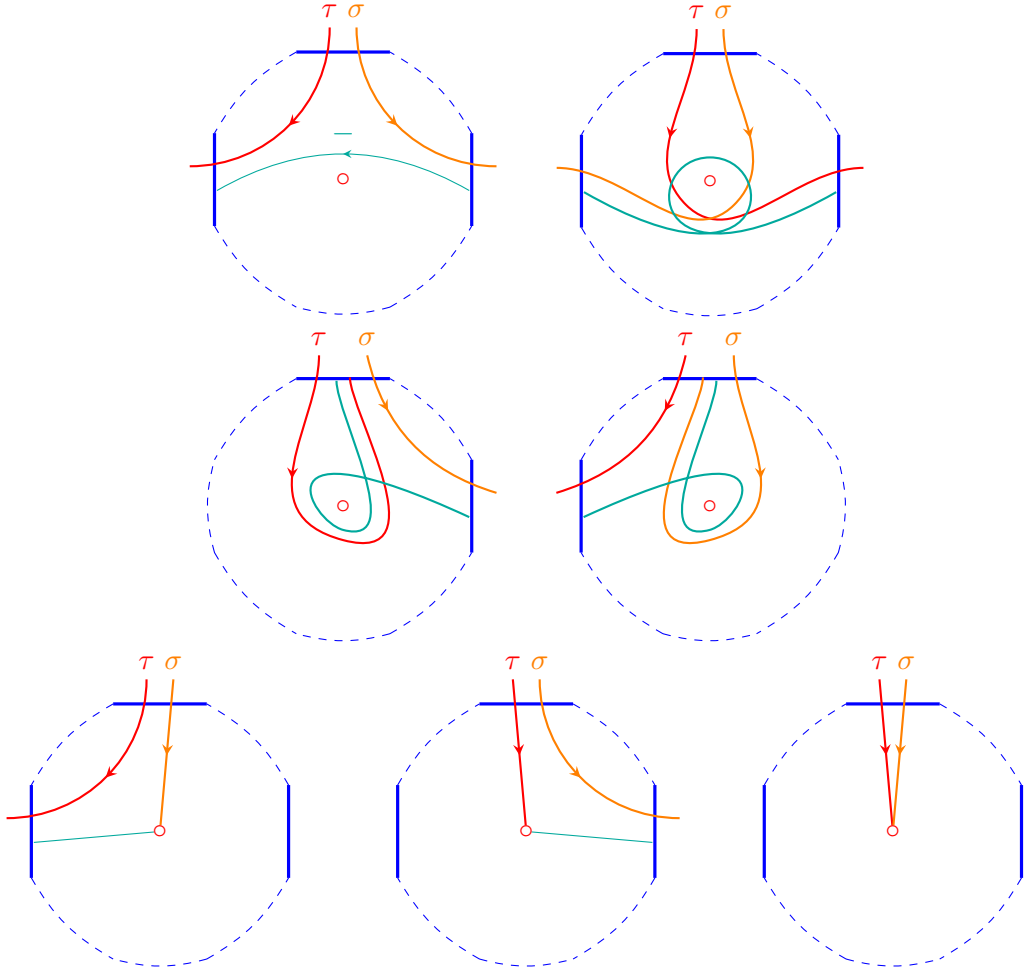

\begin{lemma}
	$\varphi(\lws,\lwt)\in Z^0\mathcal{H}om_{\EE_\ac}(X_{\lws},X_{\lwt}[\nu(\lws,\lwt)])_0$.
\end{lemma}

\begin{proof}
The proof of \cite[Lemma~A.6]{QZ2} works here.
That is, in the cases in Figure~\ref{fig:cases0} and \ref{fig:cases},
this follows from \eqref{eq:3};
in the cases in Figure~\ref{fig:intersects}, this follows from that
$f_s$ points to the right (or does not exist) while $g_s$ points to the left (or does not exist).
%The maps do not exists precisely correspond to the case when there is a non-admissible intersection
%(i.e. second picture in the first row and both pictures in the second row of Figure~\ref{fig:cases}).
\end{proof}

\begin{lemma}
	$\varphi(\lws,\lwt)$ is not zero in $\per(\EE_\ac)$.
\end{lemma}

\begin{proof}
	The proof of \cite[Lemma~A.7]{QZ2} works here. That is, for the case $\ws_{-1,0}\sim\wt_{-1,0}$ (i.e. in Figures~\ref{fig:cases} and \ref{fig:intersects}), $\varphi_0$ is the identity by definition, which does not factor through $f_1$ and $g_1$. So $\varphi(\lws,\lwt)$ is not zero. The case $\ws_{-1,0}\nsim\wt_{-1,0}$ (i.e. in Figure~\ref{fig:cases0}) is a little complicated, we refer to the proof of \cite[Lemma~A.7]{QZ2}.
\end{proof}

Thus, we have the following result.
	
\begin{proposition}\label{prop:wd-deg}
Let $\lws,\lwt\in\wXACC(\surfo)$ with the common staring point $p$. Then $\varphi(\lws,\lwt)$ is a well-defined and non-zero morphism in $\per(\EE_\ac)$ from $X_{\lws}$ to $X_{\lwt}[\ind_{p}^{\ZZ^2}(\lws,\lwt)]$.
\end{proposition}

\begin{proof}
By the above two lemmas, we only need to show $\ind_{p}^{\ZZ^2}(\lws,\lwt)=\nu(\lws,\lwt)$. Let $\rho=\ut_{-1,0}\wedge\us_{-1,0}$. Take an arbitrary representative $\lwr\in[\rho]$. Then we have
$$\begin{array}{rcl}
	\ind^{\ZZ^2}_p(\lws,\lwt)
    &\xlongequal{\eqref{eq:ksg}}&\ind^{\ZZ^2}_{V_0}(\lws,\lwr)-\ind^{\ZZ^2}_{W_0}(\lwt,\lwr) \vspace{1ex}\\
    &\xlongequal{\eqref{eq:3int}}&(\ind^{\ZZ^2}_{V_0}(\lws,\wg_{k_0}^0)+\ind^{\ZZ^2}_{V_0}(\wg_{k_0}^0,\lwr))-(\ind^{\ZZ^2}_{W_0}(\lwt,\wg_{l_0}^0)-\ind^{\ZZ^2}_{W_0}(\lwr,\wg_{l_0}^0))\vspace{1ex}\\
    &\xlongequal{\eqref{eq:hkkg}}&(1+\chi_0)-(1+\chi'_0)+(1+\ind^{\ZZ^2}_{W_0}(\wg_{l_0}^0,\lwr)-\ind^{\ZZ^2}_{V_0}(\wg_{k_0}^0,\lwr)) \vspace{1ex}\\
	&\xlongequal{\eqref{eq:1}}&\chi_0-\chi_0'+\deg(\pi_{[\rho]})\vspace{1ex}\\
	&\xlongequal{\eqref{eq:nu}}&\nu(\lws,\lwt)
\end{array}$$
\end{proof}

We have the following special case.
Recall that $\overline{?}$ denotes the inverse of a curve/arc $?$.

\begin{lemma}
	If $\lws=\lwt$, then $$\varphi(\lws,\lwt)=\varphi(\overline{\lwt},\overline{\lws}),\ \varphi(\lwt,\lws)=\varphi(\overline{\lws},\overline{\lwt}).$$
	In particular, one of the above pairs is the identity of $X_{\lws}=X_{\lwt}$.% and the other one is its Calabi-Yau dual.
\end{lemma}

\begin{proof}
	This is the last case in Figure~\ref{fig:intersects}, where $\lws$ and $\lwt$ might be switched.
\end{proof}

The above result shows that for any $\lws\in\wXACC(\surfo)$, the $2\pi$ angle from $\us$ to itself at its starting point induces a non-identity morphism $\varphi(\lws,\lws)$, which coincides with $\varphi(\overline{\lws},\overline{\lws})$, induced by the $2\pi$ angle from $\us$ to itself at its ending point.
Moreover, this morphism is in fact the Calabi-Yau dual of the identity map of $X_{\lws}$. From now on, $\varphi(?,?)$ always denote the non-identify angle-$2\pi$ morphism instead of the angle-$0$ identify.

For any $\us\in\ACC(\surfo)$, by the construction, lifts of gradings of $\us$ are in the same $\<[\XX],[1]\>$-orbit. So we denote by $X_\us$ the isoclass of the object $X_{\lws}$ in the orbit category $\D_{fd}(\Gamma_\ac)/\<[1],[\XX]\>$ for an arbitrary lift $\lws$ of a grading of $\us$. So for any $\lws,\lwt\in\wXACC(\surf)$, we have
\begin{equation}
\Hom_{\D_{fd}(\Gamma_\ac)/\<[\XX],[1]\>}(X_\sigma,X_\tau)=\Hom^{\ZZ^2}(X_{\lws},X_{\lwt})
\end{equation}
We denote by $\varphi(\us,\ut)$ the morphism in the orbit category $\D_{fd}(\Gamma_\ac)/\<[1],\XX\>$ corresponding to $\varphi(\lws,\lwt)$. Whenever we say a triangle in $\D_{fd}(\Gamma_\ac)/\<[\XX],[1]\>$, we always mean the image of a triangle in $\D_{fd}(\Gamma_\ac)$. We also simply denote
$$\Hom(X_\sigma,X_\tau):=\Hom_{\D_{fd}(\Gamma_\ac)/\<[\XX],[1]\>}(X_\sigma,X_\tau)$$
and
$$\Hom(X_{\lws},X_{\lwt}):=\Hom_{\D_{fd}(\Gamma_\ac)}(X_{\lws},X_{\lwt}).$$

\begin{proposition}\label{prop:clock}
	Let $\us_1,\us_2,\us_3\in\ACC(\surfo)$ with $\us_1(0)=\us_2(0)=\us_3(0)$.
	If the start segments of $\sigma_1,\sigma_2$ and $\sigma_3$ are in clockwise order at the starting point, see the left picture in Figure~\ref{fig:order}, then we have
%	$$\varphi(\ws_2^{m_2},\ws_3^{m_3})%[\nu(\ws_1^{m_1},\ws_2^{m_2})]
%	\circ\varphi(\ws_1^{m_1},\ws_2^{m_2})=\varphi(\ws_1^{m_1},\ws_3^{m_3})$$
	$$\varphi(\us_2,\us_3)\circ\varphi(\us_1,\us_2)=\varphi(\us_1,\us_3)$$
	in the orbit category $\D_{fd}(\Gamma_\ac)/\<[1],\XX\>$.
\end{proposition}
	
\begin{proof}
	This can be checked case by case, using the construction and the composition formula \eqref{eq:3},
which is similar to the situation in \cite{QZ2}.
\end{proof}

Let $\lws,\lwt\in\wXACC(\surfo)$ with $\us(0)=\ut(0)$. Let $\lwe=\lwt\wedge\lws$. To describe the mapping cone of $\varphi(\lws,\lwt)$, we need the following notation. When $\lwe$ is the union of two curves $\lwe_1$ and $\lwe_2$ (c.f. the pictures on the right in Figure~\ref{fig:ext.}, where $\lwe_1$ denotes the left curve in the lower picture, while $\lwe_2$ denotes the right one), we denote
$$X_{\lwe}=X_{\lwe_1}\oplus X_{\lwe_2},\
\begin{cases}
    \varphi(\overline{\lwt},\overline{\lwe})=
    \begin{pmatrix}0&\varphi(\overline{\lwt},\overline{\lwe_2})\end{pmatrix}^T,\\
    \varphi(\lwe,\overline{\lws})=\begin{pmatrix}\varphi(\lwe_1,\overline{\lws})&0\end{pmatrix}.&
\end{cases}$$
Similarly, in the orbit category $\D_{fd}(\Gamma_\ac)/\<[1],\XX\>$, we denote
$$X_{\ue}=X_{\ue_1}\oplus X_{\ue_2},\
\begin{cases}
\varphi(\overline{\ut},\overline{\ue})=
    \begin{pmatrix}0&\varphi(\overline{\ut},\overline{\ue_2})\end{pmatrix}^T,&\\
\varphi(\ue,\overline{\us})=\begin{pmatrix}\varphi(\ue_1,\overline{\us})&0\end{pmatrix}.&
\end{cases}
$$
Here $(-)^T$ denotes the transpose of a matrix.
	
\begin{proposition}\label{prop:tri1}
Let $\us,\ut\in\ACC(\surfo)$ with $\us(0)=\ut(0)$. Let $\eta=\tau\wedge\sigma$. Then we have a triangle
$$X_{\ue}\xrightarrow{\varphi(\ue,\overline{\us})}
    X_{\us}\xrightarrow{\varphi(\us,\ut)}
    X_{\ut}\xrightarrow{\varphi(\overline{\ut},\overline{\ue})}X_{\ue}
$$
in the orbit category $\D_{fd}(\Gamma_\ac)/\<[1],\XX\>$.
In particular, we have
$$\varphi(\overline{\ut},\overline{\ue})\circ\varphi(\us,\ut)=
    \varphi(\us,\ut)\circ\varphi(\ue,\overline{\us})=
    \varphi(\ue,\overline{\us})\circ\varphi(\overline{\ut},\overline{\ue})=0.$$
\end{proposition}
	
\begin{proof}
	This follows from checking the mapping cone of $\varphi(\us,\ut)$ case by case.
\end{proof}

\begin{corollary}\label{cor:realtri}
	Let $\lws, \lwt\in\wXACC(\surfo)$ with $\us(0)=\ut(0)$. Let $\lwe=\lwt\wedge\lws$. Then we have a triangle
	$$X_{\lwe}[-\nu'']\xrightarrow{\varphi(\lwe,\overline{\lws})[-\nu'']} X_{\lws}\xrightarrow{\varphi(\lws,\lwt)}X_{\lwt}[\nu]\xrightarrow{\varphi(\overline{\lwt},\overline{\lwe})[\nu]}X_{\lwe}[\nu+\nu']$$
	 in $\D_{fd}(\Gamma_\ac)$, where
	$\nu=\ind_{\sigma(0)}^{\ZZ^2}(\lws,\lwt),\ \nu'=\ind_{\tau(1)}^{\ZZ^2}(\overline{\lwt},\overline{\lwe}),\ \nu''=\ind_{\eta(0)}^{\ZZ^2}(\lwe,\overline{\lws}).$
\end{corollary}

\begin{proof}
This follows directly from Proposition~\ref{prop:tri1} by
carefully writing down the bi-degrees, using Proposition~\ref{prop:wd-deg}.
\end{proof}

%For any general closed arcs $\sigma$ and $\tau$, denote by $I(\sigma,\tau)=\{(i,j)\in [0,1]\times[0,1] \mid \sigma(i)=\tau(j)\}$ and by $I(\sigma,\tau)\cap\Tri=I(\sigma,\tau)\cap\{(i,j)\mid i,j\in\{0,1\} \}$. Then we define (nonzero) morphisms $\varphi(\sigma,\tau;p)$, $p=(i,j)\in  I(\sigma,\tau)\cap\Tri$, to be $\varphi(\sigma_{i\to 1-i},\tau_{j\to 1-j})$, where $\sigma_{0\to 1}(i)=\sigma(i)$ and $\sigma_{1\to 0}(i)=\sigma(1-i)$ for any $i\in[0,1]$.

The following two corollaries care about composition of morphisms of form $\varphi(-,-)$, where Corollary~\ref{cor:conterclock} is a generalization of \cite[Lemma~3.3]{QZ2}, with a different approach.

\begin{corollary}\label{cor:conterclock}
Let $\us_1,\us_2,\us_3\in\ACC(\surfo)$
with $\us_1(0)=\us_2(0)=\us_3(0)$.
If the start segments of $\sigma_1,\sigma_2$ and $\sigma_3$ are in counterclockwise order at the starting point, see the right picture in Figure~\ref{fig:order}, then
%	$$\varphi(\ws_2^{m_2},\ws_3^{m_3})%[\nu(\ws_1^{m_1},\ws_2^{m_2})]
%	\circ\varphi(\ws_1^{m_1},\ws_2^{m_2})=0$$
$$\varphi(\us_2,\us_3)\circ\varphi(\us_1,\us_2)=0$$
in the orbit category $\D_{fd}(\Gamma_\ac)/\<[1],\XX\>$.
\end{corollary}

\begin{proof}
%	We may assume that the angle $\theta(\us_1,\us_3)$ does not cross the starting segment of $\overline{\us_1}$ or $\overline{\us_3}$. This is because otherwise one can replace $\us_i$ by $\overline{\us_i}$, $i=1,3$, and we have
%%	$$\begin{array}{rcl}
%%	&&\varphi(\ws_2^{m_2},\ws_3^{m_3})%[\nu(\ws_1^{m_1},\ws_2^{m_2})]
%%	\circ\varphi(\ws_1^{m_1},\ws_2^{m_2})\\
%%	&=&\left(\overline{\varphi(\ws_3^{m_3}},\ws_3^{m_3})\circ\varphi(\ws_2^{m_2},\overline{\ws_3^{m_3}})\right)%[\nu(\ws_1^{m_1},\ws_2^{m_2})]
%%	\circ\left(\varphi(\ws_1^{m_1},\ws_2^{m_2})%[\nu(\ws_1^{m_1},\overline{\ws_1^{m_1}})]
%%	\circ\varphi(\ws_1^{m_1},\overline{\ws_1^{m_1}})\right).
%%	\end{array}$$
%	$$\varphi(\us_2,\us_3)\circ\varphi(\us_1,\us_2)=\left(\varphi(\overline{\us_3},\us_3)\circ\varphi(\us_2,\overline{\us_3})\right)	\circ\left(\varphi(\overline{\us_1},\us_2)	\circ\varphi(\us_1,\overline{\us_1})\right).$$
	
	By using repeatedly Proposition~\ref{prop:clock}, we have
%	$$\begin{array}{rcl}
%	&&\varphi(\ws_2^{m_2},\ws_3^{m_3})%[\nu(\ws_1^{m_1},\ws_2^{m_2})]
%	\circ\varphi(\ws_1^{m_1},\ws_2^{m_2})\\
%	&=&\varphi(\ws_1^{m_1},\ws_3^{m_3})%[\nu(\ws_1^{m_1},\ws_1^{m_1})]
%	\circ\varphi(\ws_2^{m_2},\ws_1^{m_1})%[\nu(\ws_1^{m_1},\ws_2^{m_2})]
%	\circ\varphi(\ws_1^{m_1},\ws_2^{m_2})\\
%	&=&\varphi(\ws_1^{m_1},\ws_3^{m_3})%[\nu(\ws_1^{m_1},\ws_1^{m_1})]
%	\circ\varphi(\ws_1^{m_1},\ws_1^{m_1})\\
%	&=&\varphi(\ws_1^{m_1},\ws_3^{m_3})%[\nu(\ws_1^{m_1},\ws_1^{m_1})]
%	\circ\varphi(\overline{\ws_1^{m_1}},\overline{\ws_1^{m_1}})\\
%	&=&\varphi(\ws_1^{m_1},\ws_3^{m_3})\circ\varphi(\ws_3^{m_3}\wedge\ws_{1}^{m_1},\overline{\ws_1^{m_1}})\circ\varphi(\overline{\ws_1^{m_1}},\ws_3^{m_3}\wedge\ws_{1}^{m_1})
%	\end{array}$$
	$$\begin{array}{rcl}
	\varphi(\us_2,\us_3)	\circ\varphi(\us_1,\us_2)
	&=&\varphi(\us_1,\us_3)	\circ\varphi(\us_2,\us_1)	\circ\varphi(\us_1,\us_2)\\	&=&\varphi(\us_1,\us_3)	\circ\varphi(\us_1,\us_1)\\	&=&\varphi(\us_1,\us_3)	\circ\varphi(\overline{\us_1},\overline{\us_1})\\	&=&\varphi(\us_1,\us_3)\circ\varphi(\us_3\wedge\us_{1},\overline{\us_1})\circ\varphi(\overline{\us_1},\us_3\wedge\us_{1})
	\end{array}$$
	where the last one is zero due to Proposition~\ref{prop:tri1} and $\us_3\wedge\us_1$ might be the union of two curves.
\end{proof}

\begin{corollary}\label{cor:diff}
Let $\us_1,\us_2,\us_3\in\ACC(\surfo)$
with $\sigma_1(0)=\sigma_2(0)$ and $ \sigma_2(1)=\sigma_3(1)$, see Figure~\ref{fig:order2}. Then
%	$$\varphi(\overline{\ws_2^{m_2}},\overline{\ws_3^{m_3}})\circ\varphi(\ws_1^{m_1},\ws_2^{m_2})=0.$$
	$$\varphi(\overline{\us_2},\overline{\us_3})\circ\varphi(\us_1,\us_2)=0.$$
\end{corollary}

\begin{proof}
By  Proposition~\ref{prop:tri1}, we have $\varphi(\overline{\sigma_2},\overline{\sigma_2\wedge\sigma_1})\circ\varphi(\sigma_1,\sigma_2)=0$. Using Proposition~\ref{prop:clock}, we have
%	$$\begin{array}{rcl}
%	&&\varphi(\overline{\ws_2^{m_2}},\overline{\ws_3^{m_3}})\circ\varphi(\ws_1^{m_1},\ws_2^{m_2})\\
%	&=&\varphi(\overline{\sigma_2\wedge\sigma_1},\overline{\sigma_3})\circ\varphi(\overline{\sigma_2},\overline{\sigma_2\wedge\sigma_1})\circ\varphi(\sigma_1,\sigma_2)\\
%	&=&0
%	\end{array}$$ $$
\[\varphi(\overline{\us_2},\overline{\us_3})\circ\varphi(\us_1,\us_2)=
\varphi(\overline{\sigma_2\wedge\sigma_1},\overline{\sigma_3})\circ\varphi(\overline{\sigma_2},\overline{\sigma_2\wedge\sigma_1})
\circ\varphi(\sigma_1,\sigma_2)=0,
\]
where $\us_2\wedge\us_1$ might be the union of two curves.
\end{proof}

\begin{figure}[h]\centering
    \begin{tikzpicture}[scale=1]
    \foreach \j in {2,1,3}{
    \draw[red,very thick,-<-=.5,>=stealth] (90+120*\j:3)to(0,0);
    \draw[white](90+120*\j:3)\nn(0,0)\nn;}
    \foreach \j in {2,1,3}{
    \draw[blue!50,very thick,->-=.8,>=stealth](120*\j-30:.5)to[bend left=60](120*\j-150+2:.5);
    \draw[red,very thick](90+120*\j:3)\ww(0,0)\ww;}
	\draw[red](120*1+210+15:1.5)node{\tiny{$\sigma_1$}};
	\draw[red](120*2+210+15:1.5)node{\tiny{$\sigma_3$}};
	\draw[red](120*3+210+15:1.5)node{\tiny{$\sigma_2$}};
    \end{tikzpicture}
\qquad
    \begin{tikzpicture}[scale=1]
\draw[blue!50,,very thick](-30:3)circle(.3);
\draw[white,very thick]($(-30:3)+(170:0.3)$)to[bend left=5]($(-30:3)+(165:0.3)$);
\draw[blue!50,->,>=stealth,very thick]($(-30:3)+(0.3,0)$)to[bend left=5]($(-30:3)+(0.3+.002,-.01)$);

    \draw[orange,-<-=.5,>=stealth,very thick](-150:3)to[bend left=10]node[below]{\tiny{$\sigma_3\wedge\sigma_1$}}(-30:3);
%    \draw[orange,-<-=.6,>=stealth,very thick](-30:3).. controls +(165:4) and +(-120:3.5) ..(90:3);
%    \draw[orange,very thick](-.8,.2)node[left]{\tiny{$\sigma_2\wedge\sigma_1$}};

    \foreach \j in {2,1,3}{
    \draw[red,,very thick,-<-=.5,>=stealth] (90+120*\j:3)to(0,0);
    \draw[white](90+120*\j:3)\nn(0,0)\nn;\draw[red](90+120*\j:3)\ww(0,0)\ww;
    \draw[red](120*\j+210+13:1)node{\tiny{$\sigma_\j$}};}

    \draw[blue!50,->,>=stealth,very thick](-30:.3)to[bend left=60](-150:.3);
    \draw[blue!50,very thick](240-30:.3)to[bend left=60](240-150:.3);

    \draw[blue!50,->,>=stealth,very thick](240-30:.5)to[bend left=60](240-150:.5);
    \draw[blue!50,very thick](120-30:.5)to[bend left=60](120-150:.5);

    \draw[blue!50,very thick](120-30:.5)to[bend left=60](120-150+2:.5);
    \end{tikzpicture}
\caption{Admissible closed curves intersect at a decoration}\label{fig:order}
\end{figure}
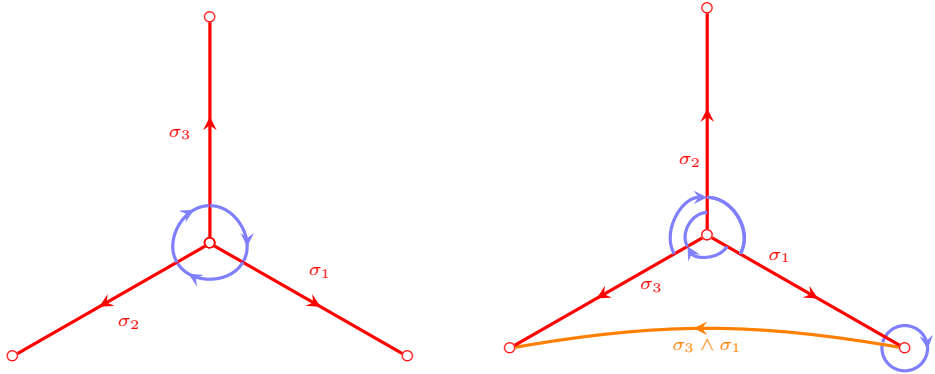

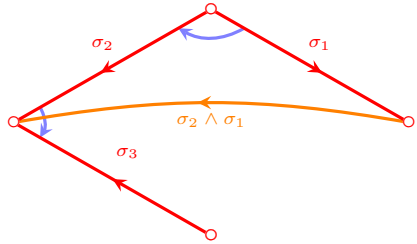
\begin{figure}[h]\centering
    \begin{tikzpicture}[scale=1]
\draw[blue!50,->,>=stealth,very thick](-30:.5)to[bend left=30](-152+2:.5);
\draw[blue!50,->,>=stealth,very thick]($(-150:3)+(30:.4)$)to[bend left=30]($(-150:3)+(-30:.4)$);
    \draw[orange,very thick,-<-=.5,>=stealth](-150:3)to[bend left=10]node[below]{\tiny{$\sigma_2\wedge\sigma_1$}}(-30:3);
    \draw[red,,very thick,->-=.5,>=stealth](0,-3)to(-150:3);
    \foreach \j in {2,1}{
    \draw[red,,very thick,-<-=.5,>=stealth] (-120*\j+90:3)to(0,0);
    \draw[white](-120*\j+90:3)\nn(0,0)\nn(0,-3)\nn;\draw[red](-120*\j+90:3)\ww(0,0)\ww(0,-3)\ww;}
    \draw[red,very thick](-120*1+90+12:1.5)node{\tiny{$\sigma_1$}}(-120*2+90-12:1.5)node{\tiny{$\sigma_2$}}
        (-120:2.2)node{\tiny{$\sigma_3$}};
    \end{tikzpicture}
    \caption{Composition of morphisms induced from different endpoints}\label{fig:order2}
\end{figure}

%=========================================================
\section{Intersection formulas}\label{sec:int=dim}
In this section,
we prove our results under some assumptions first,
and will remove such assumptions in Section~\ref{sec:general}.
\begin{assumption}\label{ass}
	We impose the following assumptions in this section:
	\begin{enumerate}
		\item For any $\ac$-polygon $P$, there is no self-folded edges, i.e. when going around its edges,
		no arc in $\ac$ will be count twice.
		\item Any two $\ac$-polygons share at most one arc in $\ac$.
	\end{enumerate}
\end{assumption}

Assumption~\ref{ass} ensures the following consequence of Lemma~\ref{lem:decomp}, a log DMS version of \cite[Lemma~3.14]{QQ}.

\begin{lemma}\label{lem:decomp2}
	For any $\lws\in\wXCA(\surfo)$ with $\Int_{\surfo}(\ac^\ZZ,\lws)>1$, i.e. $$\lws\in\wXCA(\surfo)\setminus\{\wss_i^0\mid 1\leq i\leq n \},$$
there are $\lwa,\lwb\in\wXCA(\surfo)$ such that
	$$\qqInt(\wg_i^0,\lwa)+\qqInt(\wg_i^0,\lwb)=\qqInt(\wg_i^0,\lws)$$ for any $1\leq i\leq n,$ $\Int_{\surfo}(\lwa,\lwb)=\frac{1}{2}$ and $\lws=B_\alpha(\lwb)$.
\end{lemma}

\begin{proof}
	Since $\Int_{\surfo}(\ac^\ZZ,\lws)>1$, by Assumption~\ref{ass}, there is a decoration $Z$ living in the same $\ac$-polygon $P$ with an arc segment of $\us$, such that $Z$ is not an endpoint of $\us$. Take a line segment $l$ in $P$ from $Z$ to a point in an arc segment of $\eta$ such that its interior does not cross $\eta$. Then by Lemma~\ref{lem:decomp}, we get the required $\lwa$ and $\lwb$.
\end{proof}

The first assumption in Assumption~\ref{ass} is equivalent to that any close curve in $\udac$ has different endpoints, i.e. $\udac\subset\CA(\surfo)$.

\begin{lemma}\label{lem:4.2}
	$\CA(\surfo)=\BT(\ac)\cdot\udac$ and $\BT(\surfo)=\BT(\ac)$.
\end{lemma}

\begin{proof}
	A similar proof with \cite[Lemma~4.2]{QQ} works here as follows.
	
	Since $\udac\subset\CA(\surfo)$, we have $\BT(\ac)\cdot\udac\subseteq\CA(\surfo)$. So for the first equality, we need to show  each $\eta\in\CA(\surfo)$ belongs to $\BT(\ac)\cdot\udac$. Using the induction on $l(\eta)=\Int_{\surfo}(\ac,\eta)$, when $l=1$, by definition, we have $\eta\in\udac$. Now we consider the case $l(\eta)=r$ and suppose that any closed arc with length smaller than $r$ with $r\geq 1$ is in $\BT(\ac)\cdot\udac$. Since $l(\eta)\geq 2$, by Lemma~\ref{lem:decomp2}, there are $\alpha,\beta\in\CA(\surfo)$ such that $\eta=B_\alpha(\beta)$ and $l(\eta)=l(\alpha)+l(\beta)$. So $\alpha=b(s)$ and $\beta=b'(s')$ for some $b,b'\in\BT(\ac)$ and $s,s'\in\udac$. So $$\eta=B_\alpha(\beta)\xlongequal{\eqref{eq:bt}}b\circ B_s\circ b'\circ b^{-1}(s')\in\BT(\ac)\cdot\udac$$ as required.
	
	For the second equality, by the first equality, for any $\eta\in\CA(\surfo)$, there is a $b\in\BT(\ac)$ and an $s\in\udac$ such that $\eta=b(s)$. Then we have $$B_\eta=B_{b(s)}\xlongequal{\eqref{eq:bt}}b\circ B_s\circ b^{-1}\in\BT(\ac).$$
\end{proof}

The lifted graded version of Lemma~\ref{lem:4.2} is the following.

\begin{lemma}\label{lem:CA BT}
	$\wXCA(\surfo)=\BT(\ac)\cdot(\dac)^{\ZZ}$.
\end{lemma}

\begin{proof}
	This follows directly from the action of the braid twist group on double graded admissible closed curves and Lemma~\ref{lem:4.2}.
\end{proof}

%\begin{proof}
%	For any $\wss_i^m\in(\dac)^{\ZZ}$ and any $b\in\BT(\ac)$, we have $b(\wss_i^m)$ is a lift of a grading of $b(s_i)\in\CA(\surfo)$. So we have $\BT(\ac)\cdot(\dac)^{\ZZ}\subseteq\wXCA(\surfo)$.
%	
%	Conversely, for any $\lwe\in\wXCA(\surfo)$, by Lemma~\ref{lem:4.2}, there is $b\in\BT(\ac)$ and $s\in\udac$ such that $\eta=b(s)$. Since lifts of gradings $\wss$ of $s$ are related by shifts $\ii+\sii\XX$, $\ii,\sii\in\ZZ$, so are $b(\wss)$.
%\end{proof}

%=========================================================
%\subsection{The case without interior intersections}
%=========================================================

In the case that curves do not intersect in $\surfoi$, we show that the morphisms constructed in Construction~\ref{cons:mor} form a basis of the morphism space of the corresponding objects.

\begin{proposition}\label{prop:first}
	For any $\ue_1,\ue_2\in\ACC(\surfo)$ satisfying $\Int_{\surfoi}(\ue_1,\ue_2)=0$, the morphisms
	$$\{\varphi(\us,\ut)\mid \us\in\{\ue_1,\overline{\ue_1}\}, \ut\in\{\ue_2,\overline{\ue_2}\},\us(0)=\ut(0)\}$$
	form a basis of $\Hom(X_{\eta_1},X_{\eta_2})$. In particular, we have
	$$\Int_\Tri(\eta_1,\eta_2)=\dim\Hom(X_{\eta_1},X_{\eta_2}).$$
\end{proposition}
	
\begin{proof}
	% We follow and simplify the proof of \cite[Proposition~5.9]{QQ}.
		
	Use the induction on
	$$I=\Int_{\surfo}(\ac,\eta_1)+\Int_{\surfo}(\ac,\eta_2).$$
	The starting case is $I=2$, where both $\eta_1$ and $\eta_2$ are in $\udac$. So the formula follows directly from the structure of $\EE_\ac$.
	
Now suppose that the proposition holds for any pair $(\eta_1,\eta_2)$ with $I\leq r$ for some $r\geq 2$
and consider the case when $I = r + 1$.
The arcs in $\ac$ divide $\eta_1$ and $\eta_2$ into the segments $\eta_1(-1,0), \cdots, \eta_1(p,p+1)$ and $\eta_2(-1,0),$ $\cdots,$ $\eta_2(q,q+1)$, respectively in order.
Since $r\geq 2$, we have $p+1>1$ or $q+1>1$.
Then there is a decoration $Z$ which is in the same $\ac$-polygon as an arc segment of $\eta_1$ or $\eta_2$.
Take a line segment $l$ from $Z$ to a point $p$ in an arc segment of $\eta_1$ or $\eta_2$ such that its interior does not cross any of $\eta_1$ and $\eta_2$. Without loss of generality, we assume $l$ intersects $\eta_1$. Then by Lemma~\ref{lem:decomp}, there are $\alpha,\beta\in\ACC(\surfo)$ with $\us=\ub\wedge\ua$.
Then by Proposition~\ref{prop:tri1}, there is a triangle in $\D_{fd}(\Gamma_\ac)/\<[1],[\XX]\>$:
	$$X_\alpha\xrightarrow{\varphi(\alpha,\beta)} X_\beta\xrightarrow{\varphi(\overline{\beta},\overline{\eta_1})} X_{\eta_1}\xrightarrow{\varphi(\eta_1,\overline{\alpha})} X_{\alpha}$$
	Applying $\Hom(-,X_{\eta_2})$ to the triangle, we have an exact sequence
	$$\begin{array}{cccc}
		&\Hom(X_\alpha,X_{\eta_2})&\xrightarrow{\circ\varphi(\eta_1,\overline{\alpha})}&\Hom(X_{\eta_1},X_{\eta_2})\\\xrightarrow{\circ\varphi(\overline{\beta},\overline{\eta_1})}&\Hom(X_\beta,X_{\eta_2})&
		\xrightarrow{\circ\varphi(\alpha,\beta)}&\Hom(X_\alpha,X_{\eta_2})
	\end{array}$$
	
	Since the line segment $l$ does not cross $\eta_2$ and $\Int_{\surfoi}(\eta_1,\eta_2)=0$, we have $\Int_{\surfoi}(\alpha,\eta_2)=0$ and $\Int_{\surfoi}(\beta,\eta_2)=0$.
	So using Corollary~\ref{prop:clock}, we have that
	\begin{itemize}
		\item $\eta_2(0)=\eta_1(0)$ $\Longleftrightarrow$ $\eta_2(0)=\alpha(1)$, and in this case,  $\varphi(\overline{\alpha},\eta_2)\circ\varphi(\eta_1,\overline{\alpha})=\varphi(\eta_1,\eta_2)$;
		\item $\eta_2(1)=\eta_1(0)$ $\Longleftrightarrow$ $\eta_2(1)=\alpha(1)$, and in this case,  $\varphi(\overline{\alpha},\overline{\eta_2})\circ\varphi(\eta_1,\overline{\alpha})=\varphi(\eta_1,\overline{\eta_2})$;
		\item $\eta_2(0)=\eta_1(1)$ $\Longleftrightarrow$ $\eta_2(0)=\beta(1)$, and in this case,  $\varphi(\overline{\eta_1},\eta_2)\circ\varphi(\overline{\beta},\overline{\eta_1})=\varphi(\overline{\beta},\eta_2)$;
		\item $\eta_2(1)=\eta_1(1)$ $\Longleftrightarrow$ $\eta_2(1)=\beta(1)$, and in this case,  $\varphi(\overline{\eta_1},\overline{\eta_2})\circ\varphi(\overline{\beta},\overline{\eta_1})=\varphi(\overline{\beta},\overline{\eta_2})$;
		\item $\eta_2(0)=\beta(0)$ $\Longleftrightarrow$ $\eta_2(0)=\alpha(0)$, and in this case,  $\varphi(\beta,\eta_2)\circ\varphi(\alpha,\beta)=\varphi(\alpha,\eta_2)$;
		\item $\eta_2(1)=\beta(0)$ $\Longleftrightarrow$ $\eta_2(1)=\alpha(0)$; and in this case,  $\varphi(\beta,\overline{\eta_2})\circ\varphi(\alpha,\beta)=\varphi(\alpha,\overline{\eta_2})$.
	\end{itemize}
	By inductive assumption, the proposition holds for $\alpha$ and $\beta$. Then by the above exact sequence, the above correspondences and formulas, the proposition holds for $(\eta_1,\eta_2)$.
\end{proof}

The lifted graded version of Proposition~\ref{prop:first} is the following.

\begin{proposition}\label{cor:relfirst}
	For any $\lwe_1,\lwe_2\in\wXACC(\surfo)$ satisfying $\Int_{\surfoi}(\lwe_1,\lwe_2)=0$, we have
	\[
	\qqInt(\lwe_1,\lwe_2)=\qdH(X_{\lwe_1},X_{\lwe_2})
	\]
	and each $\Hom(X_{\lwe_1},X_{\lwe_2}[{\ii+\sii\XX}])$ has a basis $$\{\varphi(\lws,\lwt)\mid \lws\in\{\lwe_1,\overline{\lwe_1}\},\lwt\in\{\lwe_2,\overline{\lwe_2}\},\lws(0)=\lwt(0),\ind^{\ZZ^2}_{\sigma(0)}(\lws,\lwt)=\ii+\sii\XX\}.$$
\end{proposition}

\begin{proof}
	This follows directly from Proposition~\ref{prop:first} and Proposition~\ref{prop:wd-deg}.
\end{proof}

Since any double graded closed arc does not intersect itself in $\surfoi$, we can describe its endomorphism space by the above result, which gives the following consequence.

\begin{corollary}
	For any $\lwe\in\wXCA(\surfo)$,  we have that $X_{\lwe}$ is (reachable) $\XX$-spherical.
\end{corollary}

\begin{proof}
	Since $\qqInt(\lwe,\lwe)=\qv^0+\qv^\XX$, by Proposition~\ref{cor:relfirst}, we have that the $\qv$-dimension $\qdH(X_{\lwe},X_{\lwe})=\qv^0+\qv^\XX$, which implies that $X_{\lwe}$ is (reachable) $\XX$-spherical.
\end{proof}

Now we start to investigate the compatibility between the action of braid twist and the action of spherical twist.

\begin{proposition}\label{cor:1}
	Let $\alpha\in\CA(\surfo)$ and $\beta\in\ACC(\surfo)$ with  $\Int_{\surfoi}(\alpha,\beta)=0$. Then
	$$X_{B_{\alpha}(\beta)}=\phi_{X_\alpha}(X_\beta)$$
\end{proposition}

\begin{proof}
	For the case $\Int_{\surfo}(\alpha,\beta)=0$, by Proposition~\ref{prop:first}, we have $\Hom(X_{\alpha},X_{\beta})=0$. So $B_\alpha(\beta)=\beta$ and $\phi_{X_\alpha}(X_\beta)=X_\beta$, which imply the required equality.
	
	For the case that $\beta\in\CA(\surfo)$, the proof of \cite[Proposition~3.1]{QZ2} works here.
	
	The last case is that $\beta(0)=\beta(1)$ and it is one endpoint of $\alpha$. Without loss of generality, we assume that $\beta(0)=\beta(1)=\alpha(0)$ and the starting segments of $\ua$, $\ub$ and $\overline{\ub}$ are in clockwise order, see Figure~\ref{fig:0011}. So we have that $\beta\wedge\alpha$ is an admissible closed curve instead of a union of two admissible closed curves. By Proposition~\ref{prop:clock}, we have $\varphi(\overline{\ub},\overline{\ub\wedge\ua})\circ\varphi(\ua,\overline{\ub})=\varphi(\ua,\overline{\ub\wedge\ua})$. So by the octahedral axiom, we have the following commutative diagram of triangles
	$$\xymatrix@R=3pc@C=3pc{
		&X(\ub\wedge\ua)\ar@{=}[r]\ar[d]&X_{\ub\wedge\ua}\ar[d]\\
		X_{\overline{\beta}\wedge\alpha}\ar[r]\ar@{=}[d]&X_{B_\alpha(\beta)}\ar[r]\ar[d]&X_{\ua}\ar[r]\ar[d]^{\varphi(\ua,\ub)}&X_{\overline{\beta}\wedge\alpha}\ar@{=}[d]\\
		X_{\overline{\beta}\wedge\alpha}\ar[r]&X_{\ua}\ar[r]^{\varphi(\ua,\overline{\ub})}\ar[d]_{\varphi(\ua,\overline{\ub\wedge\ua})}&X_{\ub}\ar[r]\ar[d]^{\varphi(\overline{\ub},\overline{\ub\wedge\ua})}&X_{\overline{\beta}\wedge\alpha}\\
		&X_{\ub\wedge\ua}\ar@{=}[r]&X_{\ub\wedge\ua}
	}$$
	Then there is a triangle
	$$X_{B_\alpha(\beta)}\to X_\alpha\oplus X_\alpha\xrightarrow{\left(\varphi(\ua,\ub),\varphi(\ua,\overline{\ub}) \right)} X_{\beta}\xrightarrow{}X_{B_\alpha(\beta)},$$
	which implies the required formula.
\begin{figure}[h]\centering
\begin{tikzpicture}[scale=.65]

\draw[very thick,orange,->-=.8,>=stealth](-90:4) .. controls +(100:1) and +(-90:3) .. ($(150:5)+(-.5,0)$)node[left]{$\ub\wedge\ua$};
\draw[very thick,orange,-<-=.8,>=stealth](0,0) .. controls +(10:3) and +(-90:2) .. ($(30:5)+(.5,0)$)node[right]{$\ub\wedge\ua$};

\draw[very thick,red,->-=.8,>=stealth](0,0) .. controls +(150:3) and +(-90:1) .. (150:5)node[right]{$\ub$};
\draw[very thick,red,-<-=.8,>=stealth](0,0) .. controls +(30:3) and +(-90:1) .. (30:5)node[left]{$\ub$};

\draw[very thick,orange,dashed]($(30:5)+(.5,0)$).. controls +(90:2.5) and +(90:2.5) ..($(150:5)+(-.5,0)$);
\draw[very thick,red,dashed](150:5).. controls +(90:2) and +(90:2) ..(30:5);
\draw[very thick,red,->-=.7,>=stealth](0,0) to node[right]{$\ua$}(-90:4);

\draw[blue,-<-=.7,>=stealth,thick](9:1.8)arc(9:27:1.8) (27:1.8);
\draw[blue,-<-=.2,>=stealth,thick](30:.5)arc(30:270:.5) (270:.5);
\draw[blue,-<-=.5,>=stealth,thick](150:.7)arc(150:270:.7) (270:.7);
\draw[blue,-<-=.4,>=stealth,thick](9:1.6)arc(9:270:1.6) (270:1.6);
\draw[Emerald]
    (20:1.8)node[right]{\footnotesize{$\varphi(\overline{\ub},\overline{\ub\wedge\ua})$}}
    (90:.5)node[above]{\footnotesize{$\varphi(\ua,\overline{\ub})$}}
    (225:.7)node[left]{\footnotesize{$\varphi(\ua,\ub)$}}
    (90:1.6)node[above]{\footnotesize{$\varphi(\ua,\overline{\ub\wedge\ua})$}};

\draw[white](0,0)\nn(-90:4)\nn;
\draw[red](0,0)\ww(-90:4)\ww;
\end{tikzpicture}
\caption{}\label{fig:0011}
\end{figure}	
\end{proof}

The lifted graded version of Proposition~\ref{cor:1} is the following.

\begin{proposition}\label{lem:rel1}
	Let $\alpha\in\CA(\surfo)$ and $\lwb\in\wXACC(\surfo)$ with $\Int_{\surfoi}(\alpha,\beta)=0$. Then
	$$\phi_{X_{\alpha}}(X_{\lwb})=X_{B_{\alpha}(\lwb)}.$$
\end{proposition}

\begin{proof}
	Since $B_{\alpha}(\lwb)$ inherits the double grading from $\lwb$, the required formula follows from Proposition~\ref{cor:1}.
\end{proof}

%\subsection{Main results}\label{sec:main}

Focusing on closed arcs, we have the following lemma, which will be generalized in Proposition~\ref{cor:6.4.2}.

\begin{lemma}\label{lem:phi.x}
	For any $s\in\udac$ and any $\lwe\in\wXCA(\surfo)$, we have
	\begin{equation}\label{eq:phi.x}
	\phi^\varepsilon_{X_s}(X_{\lwe})=X_{B^\varepsilon_s(\lwe)},\ \varepsilon=\pm1.
	\end{equation}
\end{lemma}

\begin{proof}
	Without loss of generality, we only deal the case for $\varepsilon=1$. Use the induction on $l(\eta)=\Int_{\surfo}(\ac,\eta)$. When $l(\eta)=1$, we have $\eta\in\udac$. Then $\Int_{\surfoi}(s,\eta)=0$, which implies \eqref{eq:phi.x} by Proposition~\ref{lem:rel1}.
	
	Suppose \eqref{eq:phi.x} holds for $l(\eta)<r$, for some $r\geq 2$. Consider the case $l(\eta)=r$. By Lemma~\ref{lem:decomp2}, there are $\lwa,\lwb\in\wXCA(\surfo)$ with $\Int_{\surfoi}(\lwa,\lwb)=0$ and $\lwe=B_\alpha(\lwb)$. By Proposition~\ref{lem:rel1}, we have
	\begin{equation}\label{eq:01}
	X_{\lwe}=\phi_{X_{\alpha}}(X_{\lwb}).
	\end{equation}
	Twisted by $B_s$, we have $\Int_{\surfoi}(B_s(\lwa),B_s(\lwb))=0$ and $B_s(\lwe)=B_{B_s(\alpha)}(B_s(\lwb))$. By Proposition~\ref{lem:rel1}, we have
	\begin{equation}\label{eq:02}
	X_{B_{s}(\lwe)}=\phi_{X_{B_s(\alpha)}}(X_{B_s(\lwb)}).
	\end{equation}
	By the inductive assumption, we have
	\begin{equation}\label{eq:03}
	X_{B_s(\lwa)}=\phi_{X_s}(X_{\lwa}),\ \text{and}\ X_{B_s(\lwb)}=\phi_{X_s}(X_{\lwb}).
	\end{equation}
	So
	$$\begin{array}{rcl}
	X_{B_{s}(\lwe)}&\xlongequal{\eqref{eq:02}}&\phi_{X_{B_s(\alpha)}}(X_{B_s(\lwb)}) \vspace{1ex}\\
	&\xlongequal{\eqref{eq:03}}&\phi_{\phi_{X_s}(X_{\lwa})}\left(\phi_{X_s}(X_{\lwb})\right) \vspace{1ex}\\
	&\xlongequal{\eqref{eq:st}}&(\phi_{X_s}\circ\phi_{X_{\alpha}}\circ\phi_{X_s}^{-1})\left(\phi_{X_s}(X_{\lwb})\right)\vspace{1ex}\\
	&\xlongequal{\quad}&(\phi_{X_s}\circ\phi_{X_{\alpha}})(X_{\lwb})\vspace{1ex}\\
	&\xlongequal{\eqref{eq:01}}&\phi_{X_s}(X_{\lwe})
	\end{array}$$
\end{proof}

Denote by $b_i=B_{s_i}$ for any $s_i\in\udac$ and by $\phi_i=\phi_{S_i}$ for any simple $S_i$.

\begin{proposition}\label{lem:bt->st}
There is a canonical group homomorphism
\begin{gather}\label{eq:bt->st}
    \iota:\BT(\surfo)\to \ST(\Gamma_\ac)
\end{gather}
sending the generator $b_i$ to the generator $\phi_i$. Moreover, for any $\eta\in\CA(\surfo)$, we have
\begin{gather}\label{eq:bt=st+}
\iota(B^\varepsilon_\eta)=\phi^\varepsilon_{X_\eta}.
\end{gather}
\end{proposition}

\begin{proof}
	Let $b=b^{\varepsilon_1}_{{i_1}}\circ\cdots\circ b^{\varepsilon_k}_{{i_k}}$, for some $i_j\in\{1,\cdots,n\}$, $\varepsilon_j\in\{\pm1\}$, $1\leq j\leq k$ and $k\in\mathbb{N}$. Let $\phi=\phi^{\varepsilon_1}_{i_1}\circ\cdots\circ \phi^{\varepsilon_k}_{i_k}$. If $b=1$, then $b(\wss_i^0)=\wss_i^0$, for any $1\leq i\leq n$. By (repeatedly using) Lemma~\ref{lem:phi.x}, we have
	$$S_i=X_{\wss_i^0}=X_{b(\wss_i^0)}=\phi(X_{\wss_i^0})=\phi(S_i).$$
	So $\phi=1$ in  $\ST(\Gamma_\ac)$. This implies the existence of the required group homomorphism $\iota$.
	
	To show \eqref{eq:bt=st+}, let $\eta\in\CA(\surfo)$. By Lemma~\ref{lem:4.2}, there is a $b^{\varepsilon_1}_{{i_1}}\circ\cdots\circ b^{\varepsilon_k}_{{i_k}}\in\BT(\surfo)$ and an $s\in\udac$ such that $\eta=b^{\varepsilon_1}_{{i_1}}\circ\cdots\circ b^{\varepsilon_k}_{{i_k}}(s)$.
	So 	
	\[\begin{array}{rcl}
\iota(B_\eta)&\xlongequal{\qquad}&\iota(B_{b^{\varepsilon_1}_{{i_1}}\circ\cdots\circ b^{\varepsilon_k}_{{i_k}}(s)})\\
	&\xlongequal{\eqref{eq:bt}}&\iota\left(b^{\varepsilon_1}_{{i_1}}\circ\cdots\circ b^{\varepsilon_k}_{{i_k}}\circ B_s\circ(b^{\varepsilon_1}_{{i_1}}\circ\cdots\circ b^{\varepsilon_k}_{{i_k}})^{-1}\right)\\
	&\xlongequal{Thm.~\ref{thm:st=bt}}&\phi^{\varepsilon_1}_{{i_1}}\circ\cdots\circ \phi^{\varepsilon_k}_{{i_k}}\circ \phi_{X_s}\circ(\phi^{\varepsilon_1}_{{i_1}}\circ\cdots\circ \phi^{\varepsilon_k}_{{i_k}})^{-1}\\
	&\xlongequal{\eqref{eq:st}}&\phi_{\phi^{\varepsilon_1}_{{i_1}}\circ\cdots\circ \phi^{\varepsilon_k}_{{i_k}}(X_s)}\\
	&\xlongequal{Lem.~\ref{lem:phi.x}}&\phi_{X_\eta}.
	\end{array}\]
\end{proof}

%\begin{corollary}\label{cor:bt=st+}
%For any
%\end{corollary}
%\begin{proof}
%	
%\end{proof}

The action of a braid twist on a double graded closed arc is compatible with the action of the corresponding spherical twist on the corresponding object as follows.

\begin{proposition}\label{cor:6.4.2}
	For any $b\in\BT(\surfo)$ and any $\lwe\in\wXCA(\surfo)$, we have
	$$X_{b(\lwe)}=\iota(b)(X_{\lwe}).$$
	In particular, when $b=B_\alpha$ for  $\alpha\in\CA(\surfo)$, we have
	$$X_{B_\alpha(\lwe)}=\phi_{X_\alpha}(X_{\lwe}).$$
\end{proposition}

\begin{proof}
	By Lemma~\ref{lem:4.2},  $b=b^{\varepsilon_1}_{{i_1}}\circ\cdots\circ b^{\varepsilon_k}_{{i_k}}$ for some $1\leq i_j\leq n$ and $\varepsilon_j\in\{\pm 1\}$. By using repeatedly \eqref{eq:phi.x}, we get the first formula.

	For the second formula, by Lemma~\ref{lem:4.2}, there is a $b'\in\BT(\surfo)$ and an $s\in\udac$ such that $\alpha=b'(s)$. By Proposition~\ref{lem:bt->st}, we have $\iota(B_\alpha)=\phi_{X_\alpha}$. So $$X_{B_\alpha(\lwe)}=\iota(B_\alpha)(X_{\lwe})=\phi_{X_\alpha}(X_{\lwe}).$$
\end{proof}

Now we process to prove Theorem~A (under Assumption~\ref{ass}), the Calabi-Yau-$\XX$ version of the main results in \cite{QQ, QZ2}.

\begin{theorem}\label{thm:X:}
The map $\lwe\mapsto X_{\lwe}$ gives a bijection
\begin{gather}\label{eq:X:}
    X\colon\wXCA(\surfo)\longrightarrow \Sphzz(\Gamma_\ac).
\end{gather}
\end{theorem}

\begin{proof}
	The surjectivity follows from the Definition~\ref{def:reachable}. For the injectivity, assume $\lwe,\lwe'\in\wXCA(\surfo)$ with $X_{\lwe}=X_{\lwe'}$. By Lemma~\ref{lem:CA BT}, there is $b\in\BT(\surfo)$ and $\wss^m\in(\dac)^{\ZZ}$ such that $\lwe=b(\wss^{m})$. So by Corollary~\ref{cor:6.4.2}, we have $$X_{\wss^{m}}=\iota(b)^{-1}(X_{\lwe})=\iota(b)^{-1}(X_{\lwe'})=X_{b^{-1}(\lwe')}.$$
	By Lemma~\ref{lem:5.8}, we have $\wss^{m}=b^{-1}(\lwe')$, which implies that $\lwe=\lwe'$ as required.
\end{proof}

\begin{theorem}\label{thm:st=bt}
	The group homomorphism in \eqref{eq:bt->st} is an isomorphism.
\end{theorem}

\begin{proof}
	We have that $\iota$ in \eqref{eq:bt->st} is surjective because the generators $\phi_i$, $1\leq i\leq n$, are in its image. So we only need to show that it is injective. Let $b\in\BT(\surfo)$ with $\iota(b)=1$. By Proposition~\ref{cor:6.4.2}, we have
	$$X_{b(\lwe)}=\iota(b)(X_{\lwe})=X_{\lwe}$$
	for any $\lwe\in\wXCA(\surfo)$. So by Theorem~\ref{thm:X:}, we have $b(\lwe)=\lwe$. Hence $$B_{\eta}=B_{b(\eta)}\xlongequal{\eqref{eq:bt}}b\circ B_{\eta}\circ b^{-1},$$
	which implies $b\circ B_\eta=B_\eta\circ b$, i.e. $b$ is in the center of $\BT(\surfo)$. But by \cite{QQ}, the center of $\BT(\surfo)$ is trivial when $\surf$ is not a polygon. So $b=1$ as required. In the case $\surf$ is a polygon, the isomorphism $\iota$ is shown in \cite{KS} where the groups are isomorphic to the classical braid/Artin group.
\end{proof}

To prove the $\qv$-intersection formula, we need the following lemma.

\begin{lemma}\label{lem:1}
Let $s\in\udac$ and $\eta\in\ACC(\surfo)$ such that $\eta$ is contained
in the two $\ac$-polygons that are incident with $s$.
Then $\Int_{\surfoi}(s,\eta)=0$.
\end{lemma}
\begin{proof}
Note that $s$ and $\eta$ will be represented by curves in their isotopy class that are in general position and have minimal intersections with each other and with $\ac$.
Suppose that there is an intersection $p$ between $s$ and $\eta$ such that $p\notin\Delta$.
Starting from $p$, there are two paths following $\eta$. These two paths can not intersect
before hitting $s$ (cf. green circle in Figure~\ref{fig:1}) --otherwise they form a once-decorated monogon that is not admissible. Then they can either hit $s\setminus\Delta$ or $\Delta$. If they both hit $\Delta$ first, $\eta$ is isotopy to $s$, which is a contradiction. Then one of them have to hit $s\setminus\Delta$ first, cf. Figure~\ref{fig:1}.
After that, such a path can not hit the previous footprint or $s$ (cf. blue circle in Figure~\ref{fig:1}),
which forces it to wind the corresponding decoration before hitting it.
However, this is a contradiction as it is not in a minimal position with respect to $s$.
\begin{figure}[htpb]\centering
\begin{tikzpicture}[scale=.7]
    \draw[cyan,very thick](3,0)arc(0:-180:3)(-3,0)node[above]{$\eta$}
         (3,0)arc(180:0:3)(9,0)arc(0:-180:2.5);
    \draw[cyan]
         (4,0)arc(180:0:2)(8,0)arc(0:-180:1.5)(5,0)arc(180:0:1);
    \draw[cyan,very thick,dashed](6.5,-2.5)to(1,-2.5)
        (4,0)to(4,3);
    \draw[Green, very thick] (1.7,-2.5)circle(.2);
    \draw[blue, very thick] (4,2.2)circle(.2);
    \draw[red, thick](6,0)to node[above]{$s\qquad\qquad\qquad$}(0,0);
    \draw[white](0,0)\nn(6,0)\nn;
    \draw[red](0,0)\ww(6,0)\ww;\draw(3,0)\nn node[below right]{$p$};
\end{tikzpicture}
\caption{}\label{fig:1}
\end{figure}
\end{proof}

\begin{lemma}\label{lem:int=dim}
	For any $\wss^{m}\in(\dac)^{\ZZ}$ and any $\lwe\in\wXACC(\surfo)$, we have
	\begin{equation}\label{eq:int=dim}
	\qqInt(\wss^{m},\lwe)=\qdH (X_{\wss^{m}},X_{\lwe}).
	\end{equation}
\end{lemma}

\begin{proof}
	Use the induction on $l(\lwe):=\Int_{\surfo}(\ac^{\ZZ^2},\lwe)$. For the case $l(\lwe)=1$, we have $\lwe\in(\dac)^{\ZZ}$. So $\Int_{\surfoi}(\wss^{m},\lwe)=0$. Then the formula \eqref{eq:int=dim} holds by Proposition~\ref{cor:relfirst}.
	
	Now suppose that the formula holds for $l(\lwe)<r$ for some $r\geq 2$. Consider the case $l(\lwe)=r$. Since the case that $\Int_{\surfoi}(\wss^{m},\lwe)=0$ is contained in  Proposition~\ref{cor:relfirst}, we assume that $\Int_{\surfoi}(\wss^{m},\lwe)\neq0$. There are the following two cases:
	\begin{enumerate}
	\item[(1)] If each endpoint of $\wss^{m}$ is an endpoint of $\lwe$, then by Lemma~\ref{lem:1}, the condition $\Int_{\surfoi}(\wss^{m},\lwe)\neq0$ implies that there is a decoration $Z$ which is not an endpoint of $\wss^{m}$ or $\lwe$ and which lives in the same $\ac$-polygon $P$ with an arc segment of $\eta$. Then by Lemma~\ref{lem:decomp}, there are $\lws,\lwt\in\wXACC(\surfo)$ such that $\lwe=\lwt\wedge\lws$, $l(\lwe)=l(\lws)+l(\lwt)$ and $\qqInt(\wss^{m},\lwe)=\qqInt(\wss^{m},\lws)+\qqInt(\wss^{m},\lwt)$.
	\item[(2)] If $\wss^{m}$ has an endpoint $Z$ which is not an endpoint of $\lwe$. Then consider a line segment $l\subset s$ from $Z$ to the closest intersection between $s$ and $\eta$, which decomposes $\lwe$ into $\lws$ and $\lwt$ with $l(\lwe)\leq l(\lws)+l(\lwt)\leq l(\lwe)+2$. So $\lwe=\lwt\wedge\lws$ and $\qqInt(\wss^{m},\lwe)=\qqInt(\wss^{m},\lws)+\qqInt(\wss^{m},\lwt)$.
	\end{enumerate}
	In both cases, by Corollary~\ref{cor:realtri}, we have a triangle
	$$    X_{\lwe}[\nu'']\xrightarrow{\varphi(\lwe,\overline{\lws})[-\nu'']} X_{\lws}\xrightarrow{\varphi(\lws,\lwt)}X_{\lwt}[\nu]
    \xrightarrow{\varphi(\overline{\lwt},\overline{\lwe})[\nu]}X_{\lwe}[\nu+\nu']
	$$
	where 	$\nu=\ind_{\sigma(0)}^{\ZZ^2}(\lws,\lwt), \nu'=\ind_{\tau(1)}^{\ZZ^2}(\overline{\lwt},\overline{\lwe}),\nu''=\ind_{\eta(0)}^{\ZZ^2}(\lwe,\overline{\lws}).$ Applying $\Hom(X_{\wss^{m}},-)$ to this triangle, we have an exact sequence, which implies that it suffices to show that the formula \eqref{eq:int=dim} holds for $\lws$ and $\lwt$, and
	$$\Hom(X_s,\varphi(\us,\ut))=0$$
	in the orbit category $\D_{fd}(\Gamma_\ac)/\<[1],\XX\>$.
	By construction, $\varphi(\us,\ut)$ is given by the morphism induced by an angle of the $\ac$-polygon containing $Z$. But for case (1) the curve $\gamma\in\ac$ corresponding to $s$ is not an edge of this polygon and for case (2) the starting segments of $s,\us,\ut$ are in the counterclockwise order (see Figure~\ref{fig:ab}). So by \eqref{eq:3} and Corollary~\ref{cor:conterclock}, respectively, we have that in both of these two cases,  $\Hom(X_s,\varphi(\us,\ut))=0$ holds.

	Now we only need to show that the formula \eqref{eq:int=dim} holds for $\lws$ and $\lwt$. In case (1), this follows from the inductive assumption. Next consider case (2). If both $l(\lws)$ and $l(\lwt)$ are smaller than $l(\lwe)$, this also follows from the inductive assumption. Otherwise, one of $l(\lws)$ and $l(\lwt)$ is bigger than or equals $l(\lwe)$, then the other one is not bigger then 2. In this case, we must have $\{l(\lws),l(\lwt)\}=\{2,l(\lwe)\}$ and both of $\lws$ and $\lwt$ has less intersections with $s$ in $\surfoi$ than $\lwe$. Use another induction on such an intersection number, one reduces the situation to the case that $\Int_{\surfoi}(\wss^{m},\lwe)=0$. This completes the proof.
	\begin{figure}[h]\centering
	\begin{tikzpicture}[scale=.7]
	\draw[cyan, very thick] (30:3) edge[bend right] (-45:4)(-45:4.5)
	(-3,0) edge[bend left=20](-50:4)(30:3.5)
	(-3,0) edge[bend right=15](35:3)
	(2.5,-1.5)node{$\eta$}(-1,-1)node{$\tau$}(-1,.5)node{$\sigma$};
	\draw[blue,very thick](0,-3)node[below]{$\gamma\in\ac$}to(0,3);
	\draw[red,thick] (0:-3) -- (0:3)node[right]{$s$}(1.8,.3)node[black]{{}}
	(2.1,0)node[black]{{$\bullet$}};\draw (0:-3)node[white]{$\bullet$}node[red]{$\circ$};
	\end{tikzpicture}
	\caption{The composition of $\sigma$ and $\tau$}\label{fig:ab}
	\end{figure}
\end{proof}

\begin{theorem}\label{cor:int=dim}
	For any $\lws\in\wXCA(\surfo)$ and any $\lwe\in\wXACC(\surfo)$, we have
	\begin{equation}\label{eq:int=dim3}
	\qqInt(\lws,\lwe)=\qdH (X_{\lws},X_{\lwe}).
	\end{equation}
\end{theorem}

\begin{proof}
	By Lemma~\ref{lem:CA BT}, there is an $b\in\BT(\surfo)$ and $\wss^{m}\in\ac$ such that $\lws=b(\wss^{m})$. So we have
	$$\begin{array}{rcl}
	\qqInt(\lws,\lwe)&\xlongequal{\qquad}&\qqInt(\wss^{m},b^{-1}(\lwe)) \vspace{1ex}\\
	&\xlongequal{Lem.~\ref{lem:int=dim}}&\qdH(X_{\wss^{m}},X_{b^{-1}(\lwe)}) \vspace{1ex}\\
	&\xlongequal{Cor.~\ref{cor:6.4.2}}&\qdH(X_{\wss^{m}},\iota(b)^{-1}(X_{\lwe})) \vspace{1ex}\\
	&\xlongequal{\qquad}&\qdH(\iota(b)(X_{\wss^{m}}),X_{\lwe}) \vspace{1ex}\\
	&\xlongequal{Cor.~\ref{cor:6.4.2}}&\qdH(X_{b(\wss^{m})},X_{\lwe}) \vspace{1ex}\\
	&\xlongequal{\qquad}&\qdH(X_{\lws},X_{\lwe}).
	\end{array}$$
\end{proof}

%=========================================================
\section{General case}\label{sec:general}
%=========================================================
Now let us consider the general case, i.e.
show that all the results in Section~\ref{sec:int=dim} still hold when removing Assumption~\ref{ass}.
Here we summarize the results:
\begin{itemize}
  \item[(X1)] There is a bijection $X$ as \eqref{eq:X:} in Theorem~\ref{thm:X:}.
  \item[(X2)] $X$ induces a group isomorphism \eqref{eq:bt->st} as in Theorem~\ref{thm:st=bt}
  satisfying \eqref{eq:bt=st+}.
  \item[(X3)] Formula \eqref{eq:int=dim3} holds as in Theorem~\ref{cor:int=dim}.
\end{itemize}
We will refer these results as Property-X.
The strategy consists of the following two steps.
\begin{enumerate}
  \item For any graded marked surface $\surfo$, there is a graded marked surface $\surfo^+$
  obtained from $\surfo$ by adding marked points on its boundary components such that $\surfo^+$
  admits a full formal arc system $\ac^+$ satisfying Assumption~\ref{ass}.
  Thus, $\ac^+$ satisfies Property-X by Theorems~\ref{thm:X:},~\ref{thm:st=bt} and \ref{cor:int=dim}.
  \item $\surfo$ inherits Property-X from $\surfo^+$ by deleting extra marked points/decorations.
\end{enumerate}
%Note that this strategy is similar to the one in \cite{QQ} when dealing with the Calabi-Yau-3 case.
However, there is no corresponding results from cluster theory that we can imported.

\subsection{Step~$1^\circ$}

Let us show the claim by proving the following lemma.
\begin{lemma}
Let $(g,b,m=\sum_{i=1}^b m_i)$ be the numerical data of $\surf$, where $g$ is the genus of $\surf$, $b$ is the number of boundary components of $\surf$ and $m_i$ is the number of marked points in $\M$ on the $i$-th boundary component for any $1\leq i\leq b$.
If $m_i\ge2$ for any $1\leq i\leq b$ and $m_1\ge 4g$, then $\surfo$ admits a full formal arc system satisfying Assumption~\ref{ass}.
\end{lemma}
\begin{proof}
If $b=1$, take a polygon-presentation of $\surfo$ where the vertex is one of the marked points.
Then the left picture in Figure~\ref{fig:8} shows a required full formal arc system (for the genus 2 case).
For $b>1$, we can perform adding a boundary component operation, cf. the right picture in Figure~\ref{fig:8},
to produce the required full formal arc system inductively.
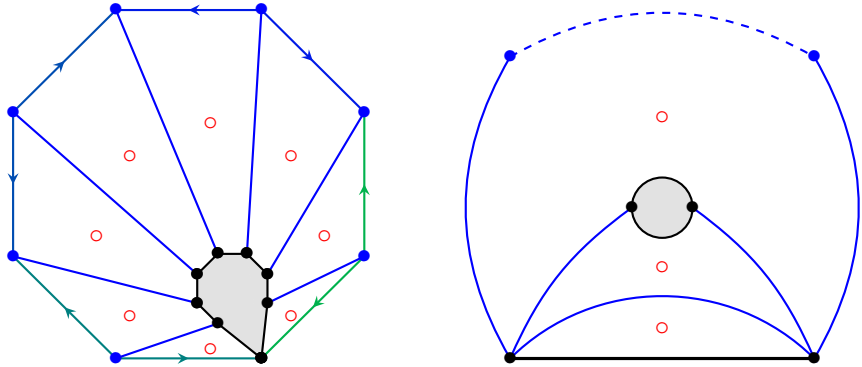
\begin{figure}[h]\centering
\begin{tikzpicture}[scale=.5,rotate=22.5]
\foreach \j in {-1,...,4}
{\draw[gray!23,fill=gray!23]($(45*\j+45:1)+(0,-3)$) \nn to($(45*\j:1)+(0,-3)$)\nn to(0,-5);}

\foreach \j in {10,30,50,70}
{\draw[thick,green!\j!blue](4.5*\j:5) edge[->-=.5,>=stealth] (4.5*\j+45:5) edge[->-=.5,>=stealth] (4.5*\j-45:5);}

\foreach \j in {-1,...,5}
{\draw[thick,blue]($(45*\j:1)+(0,-3)$) to (45*\j:5)\nn;}

\draw[thick,fill=gray!23]($(-45:1)+(0,-3)$)to(0,-5)to($(-135:1)+(0,-3)$) ;

\foreach \j in {-1,...,4}
{\draw[thick]($(45*\j+45:1)+(0,-3)$) \nn to($(45*\j:1)+(0,-3)$)\nn (0,-5)\nn;}

\foreach \j in {-2,...,5}
{\draw[red]($(45*\j+22.5:.5)+(45*\j+22.5:2.5)+(0,-1.5)$)\ww;}
\end{tikzpicture}
\qquad
\begin{tikzpicture}[xscale=.4,yscale=.4]
\draw[thick,blue](-1,5)to[bend left=-15] (-5,0) to [bend left=45] (5,0)to [bend left=-15](1,5);

\draw[thick,blue](-5,0)to [bend left=30](-5,10)\nn;
\draw[thick,blue](5,0)to [bend left=-30](5,10)\nn;
\draw[thick,blue,dashed](-5,10)to [bend left](5,10);
\draw[very thick](-5,0)\nn to(5,0)\nn;
\draw[red] (0,1) \ww (0,3) \ww (0,8) \ww;
\draw[thick,fill=gray!23](0,5) circle (1) (-1,5)\nn (1,5)\nn;
\end{tikzpicture}
\caption{A full formal arc system on polygon-presentation of a genus 2 marked surface
and adding a boundary component}\label{fig:8}
\end{figure}
\end{proof}

Therefore, we can now suppose that $\surfo^+$ is obtained from $\surfo$ by adding marked points on its boundary components
(and correspondingly decorated points in its interior).
\subsection{Step~$2^\circ$}	

\begin{definition}
A \emph{slide} of a full formal arc system $\ac$, with respect to one of its arc $\wg=AB$ and a chosen endpoint $B$ of $\wg$,
is an operation by replacing $\wg$ in $\ac$ with $\wg'$ to form another full formal arc system $\ac'$.
Here, $\wg'=AC$ is a diagonal in the $\ac$-polygon $P$ on the right of $\gamma$ when going from $A$ to $B$,
where $C$ is the adjacent vertex in $P$ of $B$ other than $A$.
See the left picture in Figure~\ref{fig:slide}.
\end{definition}

\begin{figure}[h]\centering
\begin{tikzpicture}[scale=.5,rotate=0]
\draw[blue,thick] (0,8)\nn to (3,8)node[above]{$B$}\nn to (6,8)node[above]{$C$}\nn to (9,8)\nn
    (3,8)to(3,0) (6,3)node{$P$} (3,3)node[left]{$\wg$};
\draw[blue, thick, dashed](3,0)edge node[right]{$\wg'$}(6,8);
%\draw[red](7.5,4)to(4.5,9);
\draw[red,dashed](4.5,9)to(1,4);
\draw[red]%(7,5.5)node{$\eta$}
(1.5,5.5)node{$\wss'$};
\draw[thick,thick] (0,0)node[below]{$M_1$}\nn to (3,0)node[below]{$A$}\nn to (6,0)node[below]{$M_2$}\nn to (9,0)\nn;
\draw[white](7.5,4)\nn(1,4)\nn(4.5,9)\nn;
\draw[red](7.5,4)\ww(1,4)\ww(4.5,9)\ww;
\end{tikzpicture}
\qquad
\begin{tikzpicture}[scale=.5,rotate=0]
\draw[blue,thick] (0,8)\nn to (3,8)node[above]{$B$}\nn to (6,8)node[above]{$C$}\nn to (9,8)\nn
    (7,3)node{$P$} (9,6)node[right]{$\widehat{\gamma}$};
\draw[blue,thick] plot [smooth,tension=1] coordinates {(3,0)(9,6)(6,0)};
\draw[red](7.5,4)\ww;
\draw[thick,thick] (0,0)node[below]{$M_1$}\nn to (3,0)node[below]{$A$}\nn to (6,0)node[below]{$M_2$}\nn to (9,0)\nn;
\end{tikzpicture}
\caption{The slide and removing }\label{fig:slide}
\end{figure}
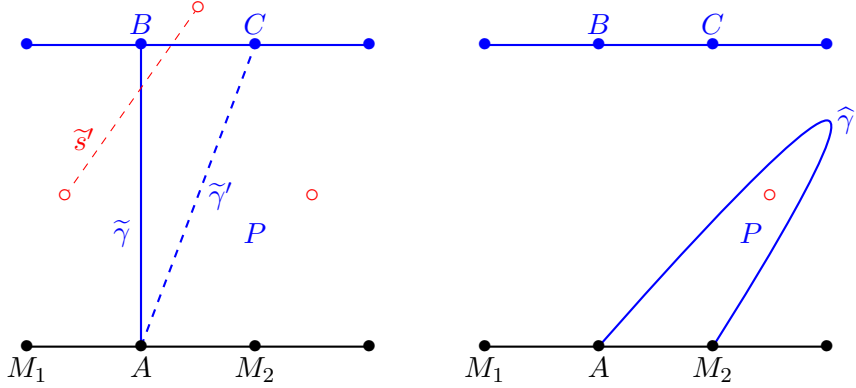

\begin{lemma}
Let $\ac$ be a full formal arc system of $\surfo^+$ satisfies Property-X. If $\ac'$ is a slide of $\ac$,
then $\ac'$ also satisfies Property-X.
\end{lemma}

\begin{proof}
The dual $\acps=(\dac\cap\acps)\cup\{\wss'\}$, where $\wss'$ is the dual of $BC$ with respect to $\dac$.
Denote by $X$ and $X'$ the maps defined in Construction~\ref{con:string} from $\wXACC(\surfo)$ to $\D_{fd}(\Gamma_{\ac})$ and $\D_{fd}(\Gamma_{\ac'})$, respectively.
Let $\SS'$ be the direct sum of the simple $\Gamma_{\ac'}$-modules. Then
$$\SS'=X'({\acps}):=\bigoplus_{\ws\in\acps}X'_{\ws^0}$$
Denote by
$$X({\acps})=\bigoplus_{\ws\in\acps}X_{\ws^0}.$$
Since $\Int_{(\surfo^+)^\circ}(\ws_1^{m_1},\ws_2^{m_2})=0$ for any $\ws_1,\ws_2\in\acps$ and $m_1,m_2\in\ZZ$, by Proposition~\ref{cor:relfirst}, both of the $\ZZ^2$-graded algebras $\Hom^{\ZZ^2}(\SS',\SS'):=\Hom^{\ZZ^2}(X'({\acps}),X'({\acps}))$ and $\Hom^{\ZZ^2}(X({\acps}),X({\acps}))$ have a basis of the form $\varphi(-,-)$,
whose compositions are give by the formulas in Proposition~\ref{prop:clock} and Corollary~\ref{cor:conterclock} and Corollary~\ref{cor:diff}. Hence we have a canonical isomorphism of $\ZZ^2$-graded algebras $$\Hom^{\ZZ^2}(\SS',\SS')\cong\Hom(X({\acps}),X({\acps})),$$
which induces a triangle equivalence
$$F: \D_{fd}(\Gamma_{\ac'})\to\D_{fd}(\Gamma_{\ac}).$$
By the constructions of $X$ and $X'$,
we have $F(X'_{\lwe})=X_{\lwe}$ for any $\lwe\in\wXACC(\surfo)$.
This implies that $\ac'$ also satisfies Property-X.
\end{proof}

Let $\ac^+$ be a full formal arc system of $\surfo^+$ satisfying Assumption~\ref{ass}. Thus, $\ac^+$ satisfies Property-X. Let $A$ be a marked point
and $M_1, M_2$ be the adjacent marked points of $A$.
Suppose that $M_i\neq A$ and the boundary arc $AM_2$ is in an $\ac^+$-polygon $P$,
which contains the open arc $\wg=AB$ as in the left picture of Figure~\ref{fig:slide}.
Now repeatedly using the operation slide to $\ac^+$ with respect to $\wg$ and the endpoint other than $A$,
we will turn $P$ into a digon, whose edges are $AM_2$ and $\widehat{\gamma}$, as shown in the right picture of Figure~\ref{fig:slide}.
Denote the resulting full formal arc system by $\ac_0$.
So the DMS $\surfo^-$ obtained from $\surfo^+$ by removing $A$ and a decoration
can be also obtained from $\surfo^+$ by cutting along $\widehat{\gamma}$ and removing $A$.
Notice that $\surfo^-$ inherits a full formal arc system from $\ac_0$, which is denoted by $\ac^-$.

\begin{lemma}
$\surfo^-$ and $\ac^-$ satisfies Property-X.
\end{lemma}
\begin{proof}
This is because the corresponding dg $\XX$-graded algebras and associated categories satisfy the following relations:
\begin{itemize}
  \item $Q_{\ac^-}$ is a full subquiver of $Q_{\ac^+}$;
  \item $\Gamma_{\ac^-}$ is a (dg $\XX$-graded) subalgebra of $\Gamma_{\ac^+}$;
  \item $\D_{fd}(\Gamma_{\ac^-})$ is subcategory of $\D_{fd}(\Gamma_{\ac^+})$,
  generated (as a thick subcategory) by the simples corresponding to the vertices of $Q_{\ac^-}$.
\end{itemize}

By repeatedly deleting points, we can recover $\surfo$ from $\surfo^+$ with Property-X.
\end{proof}

\subsection{Final statement}
Use the lemma above and induction on the marked points we add to obtain $\surfo^+$ from $\surfo$,
we have the following statement.

\begin{theorem}\label{thm:final}
Given a graded DMS $\DMS$, there is a full formal arc system $\ac$, satisfying Property-X.
\end{theorem}		

%\New{
%\begin{corollary}
%Given a graded DMS, there is a f.f.a.s $\ac$ satisfying that the map
%$$X\colon \eta^m\mapsto X(\eta^m)$$
%gives a bijection from the set of closed arcs on $\log\surfo$, whose underlying $\eta$ is in $\CA(\surfo)$,
%to $\Sph(\Gamma_\ac)[\ZZ\oplus\ZZ\XX]$.
%Moreover, for any closed arcs $\eta_1^{m'},\eta_2^{m''}$, we have
%$$
%    2\Int(\eta_1^{m'},\eta_2^{m''})=\dim\bigoplus_{n,m\in\ZZ}(X_{\eta_1^{m'}},X_{\eta_2^{m''}}[n+m\XX]).
%$$
%\end{corollary}}

%=========================================================
\section{Topological realization of Lagrangian immersion}\label{sec:Lag}
%Recall from Definition~\ref{def:f.f.a.s.} and Definition~\ref{def:dual.a.s.}
%that we have a full formal arc system $\ac$ on a graded DMS $\DMS$
%with its dual arc system $\dac$.
%=========================================================
\subsection{Derived categories of gentle algebras and Lagrangian immersion}
%=========================================================

Recall from Remark~\ref{rmk:cy} that $Q_\ac^0$ be the $\XX$-degree zero part of $Q_\ac$ with induced differential $d_\ac^0$. Then we have a dg algebra $\Gamma_\ac^0=(\k Q_\ac^0,d_\ac^0)$, which is the $\XX$-degree zero part of $\Gamma_\ac$. Denote by $\D_{fd}(\Gamma_\ac^0)$ the bounded/finite-dimensional derived category of $\Gamma_\ac^0$.
Note that, $\D_{fd}(\Gamma_\ac^0)$ is in fact (triangle equivalent to) the topological Fukaya category introduced in \cite{HKK}.

Let $Q$ be a graded quiver (i.e. its arrows are $\ZZ$-graded) and $I$ a set of paths of length at least 2, such that $\k Q/\<I\>$ is finite dimensional, where $\<I\>$ is the ideal generated by $I$. The pair $(Q,I)$ is called \emph{gentle} provided that the following hold.
\begin{enumerate}
	\item for any vertex $i\in Q_0$, there are at most two arrows in and at most two arrows out;
	\item for any arrow $\alpha\in Q_1$, there is at most one arrow $\beta\in Q_1$ with $\alpha\beta\in I$ and at most one arrow $\gamma\in Q_1$ with $\alpha\gamma\notin I$; and there is at most one arrow $\beta'\in Q_1$ with $\beta\alpha\in I$ and at most one arrow $\gamma'\in Q_1$ with $\gamma\alpha\notin I$.
\end{enumerate}
A \emph{graded gentle algebra} is (Morita equivalent to) the algebra $\k Q/I$ for a graded gentle pair $(Q,I)$.

\begin{lemma}[\cite{OPS,O}]
	$H^*(\Gamma_\ac^0)$ is a graded gentle algebra and there is a triangle equivalence
	$\D_{fd}(\Gamma_\ac^0)\simeq\D_{fd}(H^*(\Gamma_\ac^0))$.
\end{lemma}

Recall that $S_i,1\leq i\leq n$ are the simple $\Gamma_\ac$-modules corresponding to $\wg_i$. Let $S_i^0,1\leq i\leq n$ be the simple $\Gamma^0_\ac$-modules corresponding to $\wg_i$.

By \cite[Lem.~4.4]{K8}, there is a Lagrangian immersion (cf. \cite[Def.~7.2]{KQ})
\begin{equation}\label{eq:ic}
  \ic\colon \D_{fd}(\Gamma_\ac^0)\to\D_{fd}(\Gamma_\ac),
\end{equation}
sending $S_i^0$ to $S_i$ and such that there is a canonical isomorphism
\[
	\RHom_{\D_{fd}(\Gamma_\ac) } \big(\ic(X_1),\ic(X_2)\big)
	\cong\RHom_{ \D_{fd}(\Gamma_\ac^0) } (X_1,X_2)\oplus
	\RHom_{ \D_{fd}(\Gamma_\ac^0) } \big(X_2,X_1\big)^*[-\XX],
\]
for any pair of objects $X_1, X_2$ in $\D_{fd}(\Gamma_\ac^0)$. We denote by $\EE^0_\ac=\Ext^{\ZZ}(\SS^0,\SS^0)$, where $\SS^0=\oplus_{i=1}^nS_i^0$.

\begin{proposition}
	Let $\lws\in\wXACC(\surfo)$. Then $X_{\lws}$ is in $\ic( \D_{fd}(\Gamma_\ac^0) )$ if and only if $\lws$ is in the zero-sheet $\surfo^0$ of $\log\surfo$.
\end{proposition}
\begin{proof}
The curve $\lws$ is in the zero-sheet if and only if in the string $\omega(\lws)$ in \eqref{eq:string}, each $\chi_i=-\ind_{V_j}^{\ZZ^2}(\wg_{k_j}^0,\lws)=\ii_i$ for some $\ii_i\in\ZZ$. This is equivalent to that $X_{\lws}$ is in the triangulated subcategory of $\D_{fd}(\Gamma_\ac)$ generated by the simples $S_i,1\leq i\leq n$, which is $\ic(\D_{fd}(\Gamma^0_\ac))$.
\end{proof}

\begin{definition}
	We call a curve in $\wXACC(\surfo)$ a \emph{zero-level} curve if it is in the zero-sheet $\surfo^0$ of $\log\surfo$. Denote by $\ZL(\log\surfo)$ the subset of $\wXACC(\surfo)$ consisting of zero-level curves.
\end{definition}
%Note that $\ZL(\log\surfo)$ is a subset of $\wXACC(\surfo)$.

%=========================================================
\subsection{Dual arc systems on graded marked surfaces}
%=========================================================
Recall that there is a set $\M$ of open marked points and a set $\Y$ of closed marked points on the surface $\surf$.

\begin{definition}\label{def:type}
	We have the following types of curves in $\surf$ (compared with those in $\surfo$ in Definition~\ref{not:CA}).
	\begin{itemize}
		\item A curve $c$ is called \emph{open} (resp. \emph{closed}) if $c(0)$ and $c(1)$ are open (resp. closed) marked points, i.e., in $\M$ (resp. $\Y$).
		\item An \emph{open arc} (resp. \emph{closed arc}) is an open (resp. closed) curve without self-intersections in $\surf\setminus\partial\surf$. We call two open/closed arcs do not cross each other if they do not have intersections in $\surf\setminus\partial\surf$.
%		\item A \emph{closed arc} is a closed curve without self-intersections in $\surfoi$ and whose two endpoints are not the same decoration.
%		\item A closed curve is called \emph{admissible} if it does not cut out a once-decorated monogon by one of its self-intersections in $\surfoi$. See the upper right picture in Figure~\ref{fig:ext.} for the failure of the condition.
	\end{itemize}
	We always consider open/closed curves up to taking inverse and homotopy relative to endpoints.
\end{definition}
%An \emph{open} (resp. \emph{closed}) \emph{arcs} in $\surf$ is a curve whose endpoints are in $\M$ (resp. $\Y$) and without self-intersections in $\surf\setminus\partial\surf$.

Denote by $\wCC(\surf)$ (resp. $\wOC(\surf)$) the set of graded closed (resp. open) curves in $\surf$, and by $\wCA(\surf)$ (resp. $\wOA(\surf)$) the set of graded closed (resp. open) arcs in $\surf$.

%We have the notion of open/closed full formal arc systems of $(\surf,\lambda)$ as follows.

\begin{definition}
	An open (resp. closed) full formal arc system $\ac$ (resp. $\ac^*$) on $\surf^\lambda$ is a collection of pairwise non-crossing graded open (resp. closed) arcs, such that it divides $\surf$ into $\ac$-polygons (resp. $\ac^*$-polygons), each of which contains exactly one open (resp. closed) boundary segment. Here, an open (resp. closed) boundary segment is a connected component of of $\partial\surf\setminus\M$ (resp. $\partial\surf\setminus\Y$).
\end{definition}

There is a natural operation (with respect to the cut $\cut$)
\begin{equation}
\ex:\wCC(\surf)\to\ZL(\log\surfo)
\end{equation}
sending a graded closed curve $\ws$ to a zero-level curve $\ex(\ws)$
by pulling its both endpoints on $\Y$ to $\Delta$ via the corresponding curves in the cut $\cut$
(see Figure~\ref{fig:2ca} where the greed one is sent to the red one) with that $\ex(\ws)$ inherits the grading from $\ws$. This (the grading) is well-defined because $\ex(\ws)$ does not cross the cut $\cut$ and the grading $\Lambda$ of $\surfo$ is compatible with the grading $\lambda$ of $\surf$ and the cut $\cut$.
It is easy to see that $\ex$ is an bijection. It follows that $\ex$ preserves intersection indices.

\begin{figure}[h]\centering
    \begin{tikzpicture}[xscale=.75,yscale=.6]\clip(-7,-4)rectangle(9,2.5);
    \draw[Emerald!50,very thick]
        (-3,-3).. controls +(90:1) and +(180:.5) .. (-2,-.5)
        to(2+2,-.5).. controls +(0:.5) and +(90:1) .. (3+2,-3);
    \draw[blue,very thick](0,1)to(0,-1)(0+2,1)to(0+2,-1) (-4,3)to(-2,3) (2+2,3)to(4+2,3);
    \draw[dashed,blue](2,1)to[bend left=30](2+2,3)(0,1)to[bend left=-30](-2,3);
    \draw[dashed,blue](2,-1)to[bend left=-30](2+2,-3)(0,-1)to[bend left=30](-2,-3);
    \draw[dashed,blue,thin](4+2,3)to[bend left=90](4+2,-3);
    \draw[dashed,blue,thin](-4,3)to[bend left=-90](-4,-3);
    \draw[very thick](2+2,-3)\nn to(4+2,-3)\nn(-2,-3)\nn to(-4,-3)\nn;
    \draw[gray!99,very thick, dashed](3+2,0)to(3+2,-3)(-3,0)to(-3,-3);
    \draw[red,very thick](-3,0)to(3+2,0)node[white]{$\bullet$} \ww node[above]{$Z_2$};
    \draw[red](3+2,-3)node[white]{$\bullet$} \ww node[Emerald,below]{$Y_2$};
    \draw[red](-3,0)node[white]{$\bullet$} \ww node[above]{$Z_1$};
    \draw[red](-3,-3)node[white]{$\bullet$} \ww node[Emerald,below]{$Y_1$};
    \end{tikzpicture}
\caption{Closed arcs on $\surfo$ (in red) and in $\surf$ (in green)}\label{fig:2ca}
\end{figure}

The data $(\ac,\dac,\cut)$ of $\DMS$ one-one corresponds to the data
$(\ac,\ac^*)$ of $\surf^\lambda$ as follows:
\begin{itemize}
%  \item the underlying graded marked surface $(\surf,\lambda)$;
  \item the induced full formal arc system $\ac$ on $\surf$ (when forgetting the decoration set $\Delta$ and
  \item the dual arc system $\ac^*=\ex^{-1}(\dac)$ on $\surf$.
\end{itemize}
See Figure~\ref{fig:dac}, where
\begin{itemize}
	\item $\ac$ consists of blue arcs;
	\item $\dac$ consists of red ones,
	\item $\cut$ consists of dashed gray one and
	\item $\ac^*$ consists of green ones.
\end{itemize}
\begin{figure}[h]\centering
	\begin{tikzpicture}[scale=-.5]
	\foreach \j in {2,0,3}{
		\draw[red,very thick](0,0)to(90*\j:5);}
	\foreach \j in {1,2,0,3}{
		\draw[blue,very thick](90*\j+20:4)to(90*\j-20:4);
		\draw[dashed,blue,thin](90*\j+20:4)to[bend left=-15](90*\j-20+90:4);}
	\draw[very thick](90+20:4)\nn node[below right]{$M$} to(90-20:4)\nn node[below left]{$M'$};
	\draw[gray!99,very thick, dashed](0,0)to($(90+20:4)!.5!(90-20:4)$);
	\draw[red](0,0)node[white]{$\bullet$} \ww;
	\draw[red](0,3.75)node[white]{$\bullet$} \ww node[below]{$Y$};
	\end{tikzpicture}\qquad
	\begin{tikzpicture}[scale=-.5]
	\draw[Emerald,very thick](-90:5) to (0,3.75)
	(0:5) to [bend left=45] (0,3.75) (180:5) to[bend left=-45] (0,3.75);
	\foreach \j in {1,2,0,3}{
		\draw[blue,very thick](90*\j+20:4)to(90*\j-20:4);
		\draw[dashed,blue,thin](90*\j+20:4)to[bend left=-15](90*\j-20+90:4);}
	\draw[very thick](90+20:4)\nn node[below right]{$M$} to(90-20:4)\nn node[below left]{$M'$};
	\draw[gray!99,very thick, dashed](0,0)to($(90+20:4)!.5!(90-20:4)$);
	\draw[red](0,3.75)node[white]{$\bullet$} \ww node[below]{$Y$};
	\end{tikzpicture}
	\caption{The dual arc systems $\dac$ (left picture with compatible cut $\cut$) and $\ac^*$ (right picture)}\label{fig:dac}
\end{figure}

Conversely, $(\ac,\ac^*)$ also determines $(\ac,\dac,\cut)$ in the sense that
there is exactly one decorated in each $\ac$-polygon;
thus the cut (paring of $\Y$ and $\Delta$) is determined uniquely;
and then $\dac=\ex(\ac^*)$.
Moreover, $\ac^*$ is in fact a closed full formal arc system of $(\surf,\lambda)$.

%=========================================================
\subsection{Topological realization}
%=========================================================
There is an analogous string model for $(\surf,\lambda)$ as follows:
	\begin{itemize}
		\item each positive arc segment in an $\ac$-polygon of $\surf$ corresponds to an arrow in $Q_\ac^0$, which induces a $\ZZ$-graded morphism between the simples $S^0_i,1\leq i\leq n$;
		\item each $\ws\in\wCC(\surf)$ gives a string $\omega(\ws)$, which is a sequence of morphisms of degree 1 between  shifts of simples induced by the arc segments of $\ws$ divided by $\ac$;
		\item define $X^{0}_{\ws}$ to be $X^{0}_{\omega(\ws)}$, which is the dg module of $\Gamma_\ac^0$ associated to the string $\omega(\ws)$.
	\end{itemize}
Thus, we get a map
\begin{gather}\label{eq:X0}
    X^{0}\colon\wCC(\surf)\to\D_{fd}(\Gamma_\ac^0).
\end{gather}

The operation $\ex$ can be thought as
a topological realization of the Lagrangian immersion $\ic$, in the sense that
we have the following.
\begin{theorem}\label{thm:L}
	For any $\ws\in\wCC(\surf)$, we have
\begin{gather}\label{eq:ic}
    \ic(X^{0}_{\ws})=X_{\ex(\ws)} .
\end{gather}
\end{theorem}
\begin{proof}
	Since $\ex(\ws)$ does not cross the cut $\cut$, the string model for it does not contain any morphism whose associated arrow is not $\XX$-free. So this model is the same as the string model for $\ws$. Hence we are done.
\end{proof}

\begin{remark}
Note that the right hand side of \eqref{eq:ic} involves $\cut$
while the left hand side does not.
This is because although the existence of $\ic$ in \eqref{eq:ic} is in general setting,
the construction of grading of $\Gamma_\ac$ depends on the choice of the cut $\cut$.
Thus in our setting $\ic=\ic_{\cut}$ involves $\cut$ implicitly.
\end{remark}

%Let $\wCA(\surf)$ be the set of graded closed arc in $\surf$ and recall that $\wCC(\surf)$
%is the set of graded closed curve on $\surf$.
An immediate corollary is that the following $\ZZ$-graded $\qv$-intersection formula
from Corollary~\ref{cor:int=dim}, that non-zero $\XX$-degree terms all vanish.

\begin{corollary}\label{cor:int=dim2}
For any $\ws\in\wCA(\surf)$ and any $\we\in\wCC(\surf)$, we have
\begin{equation}\label{eq:int=dim2}
\qqInt(\ws,\we)=\qdH(X^0_{\ws},X^0_{\we}).
\end{equation}
\end{corollary}

We can generalize the formula above to the case when $\ws,\we$ are both graded closed curves
(recall that the difference is that an arc in $\wCA(\surf)$ has no interior self-intersections).
But we leave this small refinement to \cite{QZZ},
where we will prove such a formula for graded skew gentle algebras.

%=========================================================
\subsection{Further study: Koszul duality via graph duality}
%=========================================================
One can construct a string model for $\wOC(\surf)$,
by taking the closed full formal arc system $\ac^*$ as reference.
Namely, there is a map
\begin{gather}\label{eq:rho}
    \rho\colon \wOC(\surf)\to\per(\Gamma_{\ac}^0),
\end{gather}
sending open arcs $\wg_i$ in $\ac$ to the indecomposable projective modules of $H^*(\Gamma_\ac^0)$.
People expect that \eqref{eq:rho} and \eqref{eq:X0} are topological realizations of Koszul duality,
in the sense that
\begin{gather}\label{eq:id2}
    \qqInt(\wg,\ws)=\qdH(\rho(\wg),X^0(\ws)),
\end{gather}
for any $\wg\in\wOA(\surf), \ws\in\wCC(\surf)$.
In particular, given a silting set $\{\rho(\wg_j)\}$,
the corresponding curves $\{\wg_j\}$ will form a full formal arc system $\ac_0$
with dual full formal arc system $\ac_0^*=\{\ws_j\}$, satisfying that
$\{ X^0(\ws_j) \}$ is the set of simples of the heart dual to $\{\rho(\wg_j)\}$.
The set $\{ X^0(\ws_j) \}$ is also known as a simple minded collection.
Such type of results for Calabi-Yau-3 version of $\per(\Gamma_\ac)$ and $\D_{fd}(\Gamma_\ac)$ can be found in
Table 1 in either \cite{QQ} or \cite{QQ2} with corresponding formula \eqref{eq:id2} in \cite{QZ2}.

We plan to further investigate this for derived categories of any graded gentle algebras
(without restriction of finite dimension or finite global dimension) in \cite{QZ4}.

%=========================================================

%=========================================================
\end{document}